\documentclass[EJP]{ejpecp} 

\usepackage{comment}
\usepackage[usenames]{color}

\newcommand{\be}{\begin{equation}}
\newcommand{\ee}{\end{equation}}
\newcommand{\bea}{\begin{eqnarray}}
\newcommand{\eea}{\end{eqnarray}}
\newcommand{\beas}{\begin{eqnarray*}}
\newcommand{\eeas}{\end{eqnarray*}}

\newcommand{\bbR}{\mathbb{R}}

\renewcommand{\v}{{\mathbf{v}}}

\newcommand{\Var}{{\rm Var}}

\newcommand{\bbP}{{\mathbb{P}}}
\newcommand{\bbE}{{\mathbb{E}}}

\newcommand{\rank}{{\rm rank}}

\newcommand{\tr}{{\rm tr}}

\newcommand{\diag}{{\rm diag}}

\usepackage{prettyref}

\newrefformat{eq}{(\ref{#1})}
\newrefformat{chap}{Chapter~\ref{#1}}
\newrefformat{sec}{Section~\ref{#1}}
\newrefformat{alg}{Algorithm~\ref{#1}}
\newrefformat{fig}{Fig.~\ref{#1}}
\newrefformat{tab}{Table~\ref{#1}}
\newrefformat{rmk}{Remark~\ref{#1}}
\newrefformat{clm}{Claim~\ref{#1}}
\newrefformat{def}{Definition~\ref{#1}}
\newrefformat{cor}{Corollary~\ref{#1}}
\newrefformat{lmm}{Lemma~\ref{#1}}
\newrefformat{prop}{Proposition~\ref{#1}}
\newrefformat{app}{Appendix~\ref{#1}}
\newrefformat{hyp}{Hypothesis~\ref{#1}}
\newrefformat{th}{Theorem~\ref{#1}}
\newrefformat{ass}{Assumption~\ref{#1}}

\SHORTTITLE{Heteroskedastic Wishart-type Concentration}

\TITLE{On the Non-Asymptotic Concentration of Heteroskedastic Wishart-type Matrix\thanks{The research of Tony Cai was supported in part by NSF grants DMS-1712735 and DMS-2015259 and NIH grants R01-GM129781 and R01-GM123056. The research of Rungang Han and Anru R. Zhang was supported in part by NSF CAREER-1944904, NSF DMS-1811868, and NIH R01-GM131399.}}

\AUTHORS{%
  T.~Tony~Cai\footnote{University of Pennsylvania, United States of America.
    \EMAIL{tcai@wharton.upenn.edu}}
  \and 
  Rungang~Han\footnote{Duke University, United States of America. \EMAIL{rungang.han@duke.edu}}
  \and
  Anru~R.~Zhang\footnote{University of Wisconsin-Madison and Duke University, United States of America. \EMAIL{anru.zhang@duke.edu}}}


\AMSSUBJ{ 60B20} 
\AMSSUBJSECONDARY{46B09} 

\SUBMITTED{January 3, 2021} 
\ACCEPTED{Unknown} 

\VOLUME{0}
\YEAR{2021}
\PAPERNUM{0}
\DOI{10.1214/YY-TN}

\ABSTRACT{This paper focuses on the non-asymptotic concentration of the heteroskedastic Wishart-type matrices. Suppose $Z$ is a $p_1$-by-$p_2$ random matrix and $Z_{ij} \sim N(0,\sigma_{ij}^2)$ independently, we prove the expected spectral norm of Wishart matrix
deviations (i.e., $\bbE \left\|ZZ^\top - \bbE ZZ^\top\right\|$) is upper bounded by
	\begin{equation*}
		\begin{split}
		(1+\epsilon)\left\{2\sigma_C\sigma_R + \sigma_C^2 + C\sigma_R\sigma_*\sqrt{\log(p_1 \wedge p_2)} + C\sigma_*^2\log(p_1 \wedge p_2)\right\},
		\end{split}
	\end{equation*}
	where $\sigma_C^2 := \max_j \sum_{i=1}^{p_1}\sigma_{ij}^2$, $\sigma_R^2 := \max_i \sum_{j=1}^{p_2}\sigma_{ij}^2$ and $\sigma_*^2 := \max_{i,j}\sigma_{ij}^2$. A minimax lower bound is developed that matches this upper bound. Then, we derive the concentration inequalities, moments, and tail bounds for the heteroskedastic Wishart-type matrix under more general distributions, such as sub-Gaussian and heavy-tailed distributions. Next, we consider the cases where $Z$ has homoskedastic columns or rows (i.e., $\sigma_{ij} \approx \sigma_i$ or $\sigma_{ij} \approx \sigma_j$) and derive the rate-optimal Wishart-type concentration bounds. Finally, we apply the developed tools to identify the sharp  signal-to-noise ratio threshold for consistent clustering in the heteroskedastic clustering problem.}

\begin{document}

\section{Introduction}\label{sec:intro}

Random matrix theory is an important topic in its own right and has been proven to be a powerful tool in a wide range of applications in statistics, high-energy physics, and number theory. Wigner matrices, symmetric matrices with mean-zero independent and identically distributed (i.i.d.) entries (subject to the symmetry constraint), have been a particular focus. Asymptotic and non-asymptotic properties of the spectrum of Wigner matrices have been widely studied in the literature. See, for example, \cite{anderson2010introduction,tao2012topics,vershynin2010introduction} and the references therein. 

Motivated by a range of applications, heteroskedastic Wigner-type matrices,  random matrices with independent heteroskedastic entries, have attracted much recent attention. A central problem of interest is the characterization of the dependence of the spectral norm $\|\cdot\|$ (i.e., the largest singular value of the matrix) of a heteroskedastic Wigner-type matrix on  the variances of its entries. To answer this question, Ajanki, Erd\H{o}s, Kr\"uger \cite{ajanki2017universality} established the asymptotic behavior of the resolvent, a local law down to the smallest spectral resolution scale, and bulk universality for the heteroskedastic Wigner-type matrix. Bandeira and van Handel \cite{bandeira16sharp} proved an non-asymptotic upper bound for the spectral norm. More specifically, let $Z=(Z_{ij})$ be a $p\times p$ heteroskedastic Wigner-type matrix with $\Var(Z_{ij}):=\sigma_{ij}^2$.    Bandeira and van Handel \cite{bandeira16sharp} showed that
$\bbE\left\|Z\right\| \lesssim \sigma + \sigma_*\sqrt{\log p },$ where $\sigma^2 = \max_i \sum_{j}\sigma_{ij}^2$ and $\sigma_{\ast}^2=\max_{ij}\sigma_{ij}^2$ are the column-sum-wise and entry-wise maximum variances, respectively. This bound was improved by van Handel \cite{vanhandel17spectral} to $\bbE\left\|Z\right\| \lesssim \sigma + \max_{i,j \in [p]}\sigma_{ij}^*\log i.$ Here, the matrix $\{\sigma^*_{ij}\}$ is obtained by permuting the rows and columns of the variance matrix $\{\sigma_{ij}\}$ such that $\max_j \sigma^*_{1j} \geq \max_j \sigma^*_{2j} \geq \cdots \geq \max_j \sigma^*_{pj}$. Later, Lata\l{}a and van Handel \cite{latala2017dimension} further improved it to a tight bound:
\begin{equation}\label{ineq:bound-wigner-spectral-3}
	\bbE\left\|Z\right\| \asymp \sigma + \max_{i,j \in [p]}\sigma_{ij}^*\sqrt{\log i}.
\end{equation}

In addition to the Wigner-type matrix, the Wishart-type matrix, $ZZ^\top - \mathbb{E}ZZ^\top$, also plays a crucial role in many high-dimensional statistical problems, including the principal component analysis (PCA) and factor analysis \cite{zhang2018heteroskedastic}, matrix  denoising \cite{salmon2014poisson}, and bipartite community detection \cite{florescu2016spectral}. Though there have been many results on the asymptotic and non-asymptotic properties of the homoskedastic Wishart-type matrix, where $Z$ has i.i.d entries (see \cite{bishop2017introduction} for an introduction and the references therein), the properties of the heteroskedastic Wishart-type matrices are much less understood. 

Specifically, suppose $Z$ is a $p_1 \times p_2$ random matrix with independent and zero-mean entries. In this paper, we are interested in the Wishart-type concentration: $\bbE\left\|ZZ^\top - \bbE ZZ^\top\right\|$. Define $\sigma_C^2, \sigma_R^2, \sigma_{\ast}^2$ as the column-sum-wise, row-sum-wise, and entry-wise maximum variances:
\begin{equation}\label{eq:sigma_C-sigma_R-sigma_ast}
\sigma_C^2 = \max_j \sum_{i=1}^{p_1}\sigma_{ij}^2, \quad \sigma_R^2 = \max_i \sum_{j=1}^{p_2}\sigma_{ij}^2, \quad \sigma_\ast^2 = \max_{ij} \sigma_{ij}^2.
\end{equation} 
By the symmetrization scheme and the asymmetric Wigner-type concentration inequality in \cite{bandeira16sharp}, it is not difficult to show that
\begin{equation}\label{ineq:bound-symmetry}
	\begin{split}
		& \bbE\left\|ZZ^\top - \bbE ZZ^\top\right\| \leq \bbE\left\|ZZ^\top - Z'(Z')^\top\right\| \leq 2 \bbE \left\|ZZ^\top\right\| = 2\bbE\left\|Z\right\|^2 \\
		\lesssim & \left(\sigma_C + \sigma_R + \sigma_*\sqrt{\log(p_1 \wedge p_2)}\right)^2.
	\end{split}
\end{equation}
Since $ZZ^\top - \mathbb{E}ZZ^\top$ can be decomposed into a sum of independent random matrices,
\begin{equation*}
	ZZ^\top - \mathbb{E}ZZ^\top = \sum_{j=1}^{p_2}\left(Z_{\cdot j}Z_{\cdot j}^\top - \mathbb{E}Z_{\cdot j}Z_{\cdot j}^\top\right),
\end{equation*}
one can apply the concentration inequality for the sum of independent random matrices \cite[Theorem 1]{tropp2016expected} to show that
\begin{equation}\label{ineq:bound-tropp}
\bbE\left\|ZZ^\top - \bbE ZZ^\top\right\| \lesssim \sigma_C\sigma_R\sqrt{\log p_2} + \sigma_C^2 (\log p_2)^2.
\end{equation}
However, as we will show later, these bounds are not tight.

In this paper, we establish non-asymptotic bounds for the Wishart-type concentration $\bbE\|ZZ^\top - \bbE ZZ^\top\|$. The main results include the following. We begin by  focusing on the Gaussian case in Section \ref{sec:concentration} and prove that if all entries of $Z$ are independently Gaussian,
\begin{equation}\label{ineq:upper-bound}
\begin{split}
\mathbb{E}\left\|ZZ^\top - \mathbb{E}ZZ^\top\right\| \leq 2.01\sigma_C\sigma_R + 1.01\sigma_C^2 + C_1\sigma_R\sigma_*\sqrt{\log(p_1 \wedge p_2)} + C_2\sigma_*^2\log(p_1 \wedge p_2),
\end{split}
\end{equation}
where $C_1, C_2$ are some universal constants that does not depend on the variance components $\sigma_C,\sigma_R,\sigma_*$ or matrix dimensions $p_1,p_2$. Moreover, we can set the coefficients in front of $\sigma_C\sigma_R$ and $\sigma_C^2$ arbitrarily close to $2$ and $1$, respectively, at the sacrifice of larger constants $C_1,C_2$ in \eqref{ineq:upper-bound} (see Theorem \ref{th:wishart-Gaussian} for details).

We further justify that the constants in $2\sigma_C\sigma_R + \sigma_C^2$ are essential under the homoskedastic setting. The proof of \eqref{ineq:upper-bound} is based on a Wishart-type moment method provided in Section \ref{sec:proof-main-thm}. In Section \ref{sec:lower-bound}, we provide a lower bound to show that the upper bound \eqref{ineq:upper-bound} is minimax rate-optimal in a general class of heteroskedastic random matrices. 
	
	We then consider the more general non-Gaussian setting including  sub-Gaussian, sub-exponential, heavy tailed, and bounded distributions in  Section \ref{sec:non-Gaussian}. In particular,  we establish the following concentration bound when the entries have independent sub-Gaussian distributions:
\begin{equation}\label{ineq:bound-ours}
	\begin{split}
		\bbE \left\|ZZ^\top - \bbE ZZ^\top \right\| & \lesssim \left(\sigma_C + \sigma_R + \sigma_*\sqrt{\log(p_1 \wedge p_2)}\right)^2 - \sigma_R^2. 
	\end{split}
\end{equation}
Upper bounds for the moments and probability tails of $\|ZZ^\top - \bbE ZZ^\top\|$ are developed in Section \ref{sec:tail-bounds}. 

In Sections \ref{sec:homoskedastic rows} and \ref{sec:homoskedastic columns}, we consider two variance structures arising in statistical applications and develop tight Wishart-type concentration bounds. If the random matrix $Z$ has independent sub-Gaussian entries and homoskedastic rows, i.e., $\sigma_{ij}= \sigma_i$, we prove that
$$\mathbb{E}\|ZZ^\top - \mathbb{E}ZZ^\top\| \asymp \sum_{i=1}^{p_1} \sigma_i^2 + \sqrt{p_2\sum_{i=1}^{p_1} \sigma_i^2}\cdot \max_{i \in [p_1]}\sigma_i.$$
If $Z$ has independent sub-Gaussian entries and homoskedastic column variances, i.e., $\sigma_{ij} \asymp \sigma_j$, we prove that
$$\mathbb{E}\|ZZ^\top - \mathbb{E}ZZ^\top\| \asymp \sqrt{p_1\sum_{j=1}^{p_2} \sigma_j^4} + p_1\max_{j \in [p_2]} \sigma_j^2.$$

To illustrate the usefulness of the newly established tools, we apply these tools in Section \ref{sec:application} to solve a statistical problem in heteroskedastic clustering. Specifically, we obtain a sharp signal-to-noise ratio threshold to guarantee consistent clustering.

\section{Main Results}\label{sec:main-results}

We first introduce the notation to be used in the rest of the paper. Let $a\wedge b$ and $a\vee b$ be the minimum and maximum of real numbers $a$ and $b$, respectively. We use $[d]$ to denote the set $\{1,\ldots,d\}$ for any positive integer $d$. For any vector $v$, let $\|v\|_{q} = (\sum_i |v_i|^q)^{1/q}$ be the vector $\ell_q$ norm; specifically, $\|v\|_\infty = \sup_i |v_i|$. For any sequences $\{a_n\}, \{b_n\}$, denote $a \lesssim b$ (or $b_n\gtrsim a_n$) if there exists a uniform constant $C>0$ such that $a \leq Cb$. If $a\lesssim b$ and $a\gtrsim b$ both hold, we say $a\asymp b$. For any $\alpha\geq 1$, the Orlicz $\psi_{\alpha}$ norm of any random variable $X$ is defined as 
$$\|X\|_{\psi_\alpha} = \inf \left\{x\geq 0: \mathbb{E}\exp\left((|X|/x)^\alpha\right) \leq 2\right\}.$$
In the literature \cite{vershynin2010introduction,vladimirova2019sub}, a random variable is often called sub-Gaussian, sub-exponential, or sub-Weibull with tail parameter $(1/\alpha)$, if $\|X\|_{\psi_2}\leq C$, $\|X\|_{\psi_1}\leq C$, and $\|X\|_{\psi_{\alpha}} \leq C$, respectively. The matrix spectral norm is defined as $\|X\| = \sup_{u, v}\frac{u^\top X v}{\|u\|_2\|v\|_2}$. The capital letters $C, C_1, \tilde{C}$ and lowercase letters $c, c_1, c_0$ represent the generic large and small constants, respectively, whose exact values may vary from place to place. 

\subsection{Concentration of heteroskedastic Wishart matrix}\label{sec:concentration}

We begin by considering the Gaussian case where the entries $Z_{ij}\sim N(0, \sigma_{ij}^2)$ independently. The following theorem provides an upper bound for the concentration and is one of the main results of the paper.

\begin{theorem}[Wishart-type Concentration for Gaussian random matrix]\label{th:wishart-Gaussian}
	Suppose $Z$ is a $p_1$-by-$p_2$ random matrix and $Z_{ij}\sim N(0, \sigma_{ij}^2)$ independently. Then for any $\epsilon_1, \epsilon_2 > 0$,
	\begin{equation}\label{ineq:th1-bound}
	\begin{split}
	& \mathbb{E}\left\|ZZ^\top - \mathbb{E}ZZ^\top\right\|\\
	\leq & (1+\epsilon_1)\left\{2\sigma_C\sigma_R + (1+\epsilon_2)\sigma_C^2 + C_1(\epsilon_1)\sigma_R\sigma_*\sqrt{\log(p_1 \wedge p_2)} + C_2(\epsilon_1,\epsilon_2)\sigma_*^2\log(p_1 \wedge p_2)\right\},
	\end{split}
	\end{equation}
	where $C_1(\epsilon_1) = 10(1+\epsilon_1)\sqrt{\lceil 1/\log(1+\epsilon_1)\rceil}$ and $C_2(\epsilon_1, \epsilon_2) = (1+\epsilon_1)\lceil 1/\log(1+\epsilon_1)\rceil\left(\frac{25}{\epsilon_2} + 24\right)$.
\end{theorem}

\begin{remark}[Lower bound for  the homoskedastic case]{\rm 
	If $Z$ has independent and homoskedastic Gaussian entries, i.e., $Z_{ij} \overset{iid}{\sim} N(0,1)$, then $\sigma_C = \sqrt{p_1}, \sigma_R = \sqrt{p_2}$, and Theorem \ref{th:wishart-Gaussian} implies
	\begin{equation}\label{ineq:constant-sharp-bound}
	\bbE\left\|ZZ^\top - \bbE ZZ^\top\right\| \leq (1+\epsilon)\left(\sigma_C^2+2\sigma_C\sigma_R\right) + C_\epsilon\sigma_R\sqrt{\log(p_1 \wedge p_2)}
	\end{equation}
	for any $\epsilon > 0$ and constant $C_\epsilon$ only depending on $\epsilon$. 
On the other hand, we have
\begin{proposition}\label{pr:lower}
	If $Z$ is a $p_1$-by-$p_2$ matrix with i.i.d. homoskedastic Gaussian entries, then
	\begin{equation}\label{ineq:constant-sharp-bound2}
	\liminf_{p_1, p_2 \rightarrow \infty} \frac{\bbE\left\|ZZ^\top - \bbE ZZ^\top\right\|}{2\sigma_C\sigma_R+\sigma_C^2} \geq 1.
	\end{equation}
\end{proposition}
	Proposition \eqref{pr:lower} and \eqref{ineq:constant-sharp-bound} together indicate that $(\sigma_C^2 + 2\sigma_C\sigma_R)$ in the upper bound of Theorem \ref{th:wishart-Gaussian} are sharp in the homoskedastic case. In Section \ref{sec:lower-bound}, we establish a minimax lower bound to show that all four terms in the upper bound \eqref{ineq:bound-ours} are essential when $Z$ is a general heteroskedastic random matrix.
}\end{remark}

\subsection{Proof of Theorem \ref{th:wishart-Gaussian}}\label{sec:proof-main-thm}

The proof of Theorem \ref{th:wishart-Gaussian} relies on a moment method and the following fact: for a $p$-by-$p$ symmetric matrix $A$ (in the context of Theorem \ref{th:wishart-Gaussian}, $A = ZZ^\top - \mathbb{E}ZZ^\top$) and a even number $q \asymp \log (p)$, we have
$$\|A\|\approx \left(\tr(A^{q})\right)^{1/q}.$$ 
We introduce two lemmas for the proof of Theorem \ref{th:wishart-Gaussian}. First, Lemma \ref{lm:Gaussian-comparison} builds a comparison between the $q$-th moment of the heteroskedastic Wishart-type matrix $ZZ^\top - \bbE ZZ^\top$ with a homoskedastic analogue $HH^\top - \bbE HH^\top$. The complete proof of Lemma \ref{lm:Gaussian-comparison} is postponed to Section \ref{sec:proof-main}.
\begin{lemma}[Gaussian Comparison]\label{lm:Gaussian-comparison}
	Suppose $Z\in \mathbb{R}^{p_1\times p_2}$ has independent Gaussian entries: $Z_{ij} \sim N(0, \sigma_{ij}^2)$. Let $m_1 = \lceil \sigma_C^2\rceil + q-1$ and $m_2 = \lceil\sigma_R^2\rceil+q-1$. Suppose $H \in \bbR^{m_1 \times m_2}$ has i.i.d. $N(0,1)$ entries. Then for any $q\geq 2$,
	\begin{equation}\label{ineq:ZZHH-Gaussian}
	\mathbb{E} \tr\left\{(ZZ^\top - \mathbb{E}ZZ^\top)^{q}\right\} \leq \left(\frac{p_1}{m_1}\wedge \frac{p_2}{m_2}\right)\mathbb{E} \tr\left\{(HH^\top - \mathbb{E}HH^\top)^{q}\right\}.
	\end{equation}
\end{lemma}

\begin{remark}[Proof sketch of Lemma \ref{lm:Gaussian-comparison}]{\rm 
Previously, \cite[Proposition 2.1]{bandeira16sharp} compared the moments of the Wigner-type matrices (i.e., $Z$ is symmetric and thus $p_1 = p_2 = p, \sigma_C = \sigma_R = \sigma$) by the expansion $\bbE \tr(Z^{2p}) = \sum_{u_1,\ldots,u_{2q}}\bbE(Z_{u_1u_2}Z_{u_2u_3}\cdots Z_{u_{2p}u_1})$ and counting the cycles in a reduced unipartite graph:
\begin{equation}\label{ineq:wigner-comparison}
\bbE \tr(Z^{2q}) \leq \frac{p}{\lceil \sigma^2 \rceil + q}\bbE \tr(H^{2q}).
\end{equation}
Compared to the expansion of Wigner-type random matrix $\mathbb{E}\tr(Z^{2q})$, the expansion of Wishart-type random matrix $\mathbb{E} \tr\left\{(ZZ^\top - \mathbb{E}ZZ^\top)^{q}\right\}$ is much more complicated:
\begin{equation}\label{eq:main-sketch-1}
\begin{split}
& \mathbb{E} \tr\left\{(ZZ^\top - \mathbb{E}ZZ^\top)^{q}\right\} = \sum_{\substack{u_{q+1} = u_1, \ldots, u_{q} \in [p_1]\\v_1,\ldots, v_q \in [p_2]}} \mathbb{E}\prod_{k=1}^q \left(Z_{u_k, v_k} Z_{u_{k+1}, v_k} - \sigma_{u_k, v_k}^2\cdot 1_{\{u_k = u_{k+1}\}}\right)\\
& = \cdots = \sum_{\mathbf{c}\in ([p_1]\times [p_2])^q} \prod_{k=1}^q \sigma_{u_k, v_k} \sigma_{u_{k+1}, v_k}\prod_{(i, j)\in [p_1]\times [p_2]} \mathbb{E} G_{ij}^{\alpha_{ij}(\mathbf{c})}\left(G_{ij}^2 - 1\right)^{\beta_{ij}(\mathbf{c})},
\end{split}
\end{equation}
where $([p_1] \times [p_2])^{q}$ is the set of all cycles of length $2q$ on a $p_1$-by-$p_2$ complete bipartite graph, $G_{ij} = Z_{ij}/\sigma_{ij}$ are i.i.d. standard normal distributed, and $\alpha_{ij}(\mathbf c), \beta_{ij}(\mathbf c)$ are some graphical characteristic quantities of cycle $\mathbf c$ to be defined later. By gathering the cycles with the same ``shape" $\mathbf s$, we can show:
\begin{equation}\label{ineq:main-sketch-2}
\begin{split}
\mathbb{E}\tr\left\{(ZZ^\top - \mathbb{E}ZZ^\top)^{q}\right\}\leq  & \sum_{\mathbf{s}} \prod_{\substack{\alpha, \beta \geq 0}} \left\{\mathbb{E} G^\alpha (G^2 - 1)^\beta\right\}^{m_{\alpha, \beta}(\mathbf{s})}\\
&  \cdot \left\{p_1 \sigma_C^{2(m_L(\mathbf{s})-1)}\sigma_R^{2m_R(\mathbf{s})}\right\} \wedge \left\{p_2 \sigma_C^{2m_L(\mathbf{s})}\sigma_R^{2(m_R(\mathbf{s})-1)}\right\},
\end{split}
\end{equation}  
where $m_{\alpha,\beta}(\mathbf s), m_L(\mathbf s)$ and $m_R(\mathbf s)$ are some graphical properties of the cycles with shape $\mathbf s$ to be defined later and $G \sim N(0,1)$.  Meanwhile, we can develop a lower bound for the moment of standard Wishart matrix:
\begin{equation}\label{ineq:main-sketch-3}
\begin{split}
\mathbb{E}\tr\left((HH^\top - \bbE HH^\top)^q\right)
\geq & \sum_{\mathbf{s}}\prod_{\alpha, \beta \geq 0} \mathbb{E} \left\{G^\alpha (G^2 - 1)^\beta\right\}^{m_{\alpha, \beta}(\mathbf{s})}\\
& \cdot  \left\{m_1 \sigma_C^{2m_L(\mathbf{s})-2}\cdot \sigma_R^{2m_R(\mathbf{s})}\right\} \vee \left\{m_2 \sigma_C^{2m_L(\mathbf{s})}\cdot \sigma_R^{2m_R(\mathbf{s})-2}\right\}.
\end{split}
\end{equation}
Lemma \ref{lm:Gaussian-comparison} follows by combining \eqref{ineq:main-sketch-2} and \eqref{ineq:main-sketch-3}. 
}\end{remark}

Next, Lemma \ref{lm:iid-Gaussian-moment} gives an upper bound on the moment of the standard Wishart matrix. The complete proof is provided in Section \ref{sec:proof-main}.
\begin{lemma}\label{lm:iid-Gaussian-moment}
	Suppose $H\in \mathbb{R}^{m_1\times m_2}$ has i.i.d. standard Gaussian entries. Then for any integer $q \geq 2$,
	$$\left(\mathbb{E}\|HH^\top - \mathbb{E}HH^\top\|^{q}\right)^{1/q} \leq 2\sqrt{m_1m_2} + m_1 + 4(\sqrt{m_1}+\sqrt{m_2})\sqrt{q} + 2q,$$
	$$\left(\mathbb{E} \tr\left\{(HH^\top - \mathbb{E}HH^\top)^q\right\}\right)^{1/q} \leq 2^{1/q}(m_1\wedge m_2)^{1/q}\cdot \left(2\sqrt{m_1m_2} + m_1 + 4(\sqrt{m_1}+\sqrt{m_2})\sqrt{q} + 2q\right).$$
\end{lemma}

\begin{remark}[Proof idea of Lemma \ref{lm:iid-Gaussian-moment}]\label{rm:lemma2}{\rm
Let $\sigma_i(H)$ be the $i$-th singular value of $H$. The proof of Lemma \ref{lm:iid-Gaussian-moment} utilizes the following fact:
\begin{equation*}
\|HH^\top - \mathbb{E}HH^\top\| = \|HH^\top - m_2I_{m_1}\| = \max\left\{\sigma_1^2(H)-m_2, m_2 - \sigma_{m_1}^2(H)\right\}
\end{equation*}
and the concentration inequalities of the largest and smallest singular values of the Gaussian ensemble (e.g., \cite{vershynin2010introduction}). See Section \ref{sec:proof-main} for the complete proof. 
}
\end{remark}
Now, we are in position to finish the proof of Theorem \ref{th:wishart-Gaussian}.
\begin{proof}[Proof of Theorem \ref{th:wishart-Gaussian}]
	Without loss of generality, we assume $\sigma_\ast^2 = \max_{ij}\sigma_{ij}^2 = 1$.  Let $m_1 = \lceil\sigma_C^2\rceil + 2q-1, m_2 = \lceil\sigma_R^2\rceil + 2q-1$ for some $q$ to be specified later and $H$ be an $m_1$-by-$m_2$ random matrix with i.i.d. standard Gaussian entries. Lemmas \ref{lm:Gaussian-comparison} and \ref{lm:iid-Gaussian-moment} imply
\begin{equation}\label{ineq:Gaussian-ZZ-bound-1}
\begin{split}
& \mathbb{E}\|ZZ^\top - \mathbb{E}ZZ^\top\| \leq  \left(\mathbb{E}\tr\left\{\left(ZZ^\top - \mathbb{E}ZZ^\top\right)^{2q}\right\}\right)^{1/2q} \\
\overset{\text{Lemma \ref{lm:Gaussian-comparison}}}{\leq} & \left\{\left(\frac{p_1}{m_1} \wedge \frac{p_2}{m_2}\right) \cdot  \mathbb{E}\tr\left\{HH^\top-\mathbb{E}HH^\top\right\}^{2q}\right\}^{1/{2q}}\\
\overset{\text{Lemma \ref{lm:iid-Gaussian-moment}}}{\leq} & 2^{1/2q}\left(\left(\frac{p_1}{m_1}\wedge \frac{p_2}{m_2}\right)m_1\wedge m_2\right)^{1/2q} \left(2\sqrt{m_1m_2} + m_1 + 4(\sqrt{m_1}+\sqrt{m_2})\sqrt{2q} + 4q\right)\\
\leq & 2^{1/2q} \left(p_1 \wedge p_2\right)^{1/2q}\left(2\sigma_C\sigma_R + \sigma_C^2 + 10\sigma_C\sqrt{q} + 10\sigma_R\sqrt{q} + 24q\right).
\end{split}
\end{equation}
Let $q = K\lceil \log(p_1\wedge p_2)\rceil$ for $K = \lceil \frac{1}{\log(1+\varepsilon_1)}\rceil$, then we have
\begin{equation}\label{ineq:Gaussian-ZZ-bound-2}
\begin{split}
& \mathbb{E}\|ZZ^\top - \mathbb{E}ZZ^\top\|\\
\leq & 2^{1/2K} \left(e^{q/K}\right)^{1/2q}\left(2\sigma_C\sigma_R + \sigma_C^2 + 10\sigma_C\sqrt{q} + 10\sigma_R\sqrt{q} + 24q\right)\\
\leq & (2e)^{1/2K}\left(2\sigma_C\sigma_R + (1+\epsilon_2)\sigma_C^2 + 10\sqrt{K}\sigma_R\sqrt{\log(p_1\wedge p_2)} + \left(\frac{25}{\epsilon_2} + 24\right)K\log(p_1\wedge p_2)\right) \\
\leq & 2(1+\epsilon_1)\sigma_C\sigma_R + (1+\epsilon_1)(1+\epsilon_2)\sigma_C^2 + C_1(\epsilon_1)\sigma_R\sqrt{\log(p_1 \wedge p_2)} + C_2(\epsilon_1,\epsilon_2)\log(p_1 \wedge p_2).
\end{split}
\end{equation}
Here, 
\begin{equation*}
	\begin{split}
		C_1(\epsilon_1) & = 10(1+\varepsilon_1)\sqrt{\lceil 1/\log(1+\varepsilon_1)\rceil}, \\
		C_2(\epsilon_1, \epsilon_2) & = (1+\varepsilon_1)\lceil 1/\log(1+\varepsilon_1)\rceil\left(\frac{25}{\epsilon_2} + 24\right).
	\end{split}
\end{equation*}
\end{proof}

\subsection{Lower bounds}\label{sec:lower-bound}

To show the tightness of the upper bound given earlier, we also develop the following minimax lower bound for the heteroskedastic Wishart-type concentration. 
\begin{theorem}[Lower bound of heteroskedastic Wishart-type concentration]\label{th:hetero-wishart-lower}
	Suppose $p_1, p_2\geq 4$. Consider the following set of $p_1$-by-$p_2$ random matrices,
	\begin{equation*}
	\mathcal{F}_p(\sigma_\ast, \sigma_C, \sigma_R) = \left\{Z\in \mathbb{R}^{p_1\times p_2}: \begin{array}{ll}
	& Z_{ij}\overset{ind}{\sim} N(0, \sigma_{ij}^2), p =p_1\wedge p_2, \max_{i,j}\sigma_{ij}\leq \sigma_\ast, \\
	&\max_i \sum_{j=1}^{p_2} \sigma_{ij}^2 \leq \sigma_R^2, \max_j\sum_{i=1}^{p_1}\sigma_{ij}^2 \leq \sigma_C^2
	\end{array}\right\}.
	\end{equation*}
	For any $(\sigma_\ast, \sigma_R, \sigma_C)$ tuple satisfying $\min\{\sigma_C, \sigma_R\} \geq \sigma_{\ast}\geq \max\{\sigma_C/\sqrt{p_1}, \sigma_R/\sqrt{p_2}\}$, there exists a random Gaussian matrix $Z \in \mathcal{F}_{p}(\sigma_{\ast}, \sigma_R, \sigma_C)$ such that 
	\begin{equation}
	\mathbb{E}\|ZZ^\top - \mathbb{E}ZZ^\top\| \gtrsim \sigma_C^2 + \sigma_C\sigma_R + \sigma_R\sigma_\ast\sqrt{\log p} + \sigma_\ast^2\log p.
	\end{equation}
\end{theorem}
The proof of Theorem \ref{th:hetero-wishart-lower} is given in Section \ref{sec:proof-main}. 
\begin{remark}{\rm 
	Theorems \ref{th:wishart-Gaussian} and \ref{th:hetero-wishart-lower} together establish the minimax optimal rate of $\mathbb{E}\|ZZ^\top - \mathbb{E}ZZ^\top\|$ in the class of $\mathcal{F}_p(\sigma_\ast, \sigma_C, \sigma_R)$. In other words, Theorem \ref{th:hetero-wishart-lower} shows that \eqref{ineq:th1-bound} yields the best upper bound for heteroskedastic Wishart-type concentration among all the bounds characterized by $\sigma_C, \sigma_R, \sigma_*$. We shall point out that the upper bound of Theorem \ref{th:wishart-Gaussian} may not be tight for some specific values of $\{\sigma_{ij}^2\}$. For example, in Sections \ref{sec:homoskedastic rows} and \ref{sec:homoskedastic columns}, we develop sharper bounds via a more refined analysis when the Wishart matrix has near-homoskedastic rows or columns. 

	Generally speaking, it remains an open problem to develop a heteroskedastic Wishart-type concentration inequality that is tight for all specific values of $\{\sigma_{ij}^2\}$. We leave this problem as future work.
}\end{remark}

\section{Extensions}\label{sec:extensions}

We consider several extensions of Theorem \ref{th:wishart-Gaussian} in this section. 

\subsection{Wishart-type concentration of non-Gaussian random matrices}\label{sec:non-Gaussian}

In this section, we generalize the developed concentration inequality for heteroskedastic Wishart matrices with more general entrywise distributions, such as sub-Gaussian, sub-exponential, heavy tailed, and bounded distributions. We first introduce the following lemma as a sub-Gaussian analog of Lemma \ref{lm:Gaussian-comparison}.
\begin{lemma}[Sub-Gaussian comparison]\label{lm:subGaussian-comparison}
	Suppose $Z\in \mathbb{R}^{p_1\times p_2}$ has independent mean-zero symmetric sub-Gaussian entries:
\begin{equation}\label{eq:cond-subGaussian}
	\mathbb{E}Z_{ij}=0,\quad \Var(Z_{ij}) =  \sigma_{ij}^2, \quad \|Z_{ij}/\sigma_{ij}\|_{\psi_2} \leq \kappa.
\end{equation}	
	$M\in \mathbb{R}^{m_1\times m_2}$ has i.i.d. standard Gaussian entries. When $q\geq 1$, $m_1 = \lceil\sigma_C^2\rceil+q-1$, $m_2 = \lceil\sigma_R^2\rceil+q-1$, we have
	\begin{equation}\label{ineq:ZZHH-subGaussian}
	\mathbb{E}\tr\left\{(ZZ^\top - \mathbb{E}ZZ^\top)^q\right\} \leq (C\kappa)^{2q} \left(\frac{p_1}{m_1}\wedge \frac{p_2}{m_2}\right)\mathbb{E}\tr\left\{(HH^\top - \mathbb{E}HH^\top)^q\right\}.
	\end{equation} 
\end{lemma}
The proof of Lemma \ref{lm:subGaussian-comparison} is deferred to Section \ref{sec:proof-nonGaussian}. 

\begin{remark}[Proof ideas of Lemma \ref{lm:subGaussian-comparison}] {\rm
Compared to the proof of Lemma \ref{lm:Gaussian-comparison}, the proof of Lemma \ref{lm:subGaussian-comparison} requires more delicate scheme to bound $\bbE G_{ij}^{\alpha_{ij}(\mathbf{c})}(G_{ij}^2-1)^{\beta_{ij}(\mathbf c)}$ for non-standard-Gaussian distributed $G_{ij}:= Z_{ij}/\sigma_{ij}$. To this end, we introduce Lemma \ref{lm:Gaussian-moments} to bound $\bbE G_{ij}^{\alpha_{ij}(\mathbf{c})}(G_{ij}^2-1)^{\beta_{ij}(\mathbf c)}$ by a Gaussian analog:
\begin{equation*}
	\bbE G_{ij}^{\alpha_{ij}(\mathbf{c})}(G_{ij}^2-1)^{\beta_{ij}(\mathbf c)} \leq (C\kappa)^{2q}\bbE G^{\alpha_{ij}(\mathbf{c})}(G^2-1)^{\beta_{ij}(\mathbf c)}, \quad G\sim N(0, 1).
\end{equation*}
}\end{remark}

As a consequence of Lemma \ref{lm:subGaussian-comparison}, we have the following Wishart-type Concentration of sub-Gaussian random matrix.
\begin{corollary}[Wishart-type concentration of sub-Gaussian random matrix]\label{cr:heteroskedastic-wishart-concentration}
	Suppose $Z \in \bbR^{p_1 \times p_2}$ has independent mean-zero sub-Gaussian entries that satisfy \eqref{eq:cond-subGaussian}. Then
	\begin{equation}\label{ineq:EZZ^top}
	\begin{split}
	\mathbb{E}\left\|ZZ^\top - \mathbb{E} ZZ^\top\right\| \lesssim & \kappa^2\left( \sigma_C\sigma_R+\sigma_C^2 + \sigma_R\sigma_\ast\sqrt{\log(p_1\wedge p_2)} + \sigma_\ast^2\log(p_1\wedge p_2)\right).
	\end{split}
	\end{equation}
\end{corollary}
\begin{proof}[Proof of Corollary \ref{cr:heteroskedastic-wishart-concentration}]
When all $Z_{ij}$'s are symmetrically distributed, Corollary \ref{cr:heteroskedastic-wishart-concentration} follows from the proof of Theorem \ref{th:wishart-Gaussian} along with Lemmas \ref{lm:iid-Gaussian-moment} and \ref{lm:subGaussian-comparison}. If $Z_{ij}$'s are not all symmetric, let $Z'$ be an independent copy of $Z$, then each entry of $Z-Z'$ has independent symmetric sub-Gaussian distribution. By Jensen's inequality, we have \begin{equation*}
	\begin{split}
		\bbE \left\|ZZ^\top - \bbE ZZ^\top\right\| &= \bbE\left\|ZZ^\top + \bbE'  Z'(Z')^\top - Z(\bbE' Z')^\top - (\bbE' Z')Z^\top - 2\bbE ZZ^\top\right\|  \\
		& = \bbE\left\|\bbE\left\{ZZ^\top + Z'(Z')^\top - Z(Z')^\top - (Z') Z^\top - 2\bbE ZZ^\top\Big| Z\right\}\right\|  \\
		& \leq  \bbE\left[\bbE \left\{\left\|ZZ^\top + Z'(Z')^\top - Z(Z')^\top - (Z')Z^\top - 2\bbE ZZ^\top \right\|\Big| Z \right\} \right] \\
		& = \bbE\left[\bbE' \left\|(Z-Z')(Z-Z')^\top - \bbE(Z-Z')(Z-Z')^\top\right\| \right]  \\
		& \lesssim \kappa^2\left( \sigma_C\sigma_R+\sigma_C^2 + \sigma_R\sigma_\ast\sqrt{\log(p_1\wedge p_2)} + \sigma_\ast^2\log(p_1\wedge p_2)\right).
	\end{split}
\end{equation*}
\end{proof}
Next, we turn to the Wishart-type concentration for random matrix $Z$ with heavy-tailed entries. 
\begin{theorem}[Wishart-type concentration for heavy-tailed random matrix]\label{th:heavier-tail-wishart}
Suppose $\alpha \leq 1$, $Z\in \mathbb{R}^{p_1\times p_2}$ has independent entries, $\Var(Z_{ij}) \leq \sigma_{ij}^2$, and $\|Z_{ij}/\sigma_{ij}\|_{\psi_\alpha} \leq \kappa$ for all $i, j$. Given $\sigma_C, \sigma_R$, and $\sigma_{\ast}$ defined in \eqref{eq:sigma_C-sigma_R-sigma_ast}, we have
	\begin{equation*}
	\mathbb{E}\left\|ZZ^\top - \mathbb{E} ZZ^\top \right\| \lesssim \left(\sigma_C + \sigma_R + \sigma_{\ast}(\log (p_1\wedge p_2))^{1/2}(\log (p_1\vee p_2))^{1/\alpha - 1/2}\right)^2 - \sigma_R^2.
	\end{equation*}
\end{theorem}

In a variety of applications, the observations and random perturbations are naturally bounded (e.g., adjacency matrix in network analysis \cite{newman2013spectral} and single-nucleotide polymorphisms (SNPs) data in genomics \cite{syvanen2001accessing}). Thus, we provide a Wishart-type concentration for entrywise uniformly bounded random matrices as follows.
\begin{theorem}[Wishart-type concentration of bounded random Matrix]\label{th:heter-wishart-bounded}
	Suppose $Z\in \mathbb{R}^{p_1\times p_2}$, $\mathbb{E}Z_{ij} = 0, \Var(Z_{ij}) = \sigma_{ij}^2, |Z_{ij}|\leq B$ almost surely, then
	\begin{equation*}
	\begin{split}
	& \mathbb{E}\left\|ZZ^\top - \mathbb{E}ZZ^\top\right\|\\
	\leq & (1+\epsilon_1)\left\{2\sigma_C\sigma_R + (1+\epsilon_2)\sigma_C^2 + C_1(\epsilon_1)B\sigma_R\sqrt{\log(p_1 \wedge p_2)} + C_2(\epsilon_1,\epsilon_2)B^2\log(p_1 \wedge p_2)\right\},
	\end{split}
	\end{equation*}
	where $C_1(\epsilon_1)$ and $C_2(\epsilon_1,\epsilon_2)$ are defined as in Theorem \ref{th:wishart-Gaussian}.
	If we further have $ \max_{i,j}\sigma_{ij} \le \sigma_\ast$ and $ B\left(\log(p_1\wedge p_2)/p_1\right)^{1/2}\ll \sigma_{\ast}$ for some $\sigma_{\ast}$, then
	\begin{equation}
	\mathbb{E}\left\|ZZ^\top - \mathbb{E}ZZ^\top \right\| \leq (1+\epsilon)\left(2\sqrt{p_1p_2} + p_1\right)\sigma_\ast^2.
	\end{equation}
	
\end{theorem}

An immediate application of the previous theorem is the following Wishart-type concentration for independent Bernoulli random matrices. 
\begin{corollary}[Wishart-type Concentration of Bernoulli Random Matrix]\label{cr:wishart-bernoulli}
	Suppose $Z\in \mathbb{R}^{p_1\times p_2}, A_{ij}\overset{ind}{\sim} {\rm Bernoulli}(\theta_{ij})$, $\theta_{ij} \leq \theta_\ast$ and $\theta_\ast\geq C\log(p_1\wedge p_2)/p_1$. Then,
	\begin{equation}
	\mathbb{E}\left\|(A-\Theta)(A-\Theta)^\top - \mathbb{E}(A-\Theta)(A-\Theta)^\top\right\| \lesssim \left(\sqrt{p_1p_2} + p_1\right)\theta_\ast.
	\end{equation}
\end{corollary}

To prove Theorems \ref{th:heavier-tail-wishart} and \ref{th:heter-wishart-bounded}, we establish the corresponding comparison lemmas for random matrices with heavy tail/bounded distributions, which is more technically involved from Gaussian/sub-Gaussian distributions due to the essential difference. The proofs of Theorems \ref{th:heavier-tail-wishart} and \ref{th:heter-wishart-bounded} are provided in Section \ref{sec:proof-nonGaussian}.

\begin{remark}{\rm
	It is helpful to summarize the heteroskedastic Wishart-type concentration inequalities with Gaussian, sub-Gaussian, heavy-tail, and bounded entries in a unified form:
	$$\mathbb{E}\left\|ZZ^\top - \mathbb{E}ZZ^\top\right\| \leq C_0 \left\{\left(\sigma_C+\sigma_R + K\right)^2 - \sigma_R^2\right\},$$
	where $K = \sigma_* (\log(p_1\wedge p_2))^{1/2}$ and $C_0>1$ is a constant if the entries of $Z$ are sub-Gaussian; $K=\sigma_* (\log(p_1\wedge p_2))^{1/2}(\log(p_1\vee p_2))^{1/\alpha -1/2}$ and $C_0>1$ is a constant if $Z$ has bounded $\psi_{\alpha}$ norm; $K = C\sqrt{\log(p_1\wedge p_2)}$ and $C_0 = 1+\varepsilon$ if the entries of $Z$ are bounded; and  $ K = C\sigma_* (\log(p_1\wedge p_2))^{1/2}$  and $C_0 = (1+\varepsilon)$ if the entries of $Z$ are Gaussian. 
}\end{remark}

\subsection{Moments and tail bounds}\label{sec:tail-bounds}

We study the general $b$-th moment and the tail probability of heteroskedastic Wishart-type matrix in the following theorem. 
\begin{theorem}[High-order moments and tail probability bounds]\label{th:higher-moment-tail-bound}
	Suppose the conditions in Theorem \ref{th:wishart-Gaussian} hold. For any $b>0$, we have
	\begin{equation}
	\left\{\mathbb{E}\left\|ZZ^\top - \mathbb{E}ZZ^\top\right\|^b\right\}^{1/b} \lesssim \left(\sigma_C+\sigma_R+\sigma_*\sqrt{b\vee \log(p_1\wedge p_2)}\right)^2-\sigma_C^2.
	\end{equation}
	There exists uniform constant $C>0$ such that for any $x>0$,
	\begin{equation}\label{ineq:tail-bound}
	\bbP\left\{\left\|ZZ^\top - \mathbb{E}ZZ^\top\right\| \geq C\left(\left(\sigma_C+\sigma_R+\sigma_*\sqrt{\log(p_1\wedge p_2)}+x\right)^2 - \sigma_C^2\right) \right\}\leq \exp(-x^2).
	\end{equation}
\end{theorem}
Since neither $\|ZZ^\top - \mathbb{E}ZZ^\top\|$ nor $\|ZZ^\top-\mathbb{E}ZZ^\top\|^{1/2}$ are Lipschitz continuous in $Z$, the classic Talagrand's concentration inequality \cite[Theorem 6.10]{boucheron2013concentration} does not directly apply to give the tail probability bound of $\|ZZ^\top - \mathbb{E}ZZ^\top\|$. We instead prove \eqref{ineq:tail-bound} via a more direct moment method. The complete proof is given in Section \ref{sec:proof-higher-moment}.

\subsection{Wishart matrix with near-homoskedastic rows}\label{sec:homoskedastic rows}

In this section, we consider a special class of heteroskedastic matrices. Let $Z\in \mathbb{R}^{p_1\times p_2}$ be a random matrix with independent, sub-Gaussian, and zero-mean entries. Suppose all entries in the same row of $Z$ share similar variance (i.e., there exists $\sigma_i^2$ such that $\sigma_{ij}$ approximately equals $\sigma_i^2$ for all $i, j$). Then the $p_2$ columns of $Z$, i.e., $\{Z_{\cdot j}\}_{j=1}^{p_2}$, have approximately equal covariance matrix, $\diag(\sigma_1^2,\ldots, \sigma_{p_1}^2)$. In this case, $\frac{1}{n}ZZ^\top = \frac{1}{n}\sum_{j=1}^n Z_{\cdot j}Z_{\cdot j}^\top$ is the sample covariance matrix. It is of great interest to analyze $\|ZZ^\top - \mathbb{E} ZZ^\top\|$, i.e., the concentration of the sample covariance matrix in both probability and statistics  \cite{bai1993convergence,cai2010optimal}. 

Note that Corollary \ref{cr:heteroskedastic-wishart-concentration} directly implies
\begin{equation}\label{ineq:homo-column-nonsharp}
	\mathbb{E}\|ZZ^\top - \mathbb{E}ZZ^\top \| \lesssim \sum_{i} \sigma_i^2 + \sqrt{p_2\sum_{i} \sigma_i^2}\cdot\max_i\sigma_i + \sqrt{p_2\log(p_1\wedge p_2)}\max_i \sigma_i^2.
\end{equation}
With a more careful analysis, we can derive a better concentration inequality than \eqref{ineq:homo-column-nonsharp} without the logarithmic terms.
\begin{theorem}\label{th:column-wise-concentration}
	Suppose $Z$ is a $p_1$-by-$p_2$ random matrix with independent mean-zero sub-Gaussian entries. If there exist $\sigma_1,\ldots, \sigma_p \geq 0$ such that $\|Z_{ij}/\sigma_{i}\|_{\psi_2} \leq C_K$ for constant $C_K>0$, then
	\begin{equation}\label{ineq:column-wise-concentration-target}
	\mathbb{E}\left\|ZZ^\top - \mathbb{E}ZZ^\top\right\| \lesssim \sum_{i}\sigma_i^2 + \sqrt{p_2\sum_i\sigma_i^2}\cdot\max_i\sigma_i.
	\end{equation}
\end{theorem}
\begin{remark}{\rm 
We also note that a similar result of Theorem \ref{th:column-wise-concentration} can be derived from Koltchinskii and Lounici \cite{koltchinskii2017concentration}. Their result is based on generic chaining argument with the assumption that all columns of $Z$ are i.i.d. Here, we assume independence and an upper bound on the Orlicz-$\psi_2$ norm of each entry, while allow the distributions to be non-identical.
}\end{remark}

The following theorem gives a lower bound on the concentration of Wishart matrix with homoskedastic rows.
\begin{theorem}\label{th:row-wise-concentration-lower}
	If $Z\in \mathbb{R}^{p_1\times p_2}$, $Z_{ij}\overset{ind}{\sim}N(0, \sigma_i^2)$, we have
	\begin{equation*}
	\mathbb{E}\left\|ZZ^\top - \mathbb{E}ZZ^\top\right\| \gtrsim \sum_i \sigma_i^2 + \sqrt{p_2\sum_i \sigma_i^2}\cdot \max_i \sigma_i.
	\end{equation*}
\end{theorem}
The proof of Theorem \ref{th:row-wise-concentration-lower} is deferred to Section \ref{sec:proof-homo-row}. Theorems \ref{th:column-wise-concentration} and \ref{th:row-wise-concentration-lower} render an exact rate of Wishart-type concentration for random matrices with homoskedastic rows: 
$$\mathbb{E}\|ZZ^\top - \mathbb{E}ZZ^\top\| \asymp \sum_i \sigma_i^2 + \sqrt{p_2 \sum_i \sigma_i^2 }\max_i \sigma_i, \quad \text{if } \Var(Z_{ij}) \overset{ind}{\sim} N(0, \sigma_i^2).$$

The rest of this section is dedicated to the proof of Theorems \ref{th:column-wise-concentration}. We only prove for Gaussian Wishart-type random matrices since the sub-Gaussian case follows similarly. We first introduce a key tool to sequentially reduce the number of rows of $Z$. The tool, as summarized in the following lemma, may of independent interest.
\begin{lemma}[Variance contraction inequality of Gaussian random matrix]\label{lm:ZZtop-tildeZZtop}
	Suppose $G\in \mathbb{R}^{p_1\times p_2}$ and $\tilde{G}\in \mathbb{R}^{(p_1-1)\times p_2}$ are two random matrices with independent Gaussian entries satisfying 
	$$\mathbb{E}G_{ij} = \mathbb{E}\tilde{G}_{ij} =0,\quad \Var(G_{ij}) = \sigma_{ij}^2, \quad  \Var(\tilde{G}_{ij}) = \left\{\begin{array}{ll}
	\sigma_{ij}^2, & 1\leq i \leq p_1-2;\\
	\sigma_{p_1-1,j}^2+\sigma_{p_1, j}^2, & i=p_1-1.
	\end{array}\right.$$ 
	In other words, $G$ and $\tilde{G}$ are identical distributed in their first $(p_1-2)$ rows; the variance of the last row of $\tilde G$ is the sum of last two rows' variances of $G$. Then for any positive integer $q$, 
	\begin{equation*}
	\tr\left(\left(GG^\top - \mathbb{E}GG^\top\right)^q \right) \leq \tr\left(\left(\tilde{G}\tilde{G}^\top-\mathbb{E}\tilde{G}\tilde{G}^\top\right)^q\right).
	\end{equation*}
\end{lemma}
The proof of Lemma \ref{lm:ZZtop-tildeZZtop} is provided in Section \ref{sec:proof-homo-row}.  Now we are ready to prove Theorem \ref{th:column-wise-concentration}.
\begin{proof}[Proof of Theorem \ref{th:column-wise-concentration}]
Denote $\sigma_C^2 = \sum_i \sigma_i^2$, $\sigma_{\ast} = \max_i \sigma_i$. Assume $\sigma_\ast = 1$ without loss of generality. Set $q = 2\lceil \sigma_C^2\rceil$. We use mathematical induction on $p_1$ to show the following upper bound: for some uniform constant $C>0$ (which does not dependent on $p_1, p_2, \sigma_C$), we have 
\begin{equation}\label{ineq:induction}
\left(\mathbb{E}\tr\left\{\left(ZZ^\top - \mathbb{E}ZZ^\top \right)^q \right\}\right)^{1/q} \leq C\left(\sigma_C^2 + \sqrt{p_2}\sigma_C\right).
\end{equation}
\begin{itemize}
	\item If $p_1\leq 2q$, Lemma \ref{lm:Gaussian-comparison} yields
	\begin{equation*}
	\mathbb{E}\tr\left\{(ZZ^\top-\mathbb{E}ZZ^\top)^q\right\} \leq \left(\frac{p_1}{m_1}\wedge \frac{p_2}{m_2}\right)\mathbb{E}\tr\left\{(HH^\top - \mathbb{E}HH^\top)^q\right\}.
	\end{equation*}
	Here, $H$ is a $m_1$-by-$m_2$ dimensional matrix with i.i.d. standard Gaussian entries and
	\begin{equation}\label{eq:m_1,m_2}
	m_1 = \lceil \sigma_C^2\rceil+q-1 = 3\lceil\sigma_C^2\rceil-1, \quad m_2 = p_2 +2\lceil\sigma_C^2\rceil-1.
	\end{equation}
	Additionally, by Lemma \ref{lm:iid-Gaussian-moment}, 
	\begin{equation*}
	\begin{split}
	\left(\mathbb{E}\left\{(ZZ^\top - \mathbb{E}ZZ^\top)^q\right\}\right)^{1/q} \leq & \left(\left(\frac{p_1}{m_1}\wedge \frac{p_2}{m_2}\right)\mathbb{E} \tr\left\{\left(HH^\top - \mathbb{E}HH^\top\right)^q\right\} \right)^{1/q}\\
	\leq & \left(\mathbb{E} \left(\frac{p_1}{m_1}\wedge \frac{p_2}{m_2}\right) m_1\|HH^\top - \mathbb{E}HH^\top\|^q\right)^{1/q} \\
	\leq &  p_1^{1/q}\left(2\sqrt{m_1m_2}+m_1 + 4(\sqrt{m_1}+\sqrt{m_2})\sqrt{q}+2q\right)\\
	\overset{\eqref{eq:m_1,m_2}}{\leq} & (2q)^{1/q}\cdot C\left(\sqrt{p_2}\sigma_C + \sigma_C^2\right)\\
	\leq & C\left(\sqrt{p_2}\sigma_C + \sigma_C^2\right),
	\end{split}
	\end{equation*}
	which implies \eqref{ineq:induction}. 
	\item Suppose the statement \eqref{ineq:induction} holds for $Z\in \bbR^{(p_1-1)\times p_2}$ for some $p_1> 2q$, we further consider the case where $Z\in \mathbb{R}^{p_1\times p_2}$. Note that
	$$1 = \sigma_{\ast}^2 =\sigma_1^2 \geq \sigma_2^2 \geq \cdots \geq \sigma_{p_1}^2 \geq 0. $$
	By such the ordering,
	\begin{equation}\label{ineq:sigma_p-1sigma_p}
	\sigma_{p_1-1}^2 +\sigma_{p_1}^2 \leq \frac{2}{p_1}\sum_{i=1}^{p_1} \sigma_i^2 = \frac{2}{p_1}\sigma_C^2 \leq \frac{2\sigma_C^2}{2q} \leq \frac{2\sigma_C^2}{4\lceil\sigma_C^2\rceil}\leq 1 = \sigma_{\ast}^2.
	\end{equation}
	By Lemma \ref{lm:ZZtop-tildeZZtop}, we have
	\begin{equation*}
	\tr\left((ZZ^\top - \mathbb{E}ZZ^\top)^q\right) \leq \tr\left((\tilde{Z}\tilde{Z}^\top - \mathbb{E}\tilde{Z}\tilde{Z}^\top)^q\right).
	\end{equation*}
	where $\tilde{Z}$ is a $(p_1-1)$-by-$p_2$ random matrix with independent entries and
	$$\mathbb{E}(\tilde{Z}) = 0,\quad \Var((\tilde{Z})_{ij}) = \left\{\begin{array}{ll}
	\sigma_i^2, & \text{ if } 1\leq i \leq p_1-2;\\
	\sigma_{p-1}^2+\sigma_p^2, & \text{ if } 1\le i \leq p_1-1.\\
	\end{array}\right. $$
	By \eqref{ineq:sigma_p-1sigma_p}, we have $\max_{i,j} \Var((\tilde{Z})_{ij})\leq \sigma_{\ast}^2$. Meanwhile, $\sum_{i=1}^{p_1-1} \Var((\tilde{Z})_{ij}) = \sum_{i=1}^{p_1} \sigma_i^2 = \sigma_C^2$. Thus, the induction assumption of \eqref{ineq:induction} implies
	\begin{equation*}
	\begin{split}
	\left(\mathbb{E}\left\{(ZZ^\top - \mathbb{E}ZZ^\top)^q\right\}\right)^{1/q} \leq \left(\mathbb{E}\left\{(\tilde{Z}\tilde{Z}^\top - \mathbb{E}\tilde{Z}\tilde{Z}^\top)^q\right\}\right)^{1/q} \leq C\left(\sqrt{p_2}\sigma_C + \sigma_C^2\right).
	\end{split}
	\end{equation*}
\end{itemize}
By induction, we have proved that \eqref{ineq:induction} holds in general. Therefore,
\begin{equation*}
\begin{split}
\mathbb{E}\left\|ZZ^\top - \mathbb{E}ZZ^\top \right\| \leq & \left(\mathbb{E}\tr\left\{\left(ZZ^\top - \mathbb{E}ZZ^\top\right)^q\right\}\right)^{1/q} \lesssim \sqrt{p_2}\sigma_C + \sigma_C^2.
\end{split}
\end{equation*}
\end{proof}

\subsection{Wishart matrix with near-homoskedastic columns}\label{sec:homoskedastic columns}

Let $Z\in \mathbb{R}^{p_1\times p_2}$ be a random matrix with independent entries. We consider another case of interest that all entries in each column of $Z$ have the similar variance (i.e., there exist $\sigma_j$ such that $\sigma_{ij} \approx  \sigma_j^2$, $\forall i, i' \in [p_1]$, $\forall j \in [p_2]$). This model has been used to characterize heteroskedastic independent samples in statistical applications \cite{hong2018asymptotic}. Applying Theorem \ref{th:wishart-Gaussian}, one obtains
\begin{equation}\label{ineq:row-wise-nonsharp}
\begin{split}
\mathbb{E}\left\|ZZ^\top - \mathbb{E}ZZ^\top\right\| \lesssim \sqrt{p_1\sum_j\sigma_j^2}\max_j\sigma_j + p_1\max_j\sigma_j^2. 
\end{split}
\end{equation} 
As the direct upper bound of \eqref{ineq:row-wise-nonsharp} may be sub-optimal, we prove the following upper and lower bounds via a more careful analysis.
\begin{theorem}\label{th:row-wise-concentration}
	Suppose $Z\in \mathbb{R}^{p_1\times p_2}$ has independent, mean-zero, and sub-Gaussian entries. Assume there exist $\sigma_1, \ldots, \sigma_n \geq 0$ such that $\|Z_{ij}/\sigma_j\|_{\psi_2}\leq C_K$ for constant $C_K>0$. Then,
	\begin{equation}\label{ineq:row-wise-concentration-target}
	\mathbb{E}\left\|ZZ^\top - \mathbb{E}ZZ^\top \right\| \lesssim \sqrt{p_1\sum_j \sigma_j^4} + p_1\max_j\sigma_j^2.
	\end{equation}
\end{theorem}

\begin{theorem}\label{th:column-wise-concentration-lower}
	If $Z\in \mathbb{R}^{p_1\times p_2}$, $Z_{ij} \overset{ind}{\sim} N(0, \sigma_j^2)$, we have
	\begin{equation*}
	\mathbb{E}\left\|ZZ^\top -\mathbb{E}ZZ^\top \right\| \gtrsim\sqrt{p_1\sum_{j} \sigma_j^4} + p_1\max_j\sigma_j^2.
	\end{equation*}
\end{theorem}

The proof of Theorem \ref{th:column-wise-concentration-lower} is deferred to Section \ref{sec:proof-homo-column}. 
Now we consider the proof of Theorem \ref{th:row-wise-concentration}. Since the Gaussian comparison lemma (Lemma \ref{lm:Gaussian-comparison}) cannot give the desired term $\sum_{j=1}^{p_2}\sigma_j^4$, we turn to study the expansion of $\bbE\tr\left\{\left(\Delta(ZZ^\top)\right)^q\right\}$, where $\Delta(ZZ^\top)$ equals to $ZZ^\top$ with all diagonal entries set to zero. The expansion of $\bbE\tr\left\{\left(\Delta(ZZ^\top)\right)^q\right\}$ can be related to the cycles in a complete graph for which every edge is visited $\{0, 4, 8, 12 \ldots\}$ times. Based on this new idea, we introduce the following lemma.  
\begin{lemma}\label{lm:diagonal-deletion-comparison}
	Suppose $Z\in \mathbb{R}^{p_1\times p_2}$, $Z_{ij}\overset{ind}{\sim} N(0, \sigma_{ij}^2)$, and $\sigma_{ij} \leq \sigma_j$. For a square matrix $A$, let $\Delta(A)$ be $A$ with all diagonal entries set to zero and $D(A)$ be $A$ with all off-diagonal entries set to zero. For any integer $q \geq 1$, suppose $H\in \mathbb{R}^{p_1\times m}$ have i.i.d. standard normal entries and $m = \lceil \sum_{j=1}^{p_2} \sigma_j^4\rceil + q -1 $. Then,
	\begin{equation}\label{ineq:home-comparison}
	\bbE\tr\left\{\left(\Delta(ZZ^\top)\right)^q\right\} \leq \bbE\tr\left\{\left(\Delta(HH^\top)\right)^q\right\}.
\end{equation}
\end{lemma}
The proof of Lemma \ref{lm:diagonal-deletion-comparison} is provided in Section \ref{sec:proof-homo-column}. Next, we prove Theorem \ref{th:row-wise-concentration}.
\begin{proof}[Proof of Theorem \ref{th:row-wise-concentration}]
Denote $\sigma_R^2 = \sum_j \sigma_j^2$, $\sigma_{\ast} = \max_i \sigma_i$. Without loss of generality, we assume $\sigma_* = 1$. Note that $\bbE \left\|ZZ^\top - \bbE ZZ^\top\right\| \leq \bbE\left\|D(ZZ^\top) - \bbE ZZ^\top\right\| + \bbE\left\|\Delta(ZZ^\top)\right\|.$
It suffices to bound the two terms separately. Since $D(ZZ^\top) - \bbE ZZ^\top$ is a diagonal matrix with independent diagonal entries, we have 
\begin{equation*}
	\left\|D(ZZ^\top) - \bbE ZZ^\top\right\| = \max_{i \in [p_1]}\left|\sum_{j=1}^{p_2}Z_{ij}^2 - \bbE \sum_{j=1}^{p_2}Z_{ij}^2\right|.
\end{equation*}
With Bernstein inequality and union bound, we have 
\begin{equation*}
	\bbP\left(\max_{i\in [p_1]}\left|\sum_{j=1}^{p_2}Z_{ij}^2 - \bbE \sum_{j=1}^{p_2}Z_{ij}^2\right| > t\right) \leq 2\exp\left(\log p_1-c\left(\frac{t^2}{\sum_{j=1}^{p_2} \sigma_j^4} \wedge \frac{t}{\sigma_*^2}\right)\right).
\end{equation*}
Integration over the tail further yields
\begin{equation}\label{ineq:home-row-target1}
	\bbE \max_{i\in [p_1]}\left|\sum_{j=1}^{p_2}Z_{ij}^2 - \bbE \sum_{j=1}^{p_2}Z_{ij}^2\right| \lesssim \sqrt{\log p_1 \sum_{j=1}^{p_2} \sigma_j^4} + \sigma_*^2 \log p_1.
\end{equation}

Next, we use moment method to bound $\bbE\left\|\Delta(ZZ^\top)\right\|$.  For any even positive integer $q$, by Lemma \ref{lm:diagonal-deletion-comparison},
\begin{equation}\label{ineq:home-row-1}
	\bbE\left\|\Delta(ZZ^\top)\right\| \leq \left(\bbE\tr\left\{\left(\Delta(ZZ^\top)\right)^q\right\}\right)^{1/q} \leq \left(\bbE\tr\left\{\left(\Delta(HH^\top)\right)^q\right\}\right)^{1/q}.
\end{equation}
Here $H$ is a $p_1$-by-$m$ random matrix with i.i.d. $N(0,1)$ entries and $m = \lceil \sum_{j=1}^{p_2} \sigma_j^4\rceil + q -1  $. Thus it suffices to bound $\left(\bbE\tr\left\{\left(\Delta(HH^\top)\right)^q\right\}\right)^{1/q}$.

On the one hand, by Lemma \ref{lm:iid-Gaussian-moment}, $\forall q \geq 2$,
\begin{equation}\label{ineq:home-row-2}
	\left(\mathbb{E}\|HH^\top - \mathbb{E}HH^\top\|^{q}\right)^{1/q} \leq 2\sqrt{p_1m} + m + 4(\sqrt{p_1}+\sqrt{m})\sqrt{q} + 2q.
\end{equation}
On the other hand, note that $\left\|D(HH^\top) - \bbE HH^\top\right\| = \max_{i \in [m]}|X_i|$, where $X_i$ are independent centralized $\chi^2_{m}$ random variable. By the Chi-square concentration and union bound, we have
\begin{equation*}
	\begin{split}
		\bbP\left(\max_{i \in [p_1]}|X_i|^q > t\right) \leq 2\exp\left(\log p_1 - c\left(\frac{t^{2/q}}{m} \wedge t^{1/q}\right)\right).
	\end{split}
\end{equation*}
Integration gives
\begin{equation}\label{ineq:home-row-3}
	\bbE \max_{i \in [p_1]}|X_i|^q \leq C^q \left(\log^q p_1 + (\sqrt{m\log p_1})^{q} \right).
\end{equation}
Then it follows that
\begin{equation}\label{ineq:home-row-4}
	\begin{split}
	& \left(\bbE\tr\left\{\left(\Delta(HH^\top)\right)^q\right\}\right)^{1/q} \leq \left(p_1\bbE\left\|\Delta(HH^\top)\right\|^q\right)^{1/q} \\
	\leq &  p_1^{1/q}\left(\bbE\left\|HH^\top - \bbE HH^\top\right\|^q\right)^{1/q} + \left(\bbE\left\|D(HH^\top - \bbE HH^\top)\right\|^q\right)^{1/q}\\
	\overset{\eqref{ineq:home-row-2}\eqref{ineq:home-row-3}}{\lesssim} & p_1^{1/q}\cdot \left(\sqrt{p_1m} + p_1 + 4(\sqrt{p_1}+\sqrt{m})\sqrt{q} + 2q\right).
	\end{split}
\end{equation}
 Now we specify $q = 2p_1$ and get
 \begin{equation*}
 	\bbE\left\|\Delta(ZZ^\top)\right\| \overset{\eqref{ineq:home-row-1}}{\leq} \left(\bbE\tr\left\{\left(\Delta(HH^\top)\right)^q\right\}\right)^{1/q} \lesssim \sqrt{p_1\sum_{j=1}^n \sigma_j^4} + p_1.
 \end{equation*}
This together with \eqref{ineq:home-row-target1} completes the proof of this theorem.
\end{proof}

\section{Applications}\label{sec:application}

The concentration bounds established in the previous sections have a range of applications. 
In this section, we illustrate the usefulness of the heteroskedastic Wishart-type concentration by applications to low-rank matrix denoising and  heteroskedastic clustering. 

Consider the following ``signal + noise" model:
\begin{equation*}
	Y = X + Z,
\end{equation*}
where $X \in \bbR^{p_1\times p_2}$ is a (approximately) low-rank matrix of interest, $Z$ is the random noise with independent entries, and $Y$ is the observation. This model has attracted significant attention in probability and statistics \cite{bao2018singular,benaych2012singular,donoho2014minimax,shabalin2013reconstruction}, and has also been the prototypical setting in various applications, such as bipartite stochastic block model \cite{florescu2016spectral}, exponential family PCA \cite{liu2016pca}, top-$k$ ranking from pairwise comparison \cite{negahban2017rank}. In these applications, the leading singular values/vectors of $X$ often contain information of interest. A straightforward way to estimate the leading singular values/vectors of $X$ (which are also the square root eigenvalues and the eigenvectors of $XX^\top$) is by evaluating the spectrum of $Y$ (or equivalently $YY^\top$). Suppose $\lambda_i(YY^\top), \lambda_i(XX^\top), v_i(YY^\top), v_i(YY^\top)$ are the $i$th eigenvalue and $i$th eigenvector of $YY^\top, XX^\top$, respectively. The classic perturbation theory (e.g., Weyl \cite{weyl1912asymptotische} and David-Kahan \cite{davis1970rotation}) yield the following sharp bounds,
\begin{equation*}
|\lambda_i(YY^\top) - \lambda_i (XX^\top)| \leq \|YY^\top - XX^\top\|,
\end{equation*}
\begin{equation*}
\|v_i(YY^\top) \pm v_i(YY^\top)\|_2 \lesssim \frac{\|YY^\top-XX^\top\|}{\min_{j=i, i+1}\{\lambda_{j-1}(XX^\top) - \lambda_j(XX^\top)\}}.
\end{equation*}
Then, a tight upper bound for the perturbation $YY^\top - XX^\top$ is critical to quantify the estimation accuracy of $\lambda_i(YY^\top)$, $v_i(YY^\top)$ to $\lambda_i(XX^\top)$, $v_i(XX^\top)$. By expansion, the perturbation of $YY^\top - XX^\top$ can be written as 
\begin{equation}\label{eq:gram-decompose}
	YY^\top - XX^\top = XZ^\top + ZX^\top + \mathbb{E}ZZ^\top +  (ZZ^\top - \bbE ZZ^\top).
\end{equation}
Here, $\bbE ZZ^\top$ is a deterministic diagonal matrix; $\|XZ^\top\| = \|ZX^\top\|$ are the spectral norm of a random matrix multiplied by a deterministic matrix, which has been considered in \cite{vershynin2011spectral}; The term $\|ZZ^\top - ZZ^\top\|$ can often be the dominating and most complicated part in \eqref{eq:gram-decompose} and the heteroskedastic Wishart-type concentration inequality established in the present paper provides a powerful tool for analyzing it.

We further illustrate through a specific application to high-dimensional heteroskedastic clustering. The clustering is an ubiquitous task in statistics and machine learning \cite{hastie2009elements}. Suppose we observe a two-component Gaussian mixture:
\begin{equation}\label{eq:clustering-model}
	Y_j = l_j \mu + \varepsilon_{j},\qquad \varepsilon_{j} = (\varepsilon_{1j},\ldots, \varepsilon_{pj})^\top, \quad \varepsilon_{ij} \overset{ind}{\sim} N(0,\sigma_i^2 ),\qquad j =1,\ldots,n.
\end{equation}
Here, $\mu$ is an unknown deterministic vector in $\bbR^{p}$ and $l_j \in \{-1, 1\}$ are unknown labels of two classes. While most existing works focus on the homoskedastic setting, we consider a heteroskedastic setting where the noise variance $\sigma_i^2$ may vary across different coordinates. Then, the sample $\{Y_j\}_{j=1}^n$ can be written in a matrix form, $Y = X + Z$, where 
$$Y = \left[Y_1^\top, Y_2^\top, \cdots, Y_n^\top\right]^\top, \quad X = [l_1,l_2,\cdots,l_n]^\top \mu,\quad \text{and} \quad Z = (\varepsilon_{ij}).$$
Our goal is to cluster $\{Y_j\}_{j=1}^n$ into two groups, or equivalently to estimate the hidden label $\{l_j\}_{j=1}^n$. Let $\hat v$ be the first eigenvector of $Y Y^\top$. As $\hat{v}$ is an estimation of $l$, it is straightforward to cluster as
\begin{equation}\label{eq:clustering-est}
	\hat l_j = \text{sgn}(\hat v_j), \quad j=1,\ldots, n.
\end{equation}
Applying Theorem \ref{th:row-wise-concentration} and perturbation bound of $\|XZ^\top\|$ \cite[Lemma 3]{zhang2018heteroskedastic} on \eqref{eq:gram-decompose}, it can be shown that
\begin{equation*}
	\bbE \left\|YY^\top - \bbE ZZ^\top - XX^\top\right\| \lesssim n\left\|\mu\right\|\sigma_* + n \sigma_*^2 + \sqrt{n\sum_j \sigma_j^4}. 
\end{equation*}
Combining this with the Davis-Kahan Theorem \cite{davis1970rotation}, we obtain the following result.
\begin{theorem}\label{th:hetero-clustering}
	Let $\sigma_* = \max_i \sigma_i$ and $\tilde{\sigma} = (\sum_i \sigma_i^4)^{1/4}$. The estimator in \eqref{eq:clustering-est} satisfies
	\begin{equation}\label{eq:clustering-bound}
		\bbE \mathcal M(l,\hat l) \lesssim \frac{n\left\|\mu\right\|_2  \sigma_* + n\sigma_*^2 + \sqrt{n}\tilde{\sigma}^2 }{n\left\|\mu\right\|_2^2} \wedge 1.
	\end{equation}
	Here, $\mathcal M(l,\hat l)$ is the misclassification rate defined as 
	\begin{equation}\label{eq:def-misclassify}
	\mathcal M(l,\hat l) = \frac{1}{n}\min\left\{\sum_{i=1}^n 1_{\{l_i \neq \hat l_i\}}, ~ \sum_{i=1}^n 1_{\{l_i \neq -\hat l_i\}}\right\}.
	\end{equation}
\end{theorem}
The complete proof of Theorem \ref{th:hetero-clustering} is deferred to Section \ref{sec:proof-hetero-clustering}. By \eqref{eq:clustering-bound}, the clustering is consistent (i.e., $\bbE\mathcal M(l,\hat l) = o(1)$) as long as 
\begin{equation}\label{ineq:clustering-SNR}
	\left\|\mu\right\|_2 \gg \sigma_* \vee (\tilde{\sigma}/n^{1/4}).
\end{equation}
The following lower bound shows that the signal-noise-ration condition \eqref{ineq:clustering-SNR} is necessary to ensure a consistent classification. The proof is provided in Section \ref{sec:proof-hetero-clustering}.
\begin{theorem}\label{th:hetero-clustering-lower-bound}
	Suppose $\sigma_* \leq \tilde{\sigma} \leq p^{1/4}\sigma_*$. Consider the following class of distributions on $\bbR^{n \times p}$:
	\begin{equation*}
		\mathcal P_{l,\lambda}(\sigma_*, \tilde\sigma) = \left\{P_Y: Y = X + Z\in \mathbb{R}^{n\times p}: \begin{array}{ll}
	& X =  l\mu^\top, Z_{ij}\overset{ind}{\sim} N(0, \sigma_{j}^2), \\
	& \left\|\mu\right\| \geq \lambda, \max_{j}\sigma_{j}\leq \sigma_\ast, \sum_{i=1}^p\sigma_i^4 \leq \tilde\sigma^4
	\end{array}\right\}.
	\end{equation*}
	There exists a universal constant $c>0$, such that if $\lambda < c\left(\sigma_* \vee (\tilde\sigma/n^{1/4})\right)$, we have
	\begin{equation*}
		\inf_{\hat l} \sup_{\mathcal P_{l,\lambda}(\sigma_*,\tilde{\sigma})} \bbE\mathcal M(l,\hat l) \geq 1/4.
	\end{equation*}
\end{theorem}

\section{Additional Proofs}\label{sec:proofs}

\subsection{Proofs for main results}\label{sec:proof-main}
In this section, we collect the proofs of upper and lower bound results in Section \ref{sec:main-results} including Lemma \ref{lm:Gaussian-comparison}, Lemma \ref{lm:iid-Gaussian-moment}, Proposition \ref{pr:lower} and Theorem \ref{th:hetero-clustering-lower-bound}.
\begin{proof}[Proof of Lemma \ref{lm:Gaussian-comparison}]
This proof shares similarity but shows more distinct aspects, compared with the one of Wigner-type \cite[Proposition 2.1]{bandeira16sharp}. We assume $\sigma_\ast = 1$ throughout the proof without loss of generality. We divide the proof into two steps, which targets on the two sides of the inequalities, respectively.
\begin{enumerate}
	\item[Step 1] One can check that $\mathbb{E} ZZ^\top = \diag\left(\left\{\sum_{j=1}^{p_2} \sigma_{ij}^2\right\}_{i=1}^{p_1}\right).$
	Consider the following expansion,
	\begin{equation}\label{eq:expansion-Gaussian}
	\begin{split}
	& \mathbb{E}\tr\left\{(ZZ^\top - \mathbb{E}ZZ^\top)^{q}\right\} =  \sum_{\substack{u_1, \ldots, u_{q}, u_{q+1} \in [p_1]}} \mathbb{E} \prod_{k=1}^q\left(ZZ^\top - \mathbb{E}ZZ^\top\right)_{u_k, u_{k+1}} \\
	= & \sum_{u_1, \ldots, u_q, u_{q+1} \in [p_1]} \mathbb{E} \prod_{k=1}^q \left\{\sum_{v_k \in [p_2]} \left(Z_{u_k, v_k} Z_{u_{k+1}, v_k} - 1_{\{u_k = u_{k+1}\}}\mathbb{E} Z_{u_k, v_k}^2\right)  \right\}\\
	= & \sum_{\substack{u_1, \ldots, u_{q}, u_{q+1} \in [p_1]\\v_1,\ldots, v_q \in [p_2]}} \mathbb{E}\prod_{k=1}^q \left(Z_{u_k, v_k} Z_{u_{k+1}, v_k} - \sigma_{u_k, v_k}^2\cdot 1_{\{u_k = u_{k+1}\}}\right).
	\end{split}
	\end{equation}
	Here, the indices are in module $q$, i.e., $u_1 = u_{q+1}$.
	Next, we consider the bipartite graph from $[p_1]$ on $[p_2]$ and the cycles of length $2q$, i.e., $\mathbf c := (u_1 \to v_1 \to u_2 \to v_2 \to \ldots \to u_q \to v_q \to u_{q+1} = u_1)$. For any $(i, j) \in [p_1]\times [p_2]$, let 
	\begin{equation}\label{eq:def-alpha-beta-L}
	\begin{split}
	& \alpha_{ij}(\mathbf c) = \text{Card} \left\{k: (u_k = i, v_k = j, u_{k+1} \neq i) \text{ or } (u_k \neq i, v_k = j, u_{k+1} = i)\right\};\\
	& \beta_{ij}(\mathbf c) = \text{Card} \left\{k: u_k = u_{k+1} = i, v_k = j\right\}.
	\end{split}
	\end{equation}
	Then, $\alpha_{ij}(\mathcal{L})$ is the number of times that the edge $(i, j)$ is visited exactly once by sub-path $u_k\to v_k \to u_{k+1}$; $\beta_{ij}(\mathbf{c})$ is the number of times that the edge $(i, j)$ is visited twice by sub-path $u_k\to v_k\to u_{k+1}$ (back and forth). Since $Z_{ij}/\sigma_{ij}$ has i.i.d. standard normal distribution, we have
	\begin{equation}\label{eq:take-variance-out}
	\begin{split}
	& \mathbb{E}\tr\left\{(ZZ^\top - \mathbb{E}ZZ^\top)^q\right\} =  \sum_{\mathbf{c}\in ([p_1]\times [p_2])^q} \prod_{(i, j)\in [p_1]\times [p_2]} \mathbb{E} Z_{ij}^{\alpha_{ij}(\mathbf{c})}\left(Z_{ij}^2 - \sigma_{ij}^2\right)^{\beta_{ij}(\mathbf{c})}\\
	= & \sum_{\mathbf{c}\in ([p_1]\times [p_2])^q} \prod_{(i, j)\in [p_1]\times [p_2]}\sigma_{ij}^{\alpha_{ij}(\mathbf{c})+2\beta_{ij}(\mathbf{c})}\prod_{(i, j)\in [p_1]\times [p_2]} \mathbb{E} G^{\alpha_{ij}(\mathbf{c})}\left(G^2 - 1\right)^{\beta_{ij}(\mathbf{c})}\\
	= & \sum_{\mathbf{c}\in ([p_1]\times [p_2])^q} \prod_{k=1}^q \sigma_{u_k, v_k} \sigma_{u_{k+1}, v_k}\prod_{(i, j)\in [p_1]\times [p_2]} \mathbb{E} G^{\alpha_{ij}(\mathbf{c})}\left(G^2 - 1\right)^{\beta_{ij}(\mathbf{c})}.
	\end{split}
	\end{equation}	
	Here $G$ denotes a $N(0,1)$ random variable. Next, let $m_{\alpha, \beta}(\mathbf{c})$ be the number of edges which appear $\alpha$ times in $(u_k \to v_k)$ or $(v_k \to u_{k+1})$ with $u_k \neq u_{k+1}$, and $\beta$ times in $(u_k \to v_k \to u_{k+1})$ with $u_k = u_{k+1}$. More rigorously,
	\begin{equation}\label{eq:def-m-alpha-beta}
	\begin{split}
	m_{\alpha, \beta}(\mathbf{c}) := \text{Card}\Big\{(i, j)&\in [p_1]\times[p_2]: \beta = |\{k: u_k=u_{k+1}=i, v_k = j\}|, \\
	& \alpha = |\{k: \text{exactly one of $u_k$ or $u_{k+1}$} = i, v_{k} = j|\Big\}.
	\end{split}
	\end{equation}
	For any cycle $\mathbf{c}$, we define its \emph{shape} $\mathbf s(\mathbf u)$ by relabeling the vertices in order of appearance. For example, the cycle $2\rightarrow 4' \rightarrow 3 \rightarrow 2' \rightarrow 2 \rightarrow 4' \rightarrow 5 \rightarrow 1' \rightarrow 2$ has shape $1\rightarrow 1' \rightarrow 2 \rightarrow 2' \rightarrow 1 \rightarrow 1' \rightarrow 3 \rightarrow 3' \rightarrow 1$. Here $i$ denotes the left vertex while $i'$ denotes the right vertex. It is easy to see for any two cycles $\mathbf{c}$ and $\mathbf{c}'$ with the same shape, we must have $m_{\alpha, \beta}(\mathbf{c}) = m_{\alpha,\beta}(\mathbf{c}')$. Thus we can well define $m_{\alpha,\beta}(\mathbf s (\mathbf c)) := m_{\alpha, \beta}(\mathbf{c})$. Based on previous discussions, 
	\begin{equation}\label{eq:power-as-shape}
	\prod_{(i, j)\in [p_1]\times [p_2]} \mathbb{E} G^{\alpha_{ij}(\mathbf{c})}\left(G^2 - 1\right)^{\beta_{ij}(\mathbf{c})} = \prod_{\substack{\alpha, \beta\geq 0\\}} \left\{\mathbb{E} G^\alpha (G^2 - 1)^\beta\right\}^{m_{\alpha, \beta}(\mathbf{s}(\mathbf{c}))}.
	\end{equation}
	Then a natural observation is that $\mathbb{E} G^\alpha (G^2 - 1)^\beta \geq 0$ for all non-negative $\alpha, \beta$ and $\mathbb{E} G^\alpha (G^2 - 1)^\beta = 0$ if and only if $\alpha$ is an odd or $\alpha = 0, \beta = 1$  (see Lemma \ref{lm:Gaussian-moments} in Appendix \ref{sec:proof-technical-lemma} for details). We then define \emph{even} shape set $\mathcal S_{p_1,p_2}$ as 
	\begin{equation}\label{eq:def-even-shape}
	\mathcal S_{p_1,p_2} = \left\{\mathbf s(\mathbf c): m_{\alpha,\beta}(\mathbf s(\mathbf c)) = 0 \text{ for all } \alpha, \beta~~s.t.~\alpha \text{ is an odd or } \alpha = 0, \beta = 1\right\}.
	\end{equation}
	Then the right hand side of \eqref{eq:power-as-shape} is nonzero only for $\mathbf s(\mathbf c) \in \mathcal S_{p_1, p_2}$ and the expansion \eqref{eq:take-variance-out} can be further rewritten as
	\begin{equation}\label{ineq:trace-expansion}
	\begin{split}
	& \mathbb{E}\tr\left\{(ZZ^\top - \mathbb{E}ZZ^\top)^{q}\right\}\\
	= & \sum_{\mathbf s_0 \in \mathcal{S}_{p_1,p_2}} \sum_{\substack{\mathbf{c}: \mathbf s(\mathbf c) = \mathbf{s}_0}} \prod_{k=1}^q \sigma_{u_k, v_k} \sigma_{u_{k+1}, v_k}  \prod_{\substack{\alpha, \beta \geq 0}} \left\{\mathbb{E} G^\alpha (G^2 - 1)^\beta\right\}^{m_{\alpha, \beta}(\mathbf{s_0})}\\
	= & \sum_{\mathbf s_0 \in \mathcal{S}_{p_1,p_2}} \prod_{\substack{\alpha, \beta \geq 0}} \left\{\mathbb{E} G^\alpha (G^2 - 1)^\beta\right\}^{m_{\alpha, \beta}(\mathbf{s}_0)} \cdot \sum_{\substack{\mathbf{c}: \mathbf s(\mathbf c) = \mathbf{s}_0}} \prod_{k=1}^q \sigma_{u_k, v_k} \sigma_{u_{k+1}, v_k}. 
	\end{split}
	\end{equation}
	Now denote $m_L(\mathbf{s}_0)$ and $m_R(\mathbf{s}_0)$ be the number of distinct left and right nodes that is visited by cycles with shape $\mathbf{s}_0$, we have the following lemma:
	\begin{lemma}\label{lm:shape-bound}
		Suppose $\sigma_* \leq 1$. Then for any shape $\mathbf s_0 \in \mathcal S_{p_1 ,p_2}$,
		\begin{equation*}
		\sum_{\substack{\mathbf{c}: \mathbf s(\mathbf c) = \mathbf{s}_0}} \prod_{k=1}^q \sigma_{u_k, v_k} \sigma_{u_{k+1}, v_k} \leq  \left(p_1\sigma_C^{2m_L(\mathbf{s}_0)-2} \sigma_R^{2m_R(\mathbf{s}_0)}\right) \wedge \left(p_2\sigma_C^{2m_L(\mathbf{s}_0)} \sigma_R^{2m_R(\mathbf{s}_0)-2}\right).
		\end{equation*}
	\end{lemma} 
	\begin{proof}
	The proof of Lemma \ref{lm:shape-bound} is an analogue of \cite[Lemma 2.5]{bandeira16sharp}. We first show
\begin{equation}\label{ineq:shape-bound-1-side}
	\sum_{\substack{\mathbf{c}: \mathbf s(\mathbf{c}) = \mathbf{s}_0}} \prod_{k=1}^q \sigma_{u_k, v_k} \sigma_{u_{k+1}, v_k} \leq  p_1\sigma_C^{2m_L(\mathbf{s}_0)-2} \sigma_R^{2m_R(\mathbf{s}_0)}.
\end{equation}
Suppose $\mathbf s_0 = (s_1, s_1', \ldots, s_q, s_q')$, let $l(k) = \min\{j: s_j = k\}$, i.e., the first time in any cycle of shape $\mathbf s_0$ at which its $k$th distinct left vertex is visited. Similarly we define $r(k) = \min\{j: s_j' = k\}$. Now let $\mathbf c = (u_1, v_1, \cdots u_q, v_q)$ be a cycle with shape $\mathbf s_0$. Then the following $m_L(\mathbf s_0)$ distinct edges from right vertex to left vertex will appear in order: $v_{l(2)-1} \rightarrow u_{l(2)}, v_{l(3)-1} \rightarrow u_{l(3)}, \cdots, v_{l(m_L(\mathbf s_0))-1} \rightarrow u_{l(m_L(\mathbf s_0))}$. Similarly, we have $m_R(\mathbf s_0)$ edges from left vertex to right vertex: $u_{r(1)} \rightarrow v_{r(1)}, u_{r(2)} \rightarrow v_{r(2)}, \cdots, u_{r(m_R(\mathbf s_0))} \rightarrow v_{r(m_R(\mathbf s_0))}$. In addition, these $m_L + m_R - 1$ edges are distinct by the definition of $l(k)$ and $r(k)$. We claim each of these $m_L + m_R - 1$ edges appear at least twice. Suppose one of the above edges only appear once, then we must have $m_{1,0}(\mathbf s(\mathbf c)) \geq 1$, which contradicts $\mathbf s_0 \in \mathcal S_{p_1+p_2}$. Now for a fixed starting vertex $u_1 = u \in [p_1]$, we can bound
\begin{equation*}
	\begin{split}
		&\sum_{\substack{\mathbf c: u_1 = u \\ \mathbf{s}(\mathbf c) = \mathbf s_0 }}\prod_{k=1}^q \sigma_{u_k, v_k} \sigma_{u_{k+1}, v_k} \\
		& \leq \sum_{\substack{\mathbf c: u_1 = u \\ \mathbf{s}(\mathbf c) = \mathbf s_0 }}\left(\sigma^2_{u_{r(1)},v_{r(1)}}\cdots \sigma^2_{u_{r(m_R(\mathbf s_0))},v_{r(m_R(\mathbf s_0))}}\right) \cdot \left(\sigma^2_{u_{l(2)},v_{l(2)-1}}\cdots \sigma^2_{u_{l(m_L(\mathbf s_0))},u_{l(m_R(\mathbf s_0))-1}}\right) \\
		& = \sum_{\substack{a_2 \neq \cdots \neq a_{m_L(\mathbf s_0)} \in [p_1] \\ b_1 \neq \cdots \neq b_{m_R(\mathbf s_0)} \in [p_2] }}\left(\sigma^2_{a_{s_{r(1)}},b_1}\cdots \sigma^2_{a_{s_{r(m_R(\mathbf s_0))}},b_{m_R(\mathbf s_0)}}\right) \cdot \left(\sigma^2_{a_{2},b_{s'_{l(2)-1}}}\cdots \sigma^2_{a_{m_L(\mathbf s_0)},b_{s'_{l(m_R(\mathbf s_0))-1}}}\right) \\
		& \leq  \sigma_R^{2m_R(\mathbf s_0)} \sigma_C^{2(m_L(\mathbf s_0)-1)}.
	\end{split}
\end{equation*} 
Then \eqref{ineq:shape-bound-1-side} follows by taking different initial vertices $u \in [p_1]$. Similarly we can show 
\begin{equation*}
	\sum_{\substack{\mathbf{c}: \mathbf s(\mathbf{c}) = \mathbf{s}_0}} \prod_{k=1}^q \sigma_{u_k, v_k} \sigma_{u_{k+1}, v_k} \leq  p_2\sigma_C^{2m_L(\mathbf{s}_0)} \sigma_R^{2m_R(\mathbf{s}_0)-2}
\end{equation*}
and the proof is complete.
	\end{proof}
	Combining \eqref{ineq:trace-expansion} and Lemma \ref{lm:shape-bound}, we obtain
	\begin{equation}\label{ineq:trace-bound-hetero}
	\begin{split}
	& \mathbb{E}\tr\left\{(ZZ^\top - \mathbb{E}ZZ^\top)^{q}\right\}\\
	\leq & \sum_{\mathbf{s}_0 \in \mathcal S_{p_1,p_2}} \prod_{\substack{\alpha, \beta \geq 0}} \left\{\mathbb{E} G^\alpha (G^2 - 1)^\beta\right\}^{m_{\alpha, \beta}(\mathbf{s}_0)}\\
	&  \cdot \left\{p_1 \sigma_C^{2(m_L(\mathbf{s}_0)-1)}\sigma_R^{2m_R(\mathbf{s}_0)}\right\} \wedge \left\{p_2 \sigma_C^{2m_L(\mathbf{s}_0)}\sigma_R^{2(m_R(\mathbf{s}_0)-1)}\right\}.
	\end{split}
	\end{equation}
	
	\item[Step 2] Next, we consider the expansion for $\mathbb{E}\tr\left((HH^\top)^q\right)$, where $H\in \mathbb{R}^{m_1\times m_2}$ is with i.i.d. standard Gaussian entries. We similarly expand as Step 1 to obtain
	\begin{equation*}
	\begin{split}
	& \mathbb{E}\tr\left((HH^\top - m_2 I_{m_1})^q\right)\\
	= & \sum_{\mathbf{s}_0 \in \mathcal S_{p_1, p_2}} \prod_{\substack{\alpha, \beta \geq 0}} \left\{\mathbb{E} G^\alpha (G^2 - 1)^\beta\right\}^{m_{\alpha, \beta}(\mathbf{s}_0)} \cdot \left|\left\{\mathbf c: \mathbf s(\mathbf c) = \mathbf s_0\right\}\right| \\
	= & \sum_{\mathbf{s}_0 \in \mathcal S_{p_1,p_2}} \prod_{\alpha, \beta \geq 0} \mathbb{E} \left\{G^\alpha (G^2 - 1)^\beta\right\}^{m_{\alpha, \beta}(\mathbf{s}_0)}\\ 
	& \cdot m_1(m_1-1)\cdots (m_1 - m_L(\mathbf{s}_0)+1) m_2(m_2-1)\cdots (m_2 - m_R(\mathbf{s}_0)+1)
	\end{split}
	\end{equation*}
	Provided that $m_1 = \lceil \sigma_C^2\rceil + q-1$ and $m_2 = \lceil\sigma_R^2\rceil+q-1$, $m_L(\mathbf{s}_0), m_R(\mathbf{s}_0)\leq q$, we have
	\begin{equation*}
	\begin{split}
	& m_1(m_1-1) \cdots (m_1-m_L(\mathbf{s}_0)+1)\cdot m_2(m_2-1)\cdots (m_2 - m_R(\mathbf{s}_0)+1)\\
	\geq & m_1 \cdot (m_1-m_L(\mathbf{s}_0)+1)^{m_L(\mathbf{s}_0)-1}\cdot (m_1-m_R(\mathbf{s}_0)+1)^{m_R(\mathbf{s}_0)} \\
	\geq & m_1 \sigma_C^{2m_L(\mathbf{s}_0)-2}\cdot \sigma_R^{2m_R(\mathbf{s}_0)}.
	\end{split}
	\end{equation*}
	Similarly, 
	\begin{equation*}
	\begin{split}
	& m_1(m_1-1) \cdots (m_1-m_L(\mathbf{s}_0)+1)\cdot m_2(m_2-1)\cdots (m_2 - m_R(\mathbf{s}_0)+1)\\
	\geq & \sigma_C^{2m_L(\mathbf{s}_0)} \cdot m_2 \sigma_R^{2m_R(\mathbf{s}_0)-2}.	
	\end{split}
	\end{equation*}
	These all together imply
	\begin{equation}\label{ineq:trace-bound-homo}
	\begin{split}
	\mathbb{E}\tr\left((HH^\top - m_2 I_{m_1})^q\right)
	\geq & \sum_{\mathbf{s}_0 \in \mathcal S_{p_1+p_2}} \prod_{\alpha, \beta \geq 0} \mathbb{E} \left\{G^\alpha (G^2 - 1)^\beta\right\}^{m_{\alpha, \beta}(\mathbf{s}_0)}\\
	& \cdot \left\{m_1 \sigma_C^{2m_L(\mathbf{s}_0)-2}\cdot \sigma_R^{2m_R(\mathbf{s}_0)}\right\} \vee \left\{m_2 \sigma_C^{2m_L(\mathbf{s}_0)}\cdot \sigma_R^{2m_R(\mathbf{s}_0)-2}\right\}.
	\end{split}
	\end{equation}
\end{enumerate}
By comparing \eqref{ineq:trace-bound-hetero} and \eqref{ineq:trace-bound-homo}, we have finally proved that
\begin{equation*}
\mathbb{E}\tr\left\{(ZZ^\top - \mathbb{E}ZZ^\top)^q\right\} \leq \left(\frac{p_1}{m_1}\wedge \frac{p_2}{m_2}\right) \mathbb{E} \tr\left\{\left(HH^\top - \bbE HH^\top\right)^q\right\}.
\end{equation*}
\end{proof}

\begin{proof}[Proof of Lemma \ref{lm:iid-Gaussian-moment}]
Let $W = \max\left\{\sigma_{\max}(H) - \sqrt{m_2} - \sqrt{m_1}, \sqrt{m_2} - \sqrt{m_1} -\sigma_{\min}(M), 0\right\}$, by the tail bound of i.i.d. Gaussian matrix (c.f., \cite[Corollary 5.35]{vershynin2010introduction}), $\bbP\left(W\geq t\right) \leq 2\exp(-t^2/2)$ for all $t\geq 0$. 
Thus for any $q \geq 1$,
\begin{equation*}
\begin{split}
\mathbb{E}W^q = & q\int_0^\infty t^{q-1}\bbP\left(W\geq t\right)dt \leq 2q\int_0^\infty t^{q-1}\exp(-t^2/2)dt  = 2^{\frac{q}{2}}q\Gamma(q/2).
\end{split}
\end{equation*}
Since
\begin{equation}
\begin{split}
\|HH^\top - \mathbb{E}HH^\top\| = & \|HH^\top - m_2I_{m_1}\| = \max\left\{\sigma_{\max}^2(H)-m_2, m_2 - \sigma_{\min}^2(H)\right\}\\
\leq & \left(W+\sqrt{m_1}+\sqrt{m_2}\right)^2 - m_2 \\
= & 2\sqrt{m_1m_2} + m_1 + W^2 + 2(\sqrt{m_1}+\sqrt{m_2})W,
\end{split}
\end{equation}
we have
\begin{equation*}
\begin{split}
& \left(\mathbb{E}\|HH^\top - \mathbb{E}HH^\top\|^{q}\right)^{1/q} \\
\leq & 2\sqrt{m_1m_2}+m_1 + (\mathbb{E} W^{2q})^{1/q} + 2(\sqrt{m_1}+\sqrt{m_2})\left(\mathbb{E}W^q\right)^{1/q}\\
\leq & 2\sqrt{m_1m_2} + m_1 + \left(2^{q+1}q\Gamma(q)\right)^{1/q} + 2(\sqrt{m_1}+\sqrt{m_2})\left(2^{\frac{q}{2}}q\Gamma(q/2)\right)^{1/q}.
\end{split}
\end{equation*}
Next we claim
\begin{equation}\label{ineq:to-verify}
\left(2^{q+1}q\Gamma(q)\right)^{1/q}\leq 2q,\quad \left(2^{\frac{q}{2}}q\Gamma(q/2)\right)^{1/q} \leq 2q^{1/2}.
\end{equation}
One can verify \eqref{ineq:to-verify} for $2\leq q\leq 10$ by calculation. When $q\geq 11$, \eqref{ineq:to-verify} can be verified by the Gamma function upper bound in \cite{batir2017bounds}. In summary, we have
$$\left(\mathbb{E}\|HH^\top - \mathbb{E}HH^\top\|^{q}\right)^{1/q} \leq 2\sqrt{m_1m_2} + m_1 + 4(\sqrt{m_1}+\sqrt{m_2})\sqrt{q} + 2q.$$
which has finished the proof of the first part of this lemma.

For the second part, when $m_1\leq m_2$, since $HH^\top - \mathbb{E}HH^\top$ is an $m_1$-by-$m_1$ matrix, we know $\tr(\left(HH^\top - \mathbb{E}HH^\top\right)^q)$ is the sum of $m_1$ eigenvalues of $\left(HH^\top - \mathbb{E}HH^\top\right)^q$, while each of these eigenvalues are no more than $\|HH^\top - \mathbb{E}HH^\top\|^q$. Thus,
\begin{equation*}
\begin{split}
& \mathbb{E} \tr\left\{\left(HH^\top - \mathbb{E}HH^\top\right)^q\right\} \leq \mathbb{E} m_1 \|HH^\top - \mathbb{E}HH^\top\|^q \\
\leq &  (m_1\wedge m_2) \cdot \left(2\sqrt{m_1m_2} + m_1 + 4(\sqrt{m_1}+\sqrt{m_2})\sqrt{q} + 2q\right)^{q}.
\end{split}
\end{equation*}

When $m_1 > m_2$, we shall note that $\rank(HH^\top) \leq m_2$ and $\mathbb{E}HH^\top = m_2I_{m_1}$. Then,
\begin{equation*}
\begin{split}
\left(HH^\top - \mathbb{E}HH^\top\right)^q - (-1)^q m_2^q I_{m_1} = \sum_{k=1}^{q} (-m_2)^{q-k}\binom{q}{m_1}(HH^\top)^k,
\end{split}
\end{equation*}
which shares the eigenspace of $HH^\top$ and has rank no more than $m_2$. Thus,
\begin{equation*}
\begin{split}
& \mathbb{E} \tr\left\{\left(HH^\top - \mathbb{E}HH^\top\right)^q\right\} = \mathbb{E}\tr\left\{\left(HH^\top - \mathbb{E}HH^\top\right)^q - (-1)^q m_2^q I_{m_1}\right\} + \tr\left((-1)^q m_2^qI_{m_1}\right)\\
\leq & m_2\mathbb{E} \left\|\left(HH^\top - \mathbb{E}HH^\top\right)^q - (-1)^q m_2^q I_{m_1}\right\| + m_1m_2^q\\
\leq & m_2 \left\{\left(2\sqrt{m_1m_2} + m_1 + 4(\sqrt{m_1}+\sqrt{m_2})\sqrt{q} + 2q\right)^q + m_2^q\right\} + m_1m_2^q\\
\leq & 2m_2\left(2\sqrt{m_1m_2} + m_1 + 4(\sqrt{m_1}+\sqrt{m_2})\sqrt{q} + 2q\right)^q\\
= & 2(m_1\wedge m_2)\left(2\sqrt{m_1m_2} + m_1 + 4(\sqrt{m_1}+\sqrt{m_2})\sqrt{q} + 2q\right)^q.
\end{split}
\end{equation*}
where the last inequality is due to $m_1>m_2$.
\end{proof}

\begin{proof}[Proof of Proposition \ref{pr:lower}]
	Since $Z_{ij}\overset{iid}{\sim}N(0, 1)$, we have $\mathbb{E}ZZ^\top = p_2 I_{p_1}$ and
	\begin{equation*}
	\bbE\left\|ZZ^\top - \bbE ZZ^\top\right\| = \mathbb{E}\left\|ZZ^\top-p_2I_{p_1}\right\| \geq \bbE \left(\left\|ZZ^\top\right\| - p_2\right) = \bbE\|Z\|^2 - p_2.
	\end{equation*}
	Since $\left\|Z\right\|/(\sqrt{p_1}+\sqrt{p_2}) \rightarrow 1$ as $p_1,p_2$ tend to infinity \cite[Theorem 5.31]{vershynin2010introduction},
	\begin{equation*}
	\liminf_{p_1, p_2 \rightarrow \infty} \frac{\bbE\left\|ZZ^\top - \bbE ZZ^\top\right\|}{2\sigma_C\sigma_R+\sigma_C^2} \geq \liminf_{p_1, p_2 \rightarrow \infty} \frac{\bbE\|Z\|^2 - p_2}{2\sqrt{p_1p_2}+p_1} \geq 1.
	\end{equation*}
\end{proof}

\begin{proof}[Proof of Theorem \ref{th:hetero-wishart-lower}]
It suffices to prove the following separate lower bounds to prove this theorem.
\begin{equation}\label{ineq:hetero-lower1}
\sup_{Z \in \mathcal{F}_p(\sigma_\ast, \sigma_C, \sigma_R)} \mathbb{E}\|ZZ^\top - \mathbb{E}ZZ^\top\| \gtrsim \sigma_C^2; 
\end{equation}
\begin{equation}\label{ineq:hetero-lower2}
\sup_{Z \in \mathcal{F}_p(\sigma_\ast, \sigma_C, \sigma_R)} \mathbb{E}\|ZZ^\top - \mathbb{E}ZZ^\top\| \gtrsim \sigma_C\sigma_R; 
\end{equation}
\begin{equation}\label{ineq:hetero-lower3}
\sup_{Z \in \mathcal{F}_p(\sigma_\ast, \sigma_C, \sigma_R)} \mathbb{E}\|ZZ^\top - \mathbb{E}ZZ^\top\| \gtrsim \sigma_R \sigma_\ast \sqrt{\log p} + \sigma_\ast^2\log p.
\end{equation}
\begin{enumerate}
	\item We first set $\sigma_{i1} = \sigma_C/\sqrt{p_1}$; $\sigma_{ij} = 0, j\geq 2$. If $Z_{ij}\sim N(0, \sigma_{ij}^2)$ independently, it is easy to check that $Z\in \mathcal{F}_{p_1, p_2}(\sigma_\ast, \sigma_R, \sigma_C)$. Then $Z$ is zero except the first column. Suppose the first column of $Z$ is $z$, then $ZZ^\top - \mathbb{E}ZZ^\top = zz^\top - \frac{\sigma_C^2}{p_1} I_{p_1}$, 
	\begin{equation*}
	\begin{split}
	\mathbb{E}\|ZZ^\top - \mathbb{E}ZZ^\top\| = &  \mathbb{E}\|zz^\top - \sigma_C^2/p_1\| \geq \mathbb{E}\|zz^\top\| - \sigma_C^2/p_1 = \mathbb{E}\|z\|_2^2 - \sigma_C^2/p_1\\
	\geq & \sigma_C^2(1-1/p_1) \geq c\sigma_C^2,
	\end{split}
	\end{equation*}
	which has shown \eqref{ineq:hetero-lower1}.
	\item Let $k_1 = \lfloor\sigma_C^2/\sigma_{\ast}^2\rfloor$, $k_2 = \lfloor \sigma_R^2/\sigma_{\ast}^2\rfloor$. Construct 
	\begin{equation*}
	\sigma_{ij} = \left\{\begin{array}{ll}
	\sigma_{\ast}, & 1\leq i \leq k_1, 1\leq j \leq k_2;\\
	0, & \text{otherwise}.
	\end{array}\right.
	\end{equation*}
	By such a construction, $Z_{ij} \sim N(0, \sigma_{\ast}^2)$ for $1\leq i \leq k_1, 1\leq j \leq k_2$; $Z_{ij} = 0$ otherwise. Thus,
	\begin{equation*}
	\begin{split}
	& \mathbb{E}\left(Z_{\cdot j}Z_{\cdot j}^\top - \mathbb{E}Z_{\cdot j}Z_{\cdot j}^\top\right)^2 = \mathbb{E}Z_{\cdot j}Z_{\cdot j}^\top Z_{\cdot j}Z_{\cdot j}^\top - \left(\mathbb{E}Z_{\cdot j}Z_{\cdot j}^\top\right)^2 \\
	= &  \mathbb{E} \|Z_{\cdot j}\|_2^2 Z_{\cdot j} Z_{\cdot j}^\top - \sigma_{\ast}^4 I_{k_1} = (k_1+1)\sigma_{\ast}^4.
	\end{split}
	\end{equation*}
	Here, the last equality is due to 
	\begin{equation*}
	\left(\mathbb{E} \|Z_{\cdot j}\|_2^2 Z_{\cdot j} Z_{\cdot j}^\top\right)_{i, i'} = \mathbb{E} \|Z_{\cdot j}\|_2^2 Z_{i, j} Z_{i', j} = \left\{\begin{array}{ll}
	\left(k_1-1+3\right)\sigma_\ast^4, & 1\leq i = i' \leq k_1;\\
	0, & 1\leq i \neq i' \leq k_1.\\
	\end{array}\right.
	\end{equation*}
	Thus, 
	$$\left\|\sum_{j=1}^{k_2}\mathbb{E}\left\{Z_{\cdot j}Z_{\cdot j}^\top - \mathbb{E}Z_{\cdot j}Z_{\cdot j}^\top\right\}^2\right\| = \left\|(k_1+1)k_2\sigma_{\ast}^4 I \right\| = (k_1+1)k_2\sigma_{\ast}^4.$$
	Note that $ZZ^\top - \mathbb{E}ZZ^\top$ can be decomposed as the sum of independent random matrices,
	\begin{equation*}
	ZZ^\top - \mathbb{E}ZZ^\top = \sum_{j=1}^{k_2}\left\{Z_{\cdot j}Z_{\cdot j}^\top - \mathbb{E}Z_{\cdot j}Z_{\cdot j}^\top\right\}.
	\end{equation*}
	We apply the bound for expected norm of random matrices sum \cite{tropp2016expected} and obtain
	\begin{equation*}
	\begin{split}
	\mathbb{E} \|ZZ^\top - \mathbb{E} ZZ^\top \| \gtrsim & \sqrt{(k_1+1)k_2\sigma_{\ast}^4} = \sqrt{(\lfloor \sigma_C^2/\sigma_{\ast}^2\rfloor + 1)\cdot \lfloor \sigma_R^2/\sigma_{\ast}^2\rfloor \cdot \sigma_{\ast}^4}\\
	\geq & \sqrt{(\sigma_C^2/\sigma_{\ast}^2)\cdot \sigma_R^2/(2\sigma_{\ast}^2) \cdot \sigma_{\ast}^4} \quad \text{(since $\sigma_R \geq \sigma_{\ast}$)}\\
	\gtrsim & \sigma_R \sigma_C.
	\end{split}
	\end{equation*}	
	We thus have shown \eqref{ineq:hetero-lower2}.
	\item Set $k_1 = \lfloor\sigma_C^2/\sigma_{\ast}^2\rfloor$, $k_2 = \lfloor \sigma_R^2/\sigma_{\ast}^2\rfloor$, $m = \lfloor(p_1/k_1) \wedge (p_2/k_2)\rfloor$. If $k_2 \geq (\log p)^2$, then $\sigma_R \geq \sigma_* \log p$ and \eqref{ineq:hetero-lower3} can be implied by \eqref{ineq:hetero-lower2}. So we assume $k_2 \leq (\log p)^2$, thus 
	\begin{equation*}
	k_1m \geq k_1 \left(\frac{p_1}{2k_1} \wedge \frac{p_2}{2k_2}\right) \geq \frac{p_1}{2} \wedge \frac{p_2}{2(\log p)^2} \geq \frac{1}{2}\frac{p}{(\log p)^2}
	\end{equation*}
	and $\log(k_1m) \geq c\log p$. Let
	$$(\sigma_{ij}) = \begin{bmatrix}
	B & \\
	& B \\
	& & \ddots \\
	\end{bmatrix} = \diag(\overbrace{B, B,  \ldots, B,  B}^{m},0) \in \mathbb{R}^{p_1\times p_2}, \quad B = \sigma_{\ast} 1_{k_1}1_{k_2}^\top.$$
	Then we can rewrite down $Z$ in rowwise form as
	\begin{equation*}
	Z = \begin{bmatrix}
	\beta_1^\top & 0 & 0\\
	\vdots & \vdots & \vdots \\
	\beta_{k_1}^\top & 0 & 0\\
	0 & \beta_{k_1+1}^\top & 0\\
	\vdots & \vdots &  \vdots \\
	0 & \beta_{2k_1}^\top & 0 \\
	0 & 0 & \ddots
	\end{bmatrix}\in \mathbb{R}^{p_1\times p_2},\quad \beta_1, \ldots, \beta_{k_1m} \in \mathbb{R}^{k_2},\quad \beta_1,\ldots, \beta_{k_1m} \overset{iid}{\sim} N(0, \sigma_{\ast}^2 I_{k_2}).
	\end{equation*}
	By taking a look at the expression of $\left\|ZZ^\top - \bbE ZZ^\top\right\|$, we know
	\begin{equation*}
	\left\|ZZ^\top - \mathbb{E}ZZ^\top\right\|\geq \max_{1\leq j \leq k_1m}\left|\beta_j^\top\beta_j - k_2\sigma_{\ast}^2\right|.
	\end{equation*}
	Note that $\beta_j^\top \beta_j/\sigma_{\ast}^2 \sim \chi^2_{k_2}$. By the lower bound of right-tail of Chi-square distribution (Corollary 3 in \cite{zhang2018non}), we have $\bbP\left(\beta_j^\top \beta_j -k_2 \sigma_\ast^2 \geq \sigma_{\ast}^2x \right) \geq c\exp\left(-C(x\wedge \frac{x^2}{k_2})\right)$. Since
	\begin{equation*}
	\begin{split}
	\bbP\left(\max_j \beta_j^\top\beta_j - k_2\sigma_*^2 > \sigma_*^2 x\right) & = 1 -\bbP\left(\max_j \beta_j^\top\beta_j - k_2\sigma_*^2 \leq \sigma_*^2 x\right) \\
	& = 1 -\prod_{j=1}^{k_1m}\left(1 - \bbP\left(\beta_j^\top\beta_j - k_2\sigma_*^2 \geq \sigma_*^2 x\right)\right) \\
	& \geq 1 - \left(1-c\exp\left(-C\left(x\wedge \frac{x^2}{k_2}\right)\right)\right)^{k_1m},
	\end{split}
	\end{equation*}
	Taking $x = c_1\left(\sqrt{k_2\log(k_1m)} \vee \log(k_1m)\right)$ for some $c_1$ such that $-C\left(x\wedge \frac{x^2}{k_2}\right) \geq -\log(k_1m)$, we get 
	\begin{equation*}
	\begin{split}
	& \left(1-c\exp\left(-C\left(x\wedge \frac{x^2}{k_2}\right)\right)\right)^{k_1m} \leq \left(1 - \frac{c'}{{(k_1m)}}\right)^{k_1m} \leq e^{-c'}.
	\end{split}
	\end{equation*}
	Thus,
	\begin{equation*}
	\begin{split}
	& \mathbb{E}\left\|ZZ^\top -\mathbb{E}ZZ^\top\right\| \geq \mathbb{E}\max_{1\leq j \leq k_1m}\beta_j^\top\beta_j - k_2\sigma_{\ast}^2\\
	\geq & \sup_{x>0} x\sigma_*^2 \cdot \bbP\left(\max_j \beta_j^\top\beta_j - k_2\sigma_*^2 > x\sigma_*^2\right) \\	
	\geq & c_1(1-e^{-c'})\left(\sqrt{k_2\log(k_1m)} \vee \log(k_1m)\right) \\
	\gtrsim &c\sigma_{\ast}^2\left(\sqrt{k_2\log(k_1m)} + \log(k_1m)\right) \gtrsim c\sigma_{\ast}\sigma_R\sqrt{\log p} + c\sigma_*^2\log p. 
	\end{split}
	\end{equation*}

\end{enumerate}
\end{proof}

\subsection{Proofs for non-Gaussian distributions}\label{sec:proof-nonGaussian}
In this section, we collect the proofs of concentration for the non-Gaussian Wishart-type matrix (Lemma \ref{lm:subGaussian-comparison}, Theorem \ref{th:heavier-tail-wishart} and Theorem \ref{th:heter-wishart-bounded}) in Section \ref{sec:non-Gaussian}.

\begin{proof}[Proof of Lemma \ref{lm:subGaussian-comparison}]
Following the notations and proof idea of Lemma \ref{lm:Gaussian-comparison}, we have the same expansion of $\mathbb{E}\tr\left\{(ZZ^\top - \mathbb{E}ZZ^\top)^q\right\}$ as \eqref{eq:take-variance-out}:
	\begin{equation}\label{eq:take-variance-out-subGaussian}
	\begin{split}
	\mathbb{E}\tr\left\{(ZZ^\top - \mathbb{E}ZZ^\top)^q\right\}= & \sum_{\mathbf{c}\in ([p_1]\times [p_2])^q} \prod_{(i, j)\in [p_1]\times [p_2]} \mathbb{E} Z_{ij}^{\alpha_{ij}(\mathbf{c})}\left(Z_{ij}^2 - \sigma_{ij}^2\right)^{\beta_{ij}(\mathbf{c})}\\
	= & \sum_{\mathbf{c}\in ([p_1]\times [p_2])^q} \prod_{k=1}^q \sigma_{u_k, v_k} \sigma_{u_{k+1}, v_k}\prod_{(i, j)\in [p_1]\times [p_2]} \mathbb{E} G_{ij}^{\alpha_{ij}(\mathbf{c})}\left(G_{ij}^2 - 1\right)^{\beta_{ij}(\mathbf{c})},
	\end{split}
	\end{equation}	
where $G_{ij} := Z_{ij}/\sigma_{ij}$. Different from \eqref{eq:take-variance-out}, $E_{ij}$ in \eqref{eq:take-variance-out-subGaussian} may not have $N(0,1)$ distribution. To overcome this difficulty, we introduce the following lemma to bound $\bbE E_{ij}^\alpha(E_{ij}^2 - 1)^\beta$ via a Gaussian analogue.
\begin{lemma}[Gaussian moments]\label{lm:Gaussian-moments}
	Suppose $G \sim N(0, 1)$, $\alpha, \beta$ are non-negative integers, then 
		\begin{equation}\label{ineq:E-G(G^2-1)}
		\left\{\begin{array}{ll}
		(\alpha+2\beta-1)!! \geq \mathbb{E} G^{\alpha} (G^2 - 1)^\beta \geq (\alpha+2\beta-3)!! \cdot (\alpha + \beta-1), & \text{if $\alpha$ is even};\\
		\mathbb{E} G^{\alpha} (G^2 - 1)^\beta = 0, &  \text{if $\alpha$ is odd}.
		\end{array}\right.
		\end{equation}
	Here for odd $k$, $k!! = k(k-2) \cdots 1$. Especially, $(-1)!! = 1, (-3)!! = -1$. More generally, if $Z$ has symmetric distribution and satisfies
	\begin{equation}
	 \Var(Z) = 1, \quad \|Z\|_{\psi_2} = \sup_{q\geq 1}q^{-1/2}(\mathbb{E}|Z|^q)^{1/q} \leq \kappa.
	\end{equation} 
		Then for any integers $\alpha, \beta\geq 0$,
		\begin{equation}\label{ineq:E-G(G^2-1)sub-Gaussian}
		\left|\mathbb{E} Z^\alpha(Z^2 - 1)^\beta\right| \leq (C\kappa)^{\alpha+2\beta} \mathbb{E} G^\alpha(G^2-1)^\beta
		\end{equation}
		for some uniform constant $C>0$.
\end{lemma}
\begin{proof}[Proof of Lemma \ref{lm:Gaussian-moments}]
See \hyperref[sec:proof-technical-lemma]{Appendix}.
\end{proof}
Now, Combining \eqref{eq:take-variance-out-subGaussian} and \eqref{ineq:E-G(G^2-1)sub-Gaussian}, we have
\begin{equation*}
	\begin{split}
	& \mathbb{E}\tr\left\{(ZZ^\top - \mathbb{E}ZZ^\top)^q\right\} \\
	& \qquad \leq  \sum_{\mathbf{c}\in ([p_1]\times [p_2])^q} \prod_{k=1}^q \sigma_{u_k, v_k} \sigma_{u_{k+1}, v_k}\prod_{(i, j)\in [p_1]\times [p_2]} (C\kappa)^{\alpha_{ij}(\mathbf c) + 2\beta_{ij}(\mathbf c)}\mathbb{E} E_{ij}^{\alpha_{ij}(\mathbf{c})}\left(E_{ij}^2 - 1\right)^{\beta_{ij}(\mathbf{c})}\\
	& \qquad = (C\kappa)^{2q} \sum_{\mathbf{c}\in ([p_1]\times [p_2])^q} \prod_{k=1}^q \sigma_{u_k, v_k} \sigma_{u_{k+1}, v_k}\prod_{(i, j)\in [p_1]\times [p_2]} \mathbb{E} G^{\alpha_{ij}(\mathbf{c})}\left(G^2 - 1\right)^{\beta_{ij}(\mathbf{c})}.
	\end{split}
\end{equation*}	
The rest of the proof can similarly proceed as we did in proving Lemma \ref{lm:Gaussian-comparison}.
\end{proof}

\begin{proof}[Proof of Theorem \ref{th:heavier-tail-wishart}] 
Let $b := 2/\alpha \geq 2$ and $E_{ij} := Z_{ij}/\sigma_{ij}$. By definition, we have $\sup_{q} q^{-\frac{b}{2}} (\bbE|E_{ij}|^q)^{1/q} \leq \kappa$. Thus for any $\alpha,\beta \geq 0$,
\begin{equation}
	\begin{split}
		\left|\bbE E_{ij}^\alpha(E_{ij}^2-1)^\beta\right| & \leq \left|\bbE E_{ij}^\alpha (E_{ij}^2-1)1_{\{|E_{ij}|\leq 1\}} + \bbE E_{ij}^\alpha(E_{ij}^2-1)^\beta 1_{\{|E_{ij}|>1\}}\right| \\
		& \leq 1 + \bbE |E_{ij}|^{\alpha + 2\beta} \leq (C\kappa)^{\alpha+2\beta}(\alpha+2\beta)^{\frac{b(\alpha+2\beta)}{2}}.
	\end{split}
\end{equation}
We introduce the following technical lemma.  
\begin{lemma}\label{lm:heavy-tail-comparison}
Let $G, \tilde G$ be independent $N(0,1)$ and let $F_{ij}$ be i.i.d. copy of $G|\tilde G|^{b-1}$. Then,
\begin{equation}\label{ineq:E-G(G^2-1)heavy-tail}
	\bbE E_{ij}^\alpha (E_{ij}^2-1)^\beta \leq (C_b\kappa)^{\alpha+2\beta}\bbE F_{ij}^\alpha(F_{ij}^2-1)^\beta.
\end{equation}
Here $C_b$ is some constant which only depend on $b$.
\end{lemma}
\begin{proof}[Proof of Lemma \ref{lm:heavy-tail-comparison}]
	See \hyperref[sec:proof-technical-lemma]{Appendix}.
\end{proof}

Now let $G_{ij}, \tilde G_{ij}$ be i.i.d. $N(0,1)$ and define $F_{ij} = G_{ij}|\tilde G_{ij}|^{b-1}$. Let $\tilde Z$ be a random matrix with entries $\tilde Z_{ij} = \sigma_{ij}F_{ij}$. Then, by Lemma \ref{lm:heavy-tail-comparison} and the similar proof in Lemma \ref{lm:subGaussian-comparison}, we have
\begin{equation*}
	\bbE \tr\left\{(ZZ^\top - \bbE ZZ^\top)^q\right\} \leq (C_b\kappa)^{2q} \bbE \tr\left\{(\tilde Z \tilde Z - \bbE \tilde Z \tilde Z^\top)^q\right\}.
\end{equation*}
Thus,
\begin{equation}
	\begin{split}
		\bbE\left\|ZZ^\top - \bbE ZZ^\top\right\| & \leq  \left(\bbE \tr\left\{(ZZ^\top - \bbE ZZ^\top)^{2q}\right\}\right)^{1/2q} \\
		& \leq (C_b\kappa)^2 \left(\bbE \tr\left\{(\tilde Z\tilde Z^\top - \bbE \tilde Z \tilde Z^\top)^{2q}\right\}\right)^{1/2q}.
	\end{split}
\end{equation}
Let $q = \lceil \log (p_1 \wedge p_2) \rceil$, now it suffices to upper bound $\left(\bbE\left\|\tilde Z\tilde Z^\top - \bbE Z\tilde Z^\top\right\|^{2q}\right)^{1/2q}$. We define $\tilde \sigma_C^2 = \max_j\sum_{i=1}^{p_1} \sigma_{ij}^2|\tilde{G}_{ij}|^{2b-2}$, $\tilde{\sigma}_R^2 = \max_i\sum_{j=1}^{p_2} \sigma_{ij}^2|\tilde{G}_{ij}|^{2b-2}$ and $\tilde \sigma_* = \max_i \sigma_{ij}|\tilde{G}_{ij}|^{b-1}$ and apply Theorem \ref{th:wishart-Gaussian} conditionally on $\tilde{G}$:
\begin{equation*}
	\begin{split}
		&\bbE\left[\tr\left\{\left(\tilde Z \tilde Z^\top - \bbE \tilde Z\tilde Z^\top\right)^{2q}\right\}\left|\tilde G\right.\right] \\
		& \qquad \leq C^{2q}\left(\tilde\sigma_C^2 + \tilde\sigma_C\tilde\sigma_R + \sigma_C\sigma_*\sqrt{\log(p_1\wedge p_2)} + \sigma_*^2\log(p_1\wedge p_2)\right)^{2q}.
	\end{split}
\end{equation*}
Then,
\begin{equation}\label{ineq:heavy-tail-decomposion}
	\begin{split}
		& \left(\bbE\tr\left\{\left(\tilde Z \tilde Z^\top - \bbE \tilde Z\tilde Z^\top\right)^{2q}\right\}\right)^{1/2q} \\
		& \qquad \leq C\left(\left\|\tilde\sigma_C^2\right\|_{2q} + \left\|\tilde\sigma_C\sigma_R\right\|_{2q} + \left\|\tilde\sigma_R\tilde\sigma_*\right\|_{2q}\sqrt{\log(p_1\wedge p_2)} + \left\|\tilde\sigma_*^2\right\|_{2q}\log(p_1\wedge p_2)\right).
	\end{split}
\end{equation}
Here $\|X\|_{2q}:= (\bbE |X|^{2q})^{1/2q}$ is the $\ell_{2q}$-norm of random variable $X$. Now we bound $\left\|\tilde{\sigma}_*^2\right\|_{2q}$, $\left\|\tilde{\sigma}_R^2\right\|_{2q}$ and $\left\|\tilde\sigma_C^2\right\|_{2q}$ separately. 
\begin{itemize}
	\item $\left\|\tilde\sigma_*^2\right\|_{2q}$. For any $a>0$, since 
	\begin{equation*}
		\bbP\left(\max_{i,j} |\tilde G_{ij}| > t\right) \leq 2\exp\left(-\frac{t^2}{2}+\log(p_1p_2)\right) \leq 2\exp\left(-\frac{t^2}{4}\right), \quad
		 \forall t > 2\sqrt{\log(p_1p_2)},
	\end{equation*}
	integration yields
\begin{equation*}
	\begin{split}
		\bbE \max_{i,j}|\tilde G_{ij}|^{a} &= \int_{0}^\infty \bbP\left(\max_{i,j} |\tilde G_{ij}| > t^{1/a}\right) dt \leq \left(2\sqrt{\log(p_1p_2)}\right)^a + \int_{0}^\infty 2e^{-\frac{t^{2/a}}{4}} dt \\
		&= \left(4\log(p_1p_2)\right)^{a/2} + 4a\Gamma\left(\frac{a}{2}\right).
	\end{split}
\end{equation*}
Then it follows that
	\begin{equation}\label{ineq:heavy-tail-sigmaask}
	\begin{split}
	\left\|\tilde\sigma_*^2\right\|_{2q} & \leq \sigma_*^2\left(\bbE \max_{i,j}|\tilde G_{ij}|^{4(b-1)q}\right)^{1/2q} \\
	& \leq \sigma_*^2 \left((4\log(p_1p_2))^{2(b-1)q} + 16(b-1)q\Gamma\left(2(b-1)q\right)\right)^{1/2q}  \\
	& \lesssim \sigma_*^2 \left(\log(p_1p_2)^{b-1} + q^{b-1}\right)\lesssim \sigma_*^2 \log^{b-1}(p_1 \vee p_2).
	\end{split}
	\end{equation}
	\item $\left\|\tilde \sigma_C^2\right\|_{2q}$ and $\left\|\tilde \sigma_R^2\right\|$. By the moment bound of supremum of empirical process \cite[Theorem 11]{boucheron2005moment}, 
	\begin{equation}\label{ineq:heavy-tail-sigmaC-1}
		\begin{split}
			\left\|\tilde \sigma_C^2 \right\|_{2q} & = \left\|\max_j \sum_{i=1}^{p_1}\sigma_{ij}^2|\tilde G_{ij}|^{2b-2}\right\|_{2q} \\
			& \lesssim \bbE\max_{1\leq j\leq p_2}\sum_{i=1}^{p_1}\sigma_{ij}^2|\tilde G_{ij}|^{2b-2} + q\left\|\tilde \sigma_*^2\right\|_{2q}\\
			& \leq \bbE\max_j\sum_{i=1}^{p_1}\left(\sigma_{ij}^2|\tilde{G}_{ij}|^{2b-2} - \bbE\sigma_{ij}^2|\tilde G_{ij}|^{2b-2}\right) +\sigma_C^2 + q\left\|\tilde\sigma_*^2\right\|_{2q}.
		\end{split}
	\end{equation}
	Denote $Y_j = \sum_{i=1}^{p_1}\left(\sigma_{ij}^2|\tilde{G}_{ij}|^{2b-2} - \bbE\sigma_{ij}^2|\tilde G_{ij}|^{2b-2}\right)$, it suffices to bound $\bbE\max_{j}Y_j$. To this end, we introduce the following Generalized Bernstein-Orlicz norm defined in \cite{kuchibhotla2018moving}. For a random variable $X$, let 
$$
\left\|X\right\|_{\Psi_{\alpha,L}}:=\inf\left\{\eta>0: \bbE[\Psi_{\alpha,L}(|X|/\eta)]\leq 1\right\}
$$ be the $\Psi_{\alpha,L}$-norm where $\Psi_{\alpha, L}$ is defined via its inverse function
$$
\Psi_{\alpha,L}^{-1}(t) := \sqrt{\log(1+t)} + L(\log(1+t))^{1/\alpha},\qquad \forall t \geq 0.
$$
Now fix $j\in [p_2]$ and let $\alpha = 1/(b-1)$ and $L=\frac{4^{b-1}\sigma_*^2}{\sqrt{2}\sqrt{\sum_{i=1}^{p_1}\sigma_{ij}^4}}$. By \cite[Theorem 3.1]{kuchibhotla2018moving},
$$
		\left\|Y_j\right\|_{\Psi_{\alpha,L}} \leq C\sqrt{\sum_{i=1}^{p_1}\sigma_{ij}^4};
$$
$$
\bbP\left(|Y_j| \geq C\sqrt{\sum_{i=1}^{p_1}\sigma_{ij}^4}\left\{\sqrt{t}+Lt^{1/\alpha}\right\}\right) \leq 2\exp(-t),\qquad t\geq 0.
$$
This yields
\begin{equation*}
	\begin{split}
	\bbP\left(|Y_j| \geq C\left\{\sigma_C\sigma_*\sqrt{t}+\sigma_*^2t^{b-1}\right\}\right) \leq 2\exp(-t),\qquad t\geq 0,
	\end{split}
\end{equation*}
which can be rewritten as 
\begin{equation*}
	\begin{split}
	\bbP\left(|Y_j| \geq t\right) \leq 2\exp\left(-c\left(\frac{t^2}{\sigma_C^2\sigma_*^2} \wedge \left(\frac{t}{\sigma_*^2}\right)^{1/(b-1)}\right)\right),\qquad t\geq 0.
	\end{split}
\end{equation*}
Applying union bound, we get
\begin{equation*}
	\begin{split}
	\bbP\left(\max_j |Y_j| \geq t \right) \leq 2\exp\left(\log p_2-c\left(\frac{t^2}{\sigma_C^2\sigma_*^2} \wedge \left(\frac{t}{\sigma_*^2}\right)^{1/(b-1)}\right)\right),\qquad t\geq 0.
	\end{split}
\end{equation*}

Now it follows that
\begin{equation*}
	\begin{split}
		\bbE\max_j Y_j & \leq \bbE\max_j |Y_j| = \int_0^\infty  \bbP\left(\max_j|Y_j|>t\right) dt \\
		& \leq C\left(\sigma_C\sigma_*\sqrt{\log p_2} + \sigma_*^2\log^{b-1}(p_2)\right) + \int_{C\left(\sigma_C\sigma_*\sqrt{\log p_2} + \sigma_*^2\log^{b-1}(p_2)\right)}^\infty \bbP\left(\max_j|Y_j|>t\right)dt \\
		& \leq C\left(\sigma_C\sigma_*\sqrt{\log p_2} + \sigma_*^2\log^{b-1}(p_2)\right) + \int_{0}^\infty \left(\exp\left(-c\frac{t^2}{\sigma_C\sigma_*}\right) + \exp\left(-c\frac{t^{1/(b-1)}}{\sigma_*^{2/(b-1)}}\right)\right)dt \\
		& \lesssim \sigma_C\sigma_*\sqrt{\log p_2} + \sigma_*^2\log^{b-1}(p_2) \\
		& \lesssim \sigma_C^2 + \sigma_*^2\log^{b-1}(p_2). \\
	\end{split}
\end{equation*}
Combining with \eqref{ineq:heavy-tail-sigmaC-1}, we obtained
\begin{equation}\label{ineq:heavy-tail-sigmaC}
	\left\|\tilde\sigma_C^2\right\|_{2q} \lesssim \sigma_C^2 + \sigma_*^2\log^{b-1}(p_1\vee p_2)\log(p_1\wedge p_2).
\end{equation}
Similarly we can obtain
\begin{equation}\label{ineq:heavy-tail-sigmaR}
		\begin{split}
			\left\|\tilde\sigma_R\right\|_{2q} \lesssim \sigma_R^2 + \sigma_*^2\log^{b-1}(p_1\vee p_2)\log(p_1\wedge p_2).
		\end{split}
\end{equation}

\end{itemize}
Combining \eqref{ineq:heavy-tail-decomposion}, \eqref{ineq:heavy-tail-sigmaask}, \eqref{ineq:heavy-tail-sigmaC}, \eqref{ineq:heavy-tail-sigmaR} and applying Cauchy-Schwarz inequality, we obtain
\begin{equation*}
	\begin{split}
	\bbE\left\|ZZ^\top - \bbE ZZ^\top\right\| & \lesssim \sigma_C^2 + \sigma_R\sigma_C  + \sigma_R\sigma_*\log^{(b-1)/2}(p_1\vee p_2)\sqrt{\log(p_1\wedge p_2)} \\
	& \qquad + \sigma_*^2\log^{b-1}(p_1\vee p_2)\log(p_1\wedge p_2).
	\end{split}
\end{equation*}
This completes the proof.  
\end{proof}

\begin{proof}[Proof of Theorem \ref{th:heter-wishart-bounded}]
We first prove the following comparison Lemma. 
\begin{lemma}\label{lm:Z-M-comparison-uniform}
	Suppose $Z$ is a $p_1$-by-$p_2$ random matrix with independent entries satisfying $\mathbb{E}Z_{ij} = 0, \Var(Z_{ij}) = \sigma_{ij}^2$, $|Z| \leq 1$. $H$ is an $m_1$-by-$m_2$ dimensional matrix with i.i.d. standard Gaussian entries. When $q\geq 1$, $m_1 = \lceil\sigma_C^2\rceil+q-1$, $m_2 = \lceil\sigma_R^2\rceil+q-1$, we have
	\begin{equation*}
	\mathbb{E} \tr\left\{(ZZ^\top - \mathbb{E}ZZ^\top)^{q}\right\} \leq \left(\frac{p_1}{m_1}\wedge \frac{p_2}{m_2}\right)\mathbb{E} \tr\left\{(HH^\top - \mathbb{E}HH^\top)^{q}\right\}.
	\end{equation*}
\end{lemma}  
\begin{proof}
Recall $Z\in \mathbb{R}^{p_1\times p_2}$, $|Z|\leq 1$ almost surely, $\mathbb{E} Z_{ij}=0$, $\Var(Z_{ij}) = \sigma_{ij}^2$. Similarly as the proof of Lemma \ref{lm:Gaussian-comparison}, let $\mathbf{c} = (u_1,v_1,\ldots, u_q, v_q)\in ([p_1\times [p_2]])^q$ be the cycle of length $2q$ on bipartite graph $[p_1]\to [p_2]$, $\alpha_{ij}(\mathbf{c})$ and $\beta_{ij}(\mathbf{c})$ be defined as \eqref{eq:def-alpha-beta-L}. We similarly have the following expansion,
\begin{equation*}
\begin{split}
& \mathbb{E}\left\{(ZZ^\top - \mathbb{E}ZZ^\top)^q\right\} = \sum_{u_1,\ldots, u_q \in [p_1]} \mathbb{E} \prod_{j=1}^q (ZZ^\top - \mathbb{E}ZZ^\top)_{u_j, u_{j+1}}\\
= & \mathbb{E} \sum_{u_1,\ldots, u_q \in [p_1]} \prod_{j=1}^q \left\{\sum_{v_j \in [p_2]} Z_{u_j, v_j}(Z^\top)_{v_j, u_{j+1}} - 1_{\{u_j = u_{j+1}\}}\sum_{v_j \in [p_2]}\mathbb{E} Z_{u_j, v_j}^2 \right\}\\
= & \sum_{\substack{u_1, \ldots, u_{q} \in [p_1]\\v_1,\ldots, v_q \in [p_2]}} \mathbb{E}\prod_{j=1}^q \left(Z_{u_j, v_j} Z_{u_{j+1}, v_j} - \sigma_{u_j, v_j}^21_{\{u_j = u_{j+1}\}}\right)\\
= & \sum_{\mathbf{c}\in ([p_1]\times [p_2])^q} \prod_{(i, j)\in [p_1]\times [p_2]} \mathbb{E} Z_{ij}^{\alpha_{ij}(\mathbf{c})}\left(Z_{ij}^2 - \sigma_{ij}^2\right)^{\beta_{ij}(\mathbf{c})}.
\end{split}
\end{equation*}
Since $Z_{ij}$ is symmetric distributed and $\mathbb{E}Z^2_{ij} = \sigma_{ij}^2$, we have 
$$\mathbb{E} Z_{ij}^\alpha\left(Z_{ij}^2 - \sigma_{ij}^2\right)^\beta = 0,\quad \text{if $\alpha$ is odd or $\{\alpha = 0, \beta = 1\}$}.$$
For any $(i,j) \in [p_1]\times [p_2]$, we shall note that $0\leq Z_{ij}^2\leq 1$ and $|Z_{ij}^2 - \sigma_{ij}^2|\leq 1$. If $\alpha\geq 2$ and $\alpha$ is even,
\begin{equation*}
\mathbb{E}\left|Z_{ij}^{\alpha_{ij}(\mathbf{c})}  (Z_{ij}^2 - \sigma_{ij}^2)^{\beta_{ij}(\mathbf{c})}\right| = \mathbb{E} |Z_{ij}^2|\cdot |Z_{ij}^{\alpha_{ij}(\mathbf{c})-2}  (Z_{ij}^2 - \sigma_{ij}^2)^{\beta_{ij}(\mathbf{c})}| \leq \mathbb{E}Z_{ij}^2 =  \sigma_{ij}^2;
\end{equation*}
if $\alpha \geq 0$, $\beta \geq 2$, one has
\begin{equation*}
\begin{split} 
& \mathbb{E}\left|Z_{ij}^{\alpha_{ij}(\mathbf{c})}  (Z_{ij}^2 - \sigma_{ij}^2)^{\beta_{ij}(\mathbf{c})}\right| \leq  \mathbb{E}(Z_{ij}^2 - \sigma_{ij}^2)^2\cdot \left|Z_{ij}^{\alpha_{ij}(\mathbf{c})}(Z_{ij}^2 - \sigma_{ij}^2)^{\beta_{ij}(\mathbf{c})-2}\right|\\
\leq & \mathbb{E} Z_{ij}^4 - \sigma_{ij}^4 \leq \mathbb{E}Z_{ij}^2 \cdot \|Z_{ij}\|_\infty^2 - \sigma_{ij}^4 = \sigma_{ij}^2 - \sigma_{ij}^4 \leq \sigma_{ij}^2.
\end{split}
\end{equation*}
Therefore, for any $\alpha, \beta \geq 0$, we have
\begin{equation*}
\begin{split}
\mathbb{E}Z_{ij}^{\alpha} (Z_{ij}^2 - \sigma_{ij}^2)^\beta\left\{\begin{array}{ll}
\leq \sigma_{ij}^2 \cdot \mathbb{E}G^{\alpha} \left(G^2 - 1\right)^{\beta}, & \alpha \text{ is even and } (\alpha,\beta) \neq (0,0);\\
= 1, & \alpha = 0, \beta = 0;\\
= 0, & \alpha \text{ is odd}.\\
\end{array}\right.
\end{split}
\end{equation*}
Here, $G\sim N(0, 1)$. Thus,
\begin{equation*}
\begin{split}
& \mathbb{E}\left\{(ZZ^\top - \mathbb{E}ZZ^\top)^q\right\} \\
\leq & \sum_{\mathbf{c}\in ([p_1]\times [p_2])^q} \prod_{(i, j)\in [p_1]\times [p_2]} \sigma_{ij}^2 1_{\left\{(\alpha_{ij}(\mathbf{c}), \beta_{ij}(\mathbf{c}))\neq (0, 0)\right\}}\mathbb{E} G^{\alpha_{ij}(\mathbf{c})}\left(G^2 - 1\right)^{\beta_{ij}(\mathbf{c})}.\\
\end{split}
\end{equation*}
Let $\mathbf{s}$ be the shape of any loop $\mathbf{c} \in ([p_1]\times [p_2])^q$, $m_L(\mathbf{s})$ and $m_R(\mathbf{s})$ be the number of distinct left and right nodes respectively visited by any $\mathbf{c}$ with shape $\mathbf{s}$; $m_{\alpha, \beta}(\mathbf{s}) = m_{\alpha, \beta}(\mathbf{c})$ is defined as \eqref{eq:def-m-alpha-beta}). Then,
\begin{equation*}
\begin{split}
& \mathbb{E}\left\{(ZZ^\top - \mathbb{E}ZZ^\top)^q\right\}\\
\leq & \sum_{\mathbf{c}\in ([p_1]\times [p_2])^q} \prod_{(i, j)\in [p_1]\times [p_2]} \sigma_{ij}^2 1_{\left\{(\alpha_{ij}(\mathbf{c}), \beta_{ij}(\mathbf{c}))\neq (0, 0)\right\}}\mathbb{E} G^{\alpha_{ij}(\mathbf{c})}\left(G^2 - 1\right)^{\beta_{ij}(\mathbf{c})}\\
= & \sum_{\mathbf{s}} \sum_{\substack{\mathbf{c}:\\\mathbf{c} \text{ has shape }\mathbf{s}}} \cdot \prod_{\substack{(i, j)\in [p_1]\times [p_2]\\\mathbf{c} \text{ pass $(i, j)$ for positive even times}}} \sigma_{ij}^2 \cdot \prod_{\substack{\alpha, \beta\geq 0\\\alpha \text{ is even}}}\left\{G^\alpha(G^2-1)^\beta\right\}^{m_{\alpha, \beta}(\mathbf{s})} \\
\leq & \sum_{\mathbf{s}} \left(p_1 \sigma_C^{2m_L(\mathbf{s})-2}\sigma_R^{2m_R(\mathbf{s})} \wedge p_2 \sigma_C^{2m_L(\mathbf{s})}\sigma_R^{2m_R(\mathbf{s})-2} \right) \cdot \prod_{\substack{\alpha, \beta\geq 0\\\alpha \text{ is even}}}\left\{G^\alpha(G^2-1)^\beta\right\}^{m_{\alpha, \beta}(\mathbf{s})}.
\end{split}
\end{equation*}
On the other hand, we have
\begin{equation*}
\begin{split}
 & \mathbb{E}\tr\left((HH^\top - \mathbb{E}HH^\top)^q\right) = \mathbb{E}\tr\left((HH^\top - m_2 I_{m_1})^q\right)\\
= & \sum_{\mathbf{s}} m_1\cdots (m_1 - m_L(\mathbf{s})+1)\cdot m_2\cdots (m_2 - m_R(\mathbf{s})+1)\cdot\prod_{\alpha, \beta \geq 0} \mathbb{E} \left\{G^\alpha (G^2 - 1)^\beta\right\}^{m_{\alpha, \beta}(\mathbf{s})}.
\end{split}
\end{equation*}
Provided that $m_1 = \lceil\sigma_C^2\vee 1\rceil + q-1$ and $m_2 = \lceil\sigma_R^2\vee 1\rceil + q-1$, we have
$$\sigma_C^{2m_L(\mathbf{s}) - 2} \leq \frac{m_1(m_1-1)\cdots (m_1-m_L(\mathbf{s})+1)}{m_1}, \quad \sigma_R^{2m_R(\mathbf{s})} \leq m_2(m_1-2)\cdots (m_2-m_L(\mathbf{s})+1);$$
$$\sigma_C^{2m_L(\mathbf{s})} \leq m_1(m_1-1)\cdots (m_1-m_L(\mathbf{s})+1), \quad \sigma_R^{2m_R(\mathbf{s}) - 2} \leq \frac{m_2(m_1-2)\cdots (m_2-m_L(\mathbf{s})+1)}{m_2}.$$
Thus
\begin{equation*}
\mathbb{E}\tr\left\{(ZZ^\top - \mathbb{E}ZZ^\top)^q \right\} \leq \left(\frac{p_1}{m_1} \wedge \frac{p_2}{m_2}\right) \cdot \mathbb{E}\tr\left\{(HH^\top - \mathbb{E}HH^\top)^q\right\},
\end{equation*}
which has finished the proof of this lemma.
\end{proof}
Assume $B=1$ without loss of generality. With Lemma \ref{lm:Z-M-comparison-uniform} and Lemma \ref{lm:iid-Gaussian-moment}, the proof of Theorem ref{th:heter-wishart-bounded} is the same as Theorem \ref{th:wishart-Gaussian}.
\end{proof}

\subsection{Proof for tail bounds}\label{sec:proof-higher-moment}
\begin{proof}[Proof of Theorem \ref{th:higher-moment-tail-bound}]
Without loss of generality, we assume $\sigma_\ast = 1$. Let $q \geq 2$ be an even integer. Let
$$m_1 = \lceil \sigma_C^2\rceil + qb-1, \quad m_2 = \lceil \sigma_R^2 \rceil + qb-1.$$
By Lemmas \ref{lm:Gaussian-comparison} and \ref{lm:iid-Gaussian-moment},
\begin{equation*}
\begin{split}
& \mathbb{E}\tr\left\{\left(ZZ^\top -\mathbb{E}ZZ^\top\right)^{qb}\right\} \leq \left(\frac{p_1}{m_1}\wedge \frac{p_2}{m_2}\right)\mathbb{E}\tr\left((HH^\top - \mathbb{E}HH^\top)^{qb}\right)\\
\leq & \left(\frac{p_1}{m_1}\wedge \frac{p_2}{m_2}\right) (m_1\wedge m_2)\left(2\sqrt{m_1m_2}+m_1+4(\sqrt{m_1}+\sqrt{m_2})\sqrt{qb}+2qb\right)^{qb}\\
\leq & \left(p_1\wedge p_2\right)\left(2\sqrt{m_1m_2}+m_1+4(\sqrt{m_1}+\sqrt{m_2})\sqrt{qb}+2qb\right)^{qb}.
\end{split}
\end{equation*}
Thus,
\begin{equation*}
\begin{split}
& \mathbb{E}\left\|ZZ^\top -\mathbb{E}ZZ^\top\right\|^b \leq \left(\mathbb{E}\tr\left(ZZ^\top - \mathbb{E}ZZ^\top\right)^{qb}\right)^{1/q}\\
\leq & \left(p_1\wedge p_2\right)^{1/q}\left(2\sqrt{m_1m_2}+m_1+4(\sqrt{m_1}+\sqrt{m_2})\sqrt{qb}+2qb\right)^{b}\\
= & \left\{\left(p_1\wedge p_2\right)^{1/(qb)}\left(2\sqrt{m_1m_2}+m_1+4(\sqrt{m_1}+\sqrt{m_2})\sqrt{qb}+2qb\right)\right\}^b\\
\leq & \left\{C(p_1\wedge p_2)^{1/(qb)}\left(\sigma_R\sigma_C+\sigma_C^2+(\sigma_R+\sigma_C)\sqrt{qb} + qb\right)\right\}^b.
\end{split}
\end{equation*}
We set $q = 2 \lceil\log(p_1\wedge p_2)/b\rceil$ and consider the following two cases:
\begin{enumerate}
	\item If $ b\geq \log(p_1\wedge p_2)$, we have $q = 2$ and
	\begin{equation*}
	\begin{split}
	\mathbb{E}\left\|ZZ^\top -\mathbb{E}ZZ^\top\right\|^b \leq & \left\{C(p_1\wedge p_2)^{1/(2b)}\left(\sigma_R\sigma_C+\sigma_C^2+(\sigma_R+\sigma_C)\sqrt{b} + b\right)\right\}^b\\
	\leq & \left\{C\left((\sigma_C+\sigma_R+\sqrt{b})^2-\sigma_R^2\right)\right\}^b.
	\end{split}
	\end{equation*}
	\item If $b < \log(p_1\wedge p_2)$, we have 
	$$2\log(p_1\wedge p_2)/b\leq q = 2\lceil\log(p_1\wedge p_2)/b\rceil \leq 2\left(\log(p_1\wedge p_2)/b +1\right) \leq 4\log(p_1\wedge p_2)/b.$$
	Then,
	\begin{equation*}
	\begin{split}
	& \mathbb{E}\left\|ZZ^\top -\mathbb{E}ZZ^\top\right\|^b\\ 
	\leq & \left\{C(p_1\wedge p_2)^{1/(2\log(p_1\wedge p_2))}\left(\sigma_R\sigma_C+\sigma_C^2+(\sigma_R+\sigma_C)\sqrt{4\log(p_1\wedge p_2)} + 4\log(p_1\wedge p_2)\right)\right\}^b\\
	\leq & \left\{C\left((\sigma_C+\sigma_R+\sqrt{\log(p_1\wedge p_2)})^2-\sigma_C^2\right)\right\}^b.
	\end{split}
	\end{equation*}
\end{enumerate}
In summary, there exists a uniform constant $C_0>0$ such that
\begin{equation*}
\mathbb{E}\left\|ZZ^\top - \mathbb{E}ZZ^\top\right\|^b \leq \left\{C_0\left((\sigma_C+\sigma_R+\sqrt{b\vee \log(p_1\wedge p_2)})^2-\sigma_C^2\right)\right\}^b.
\end{equation*}
In fact, the statement holds for all $b>0$ including non-integers.

Next we consider the tail bound inequality for $\|ZZ^\top - \mathbb{E}ZZ^\top\|$. Let $C_1$ be a to-be-specified constant. By Markov inequality,
\begin{equation*}
\begin{split}
& \bbP\left(\left\|ZZ^\top - \mathbb{E}ZZ^\top\right\| \geq C_1\left((\sigma_C+\sigma_R+\sqrt{\log(p_1\wedge p_2)} + x)^2-\sigma_C^2\right)\right) \\
\leq & \frac{\mathbb{E}\left\|ZZ^\top - \mathbb{E}ZZ^\top\right\|^b}{\left\{C_1\left((\sigma_C+\sigma_R+\sqrt{\log(p_1\wedge p_2)} + x)^2-\sigma_C^2\right)\right\}^b}\\
\leq & \left\{\frac{C_0\left((\sigma_C+\sigma_R+\sqrt{b\vee \log(p_1\wedge p_2)})^2-\sigma_C^2\right)}{C_1\left((\sigma_C+\sigma_R+\sqrt{\log(p_1\wedge p_2)} + x)^2-\sigma_C^2\right)}\right\}^b.
\end{split}
\end{equation*}
We set $b = x^2$, $C_1 = eC_0$, we have
\begin{equation*}
\begin{split}
& \bbP\left(\left\|ZZ^\top - \mathbb{E}ZZ^\top\right\| \geq C_1\left((\sigma_C+\sigma_R+\sqrt{\log(p_1\wedge p_2)} + x)^2-\sigma_C^2\right)\right) \\
\leq & \left\{\frac{C_0\left((\sigma_C+\sigma_R+\sqrt{\log(p_1\wedge p_2)} + \sqrt{b})^2-\sigma_C^2\right)}{C_1\left((\sigma_C+\sigma_R+\sqrt{\log(p_1\wedge p_2)} + x)^2-\sigma_C^2\right)}\right\}^b = \exp(-x^2).
\end{split}
\end{equation*}
Therefore, we have finished the proof of this theorem.
\end{proof}

\subsection{Proofs for Section \ref{sec:homoskedastic rows}}\label{sec:proof-homo-row}

\begin{proof}[Proof of Lemma \ref{lm:ZZtop-tildeZZtop}]
The proof of this lemma relies on a more careful counting scheme for each cycle. For convenience, we define 
$$\tilde{\sigma}_{ij}^2 = \Var(\tilde{G}_{ij}) = \left\{\begin{array}{ll}
\sigma_{ij}^2, & 1\leq i \leq p_1-2, 1\leq j \leq p_2;\\
\sigma_{p_1-1,j}^2+\sigma_{p_1,j}^2, & i = p-1, 1\leq j \leq p_2.
\end{array}\right.$$
$$(G_0)_{ij} = G_{ij}/\sigma_{ij}, 1\leq i \leq p_2, 1\leq j \leq p_1; \quad (\tilde{G}_0)_{ij} = \tilde{G}_{ij}/\tilde{\sigma}_{ij}, 1\leq i \leq p_1-1, 1\leq j \leq p_2$$ 
as the variances and standardizations of each entry of $G$ and $\tilde{G}$. Since the proof is lengthy, we divide into steps for a better presentation.

\begin{itemize}
	\item[Step 1] In this step, we consider the expansions for both $\mathbb{E}\tr(GG^\top - \mathbb{E}GG^\top)^q$ and $\mathbb{E}\tr(\tilde{G}\tilde{G}^\top - \mathbb{E}\tilde{G}\tilde{G}^\top)^q$,
	\begin{equation}
	\begin{split}
	& \mathbb{E}\tr\left\{(GG^\top - \mathbb{E}GG^\top)^{q}\right\} =  \mathbb{E}\sum_{\substack{u_1, \ldots, u_{q} \in [p_1]}}  \prod_{k=1}^q\left(GG^\top - \mathbb{E}GG^\top\right)_{u_k, u_{k+1}} \\
	= & \sum_{\substack{u_1, \ldots, u_{q} \in [p_1]\\v_1,\ldots, v_q \in [p_2]}} \mathbb{E}\prod_{k=1}^q \sigma_{u_k, v_k}\sigma_{u_{k+1}, v_{k}}\left((G_0)_{u_k, v_k} (G_0)_{u_{k+1}, v_k} - 1_{\{u_k = u_{k+1}\}}\right)\\
	= & \sum_{\Omega\subseteq [q]} \sum_{u_{\Omega^c} \in [p_1-2]} \sum_{v_1,\ldots, v_q\in [p_2]} \\
	& \qquad \qquad \cdot \left\{\sum_{u_{\Omega} \in \{p_1-1, p_1\}} \mathbb{E}\prod_{k=1}^q \sigma_{u_k, v_k} \sigma_{u_{k+1}, v_k}\left((G_0)_{u_k, v_k} (G_0)_{u_{k+1}, v_k} - 1_{\{u_k = u_{k+1}\}}\right)\right\}.
	\end{split}
	\end{equation}
	Here $u_{q+1}:=u_1$. Similarly,
	\begin{equation*}
	\begin{split}
	& \mathbb{E}\tr\left\{(\tilde{G}\tilde{G}^\top - \mathbb{E}\tilde{G}\tilde{G}^\top)^{q}\right\} =  \mathbb{E} \sum_{\substack{u_1, \ldots, u_{q} \in [p_1-1]}}  \prod_{k=1}^q\left(\tilde{G}\tilde{G}^\top - \mathbb{E}\tilde{G}\tilde{G}^\top\right)_{u_k, u_{k+1}} \\
	= & \sum_{\Omega\subseteq [q]} \sum_{u_{\Omega^c} \in [p_1-2]} \sum_{v_1,\ldots, v_q\in [p_2]} \\
	& \qquad \qquad \cdot \left\{\sum_{u_{\Omega} = p_1-1} \mathbb{E}\prod_{k=1}^q \tilde{\sigma}_{u_k, v_k} \tilde{\sigma}_{u_{k+1}, v_k}\left((\tilde{G}_0)_{u_k, v_k} (\tilde{G}_0)_{u_{k+1}, v_k} - 1_{\{u_k = u_{k+1}\}}\right)\right\}.
	\end{split}
	\end{equation*}
	Thus, in order prove this lemma, we only need show for any fixed $v_1,\ldots, v_q \in [p_2]$, $\Omega\subseteq [q], u_{\Omega^c} \in [p_1-2]$, one has
	\begin{equation}\label{ineq:ZZtop-tildeZZtop-1}
	\begin{split}
	& \sum_{u_{\Omega} \in \{p_1-1, p_2\}} \mathbb{E}\prod_{k=1}^q \sigma_{u_k, v_k}\sigma_{u_{k+1}, v_k}\left((G_0)_{u_k, v_k} (G_0)_{u_{k+1}, v_k} - 1_{\{u_k = u_{k+1}\}}\right) \\
	\leq & \mathbb{E}\prod_{k=1}^q \tilde{\sigma}_{\tilde{u}_k, v_k}\tilde{\sigma}_{\tilde{u}_{k+1}, v_k}\left((\tilde{G}_0)_{\tilde{u}_k, v_k} (\tilde{G}_0)_{\tilde{u}_{k+1}, v_k} - 1_{\{\tilde{u}_k = \tilde{u}_{k+1}\}}\right).
	\end{split}
	\end{equation}
	Here,
	\begin{equation}\label{eq:tilde-u-k}
	\tilde{u}_k  = u_k \in [p_1-2], \text{if } k\in \Omega^c; \quad \tilde{u}_k = p_1-1,  \text{if } k \in \Omega.
	\end{equation}
	\item[Step 2] To prove \eqref{ineq:ZZtop-tildeZZtop-1}, we shall first recall that the definition of $u_1,\ldots, u_q, v_1,\ldots, v_q$ are cyclic, i.e., $u_1 = u_{q+1}$, we also denote $v_0 = v_q$. Thus,
	\begin{equation}\label{eq:ZZtop-tildeZZtop-3}
	\begin{split}
	& \sum_{u_{\Omega} \in \{p_1-1, p_1\}} \prod_{k=1}^q\sigma_{u_k, v_k}\sigma_{u_{k+1}, v_k} = \sum_{u_{\Omega} \in \{p_1-1, p_1\}} \prod_{k=1}^q\sigma_{u_k, v_k}\sigma_{u_k, v_{k-1}}\\
	= & \prod_{k\in \Omega^c} \sigma_{u_k, v_k}\sigma_{u_k, v_{k-1}} \cdot  \left(\prod_{k \in \Omega}\sigma_{p_1-1, v_k}\sigma_{p_1-1, v_{k-1}} + \prod_{k \in \Omega}\sigma_{p_1, v_k}\sigma_{p_1, v_{k-1}}\right)\\
	\leq & \prod_{k\in \Omega^c} \sigma_{u_k, v_k}\sigma_{u_k, v_{k-1}} \cdot  \prod_{k \in \Omega}\left(\sigma_{p_1-1, v_k}\sigma_{p_1-1, v_{k-1}} +\sigma_{p_1, v_k}\sigma_{p_1, v_{k-1}}\right)\\
	\overset{\text{Cauchy-Schwarz}}{\leq} & \prod_{k \in \Omega^c} \tilde{\sigma}_{u_k, v_k}\tilde{\sigma}_{u_k, v_{k-1}} \cdot \prod_{k \in \Omega} \left((\sigma_{p_1-1,v_k}^2+\sigma_{p_1,v_k}^2)\cdot (\sigma_{p_1-1,v_{k-1}}^2+\sigma_{p_1,v_{k-1}}^2)\right)^{1/2} \\
	= &  \prod_{k \in \Omega^c}\tilde{\sigma}_{u_k, v_k}\tilde{\sigma}_{u_k, v_{k-1}} \cdot \prod_{k \in \Omega} \tilde{\sigma}_{p_1-1, v_k}\tilde{\sigma}_{p_1-1, v_{k-1}} \\
	= & \sum_{u_\Omega = p_1-1} \prod_{k=1}^q \tilde{\sigma}_{u_k, v_k}\tilde{\sigma}_{u_{k}, v_{k-1}} = \sum_{u_\Omega = p_1-1} \prod_{k=1}^q \tilde{\sigma}_{u_k, v_k}\tilde{\sigma}_{u_{k+1}, v_{k}}.
	\end{split}
	\end{equation}
	
	\item[Step 3] For any fixed $\Omega = \{k: u_k\in [p_1-2]\}$ and a cycle $\mathbf c = (u_1 \to v_1 \to u_2 \to v_2 \to \ldots \to u_q \to v_q \to u_1)$ such that $u_{\Omega} \in \{p_1-1,p\}$ and $u_\Omega \in [p_1-2]$, recall $\tilde{u}_k$ is defined as \eqref{eq:tilde-u-k}. We aim to show in this step that
	\begin{equation}\label{ineq:ZZtop-tildeZZtop-2}
	\mathbb{E}\prod_{k=1}^q \left((G_0)_{u_k, v_k} (G_0)_{u_{k+1}, v_k} - 1_{\{u_k = u_{k+1}\}}\right) \leq \mathbb{E}\prod_{k=1}^q \left((\tilde{G}_0)_{\tilde{u}_k, v_k} (\tilde{G}_0)_{\tilde{u}_{k+1}, v_k} - 1_{\{\tilde{u}_k = \tilde{u}_{k+1}\}}\right).
	\end{equation}
	We can rearrange the left hand side and the right hand side of \eqref{ineq:ZZtop-tildeZZtop-2} to
	\begin{equation*}
	\begin{split}
	\mathbb{E}\prod_{i=1}^{p_1} \prod_{j=1}^{p_2} (G_0)_{ij}^{\alpha_{ij}} ((G_0)_{ij}^2-1)^{\beta_{ij}}, \quad \text{and} \quad \mathbb{E}\prod_{i=1}^{p_1-1} \prod_{j=1}^{p_2} (\tilde{G}_0)_{ij}^{\tilde{\alpha}_{ij}} ((\tilde{G}_0)_{ij}^2-1)^{\tilde{\beta}_{ij}}.
	\end{split}
	\end{equation*}
	Here, $\alpha_{ij}, \beta_{ij}, \tilde{\alpha}_{ij},$ and $\tilde{\beta}_{ij}$ are defined as
	\begin{equation*}
	\begin{split}
	& \alpha_{ij} = \left| \left\{k: (u_k = i, v_k = j, u_{k+1} \neq i) \text{ or } (u_k \neq i, v_k = j, u_{k+1} = i)\right\}\right|,\\
	& \beta_{ij} = \left| \left\{k: u_k = u_{k+1} = i, v_k = j\right\}\right|,\\
	& \tilde{\alpha}_{ij} = \left| \left\{k: (\tilde{u}_k = i, v_k = j, \tilde{u}_{k+1} \neq i) \text{ or } (\tilde{u}_k \neq i, v_k = j, \tilde{u}_{k+1} = i)\right\}\right|\\
	& \tilde{\beta}_{ij} = \left| \left\{k: \tilde{u}_k = \tilde{u}_{k+1} = i, v_k = j\right\}\right|.
	\end{split}
	\end{equation*}
	Then, $\alpha_{ij}$ (or $\tilde{\alpha}_{ij}$) is the number of times that the edge $(i, j)$ is visited exactly once by sub-path $u_k\to v_k \to u_{k+1}$ (or $\tilde{u}_k\to v_k \to \tilde{u}_{k+1}$); $\beta_{ij}$ (or $\tilde{\beta}_{ij}$) is the number of times that the edge $(i, j)$ is visited twice (back and forth) by sub-path $u_k\to v_k\to u_{k+1}$ (or $\tilde{u}_k\to v_k\to \tilde{u}_{k+1}$). 
	
	Here, by comparing the order of $(\tilde{G}_0)_{ij}$ and $(G_0)_{ij}$ in these two monomials \eqref{ineq:ZZtop-tildeZZtop-2}, $\tilde{\alpha}_{ij}, \tilde{\beta}_{ij}, \alpha_{ij}, \beta_{ij}$ are related as
	\begin{equation*}
	\begin{split}
	& \tilde{\alpha}_{ij} = \alpha_{ij}, \quad \tilde{\beta}_{ij} = \beta_{ij}, \quad \text{if } 1\leq i \leq p_1-2, 1\leq j \leq n,\\
	\end{split}
	\end{equation*}
	The relationship among $\tilde{\alpha}_{p_1-1,j}, \tilde{\beta}_{p_1-1,j}, \alpha_{p_1-1,j}, \alpha_{p_1, j}, \beta_{p_1-1,j}, \beta_{p_1, j}$ is more involved. To analyze them, for any fixed $1\leq j \leq p_2$ we define
	\begin{equation*}
	\begin{split}
	& x_1^{(j)} = \left|\left\{k: (u_k \to v_k \to u_{k+1}) = ((p_1-1)\to j\to \{p_1-1, p_1\}^c) \text{ or } (\{p_1-1, p_1\}^c\to j\to (p_1-1)) \right\}\right|,\\
	& x_2^{(j)} = \left|\left\{k: (u_k \to v_k \to u_{k+1}) = (p_1\to j\to \{p_1-1, p_1\}^c) \text{ or } (\{p_1-1, p_1\}^c\to j\to p_1) \right\}\right|,\\
	& x_3^{(j)} = \left|\left\{k: (u_k \to v_k \to u_{k+1}) = ((p_1-1)\to j\to (p_1-1)) \right\}\right|,\\
	& x_4^{(j)} = \left|\left\{k: (u_k \to v_k \to u_{k+1}) = (p_1\to j\to p_1) \right\}\right|,\\
	& x_5^{(j)} = \left|\left\{k: (u_k \to v_k \to u_{k+1}) = ((p_1-1)\to j\to p_1) \text{ or } (p_1\to j\to (p_1-1)) \right\}\right|.\\
	\end{split}
	\end{equation*}
	Then by definitions, we have
	\begin{equation*}
	\begin{split}
	& \alpha_{p_1-1, j} = x_1^{(j)} + x_5^{(j)}, \quad \alpha_{p_1, j} = x_2^{(j)}+x_5^{(j)}, \quad \tilde{\alpha}_{p_1-1,j} = x_1^{(j)}+x_2^{(j)};\\
	& \beta_{p_1-1, j} = x_3^{(j)}, \quad \beta_{p_1, j} = x_4^{(j)},\quad \tilde{\beta}_{p_1-1,j} = x_3^{(j)}+x_4^{(j)}+x_5^{(j)}.
	\end{split}
	\end{equation*}
	We introduce the following Lemma before we proceed.
	\begin{lemma}\label{lm:Gaussian-moment-hetero-PCA}
	Suppose $Z_1, Z_2$ are independent and symmetric distributed random variables. $\Var(Z_1) = \Var(Z_2) = 1$, $\|Z_1\|_{\psi_2}, \|Z_2\|_{\psi_2} \leq \kappa$. $G$ is standard Gaussian distributed. For any non-negative integers $x_1, \ldots, x_5$, we have
	\begin{equation}\label{ineq:Gaussian-moment-hetero-PCA}
		\begin{split}
			&\left|\mathbb{E}Z_1^{x_1+x_5}Z_2^{x_2+x_5}(Z_1^2-1)^{x_3} (Z_2^2 -1)^{x_4}\right| \\
			&\qquad\qquad\qquad\leq (C\kappa)^{x_1+x_2+2(x_3+x_4+x_5)}\mathbb{E} G^{x_1+x_2}(G^2 - 1)^{x_3+x_4+x_5}.
		\end{split}
	\end{equation}
	Especially when $Z_1, Z_2, G$ are all standard Gaussian,
	\begin{equation}\label{ineq:Gaussian-moment-hetero-PCA-gaussian}
	\left|\mathbb{E}Z_1^{x_1+x_5}Z_2^{x_2+x_5}(Z_1^2-1)^{x_3} (Z_2^2 -1)^{x_4}\right| \leq \mathbb{E} G^{x_1+x_2}(G^2 - 1)^{x_3+x_4+x_5}.
	\end{equation}	
	\end{lemma}
	\begin{proof}
		See \hyperref[sec:proof-technical-lemma]{Appendix}.
	\end{proof}
	By Lemma \ref{lm:Gaussian-moment-hetero-PCA}, 
	\begin{equation*}
	\begin{split}
	& \left|\mathbb{E}(G_0)_{p_1-1,j}^{\alpha_{p_1-1,j}}((G_0)_{p_1-1,j}^2 - 1)^{\beta_{p_1-1,j}}\right| \cdot \left|\mathbb{E}(G_0)_{p_1,j}^{\alpha_{p_1,j}}((G_0)_{p_1,j}^2 - 1)^{\beta_{p_1,j}}\right| \\
	= & \left|\mathbb{E}(G_0)_{p_1-1,j}^{x_1^{(j)}+x_5^{(j)}}((G_0)_{p_1-1,j}^2-1)^{x_3^{(j)}}\right| \cdot \left|\mathbb{E}(G_0)_{p_1,j}^{x_2^{(j)}+x_5^{(j)}}((G_0)_{p_1,j}^2-1)^{x_4^{(j)}}\right|\\
	\leq & \mathbb{E}(\tilde{G}_0)_{p_1-1,j}^{x_1^{(j)}+x_2^{(j)}}((\tilde{G}_0)^2_{p_1-1,j}-1)^{x_3^{(j)}+x_4^{(j)}+x_5^{(j)}} = \mathbb{E}(\tilde{G}_0)_{p_1-1,j}^{\tilde{\alpha}_{p_1-1,j}}((\tilde{G}_0)^2_{p_1-1,j}-1)^{\tilde{\beta}_{p_1-1,j}}.
	\end{split}
	\end{equation*}
	Thus,
	\begin{equation}\label{ineq:ZZtop-tildeZZtop-4}
	\mathbb{E}\prod_{i=1}^{p_1} \prod_{j=1}^{p_2} (G_0)_{ij}^{\alpha_{ij}} ((G_0)_{ij}^2-1)^{\beta_{ij}} \leq \mathbb{E}\prod_{i=1}^{p_1-1} \prod_{j=1}^{p_2} (\tilde{G}_0)_{ij}^{\tilde{\alpha}_{ij}} ((\tilde{G}_0)_{ij}^2-1)^{\tilde{\beta}_{ij}}. 
	\end{equation}
	This gives \eqref{ineq:ZZtop-tildeZZtop-2}. 
	
	\item[Step 4] Combining \eqref{eq:ZZtop-tildeZZtop-3} and \eqref{ineq:ZZtop-tildeZZtop-2}, we finally have
	\begin{equation*}
	\begin{split}
	& \sum_{u_{\Omega} \in \{p_1-1, p_1\}} \mathbb{E}\prod_{k=1}^q \sigma_{u_k, v_k}\sigma_{u_{k+1}, v_k}\left((G_0)_{u_k, v_k} (G_0)_{u_{k+1}, v_k} - 1_{\{u_k = u_{k+1}\}}\right) \\
	= & \sum_{u_{\Omega} \in \{p_1-1, p_1\}} \prod_{k=1}^q \sigma_{u_k, v_k} \sigma_{u_{k+1}, v_k} \cdot \mathbb{E}\prod_{k=1}^q \left((G_0)_{u_k, v_k} (G_0)_{u_{k+1}, v_k} - 1_{\{u_k = u_{k+1}\}}\right)\\
	\overset{\eqref{ineq:ZZtop-tildeZZtop-2}}{\leq} & \sum_{u_{\Omega} \in \{p_1-1, p_1\}} \prod_{k=1}^q  \sigma_{u_k, v_k} \sigma_{u_{k+1}, v_k} \cdot  \mathbb{E}\prod_{k=1}^q \left((\tilde{G}_0)_{\tilde{u}_k, v_k} (\tilde{G}_0)_{\tilde{u}_{k+1}, v_k} - 1_{\{\tilde{u}_k = \tilde{u}_{k+1}\}}\right)\\
	\overset{\eqref{eq:ZZtop-tildeZZtop-3}}{\leq} & \mathbb{E}\prod_{k=1}^q \tilde{\sigma}_{\tilde{u}_k, v_k} \tilde{\sigma}_{\tilde{u}_{k+1}, v_k}\left((\tilde{G}_0)_{\tilde{u}_k, v_k} (\tilde{G}_0)_{\tilde{u}_{k+1}, v_k} - 1_{\{\tilde{u}_k = \tilde{u}_{k+1}\}}\right),
	\end{split}
	\end{equation*}
	which yields \eqref{ineq:ZZtop-tildeZZtop-1} and additionally finishes the proof of this lemma.  \quad $\square$
\end{itemize}
\end{proof}

\begin{proof}[Proof of Theorem \ref{th:row-wise-concentration-lower}]
Denote $\sigma_C^2 = \sum_i \sigma_i^2, \sigma_{\ast}=\max_i\sigma_i$, $Z = [Z_1, \ldots, Z_{p_2}]$, and $S_k = Z_kZ_k^\top - \mathbb{E}Z_kZ_k^\top$. Then
\begin{equation*}
\mathbb{E}\left\|ZZ^\top - \mathbb{E}ZZ^\top \right\| = \mathbb{E}\left\|\sum_{k=1}^{p_2} S_k\right\|.
\end{equation*}
By the lower bound for expected norm of independent random matrices sum \cite[Theorem I and Section 1.3]{tropp2016expected},
\begin{equation}\label{ineq:Tropp-lower-bound}
\mathbb{E}\|ZZ^\top - \mathbb{E}ZZ^\top\| \gtrsim \left(\left\|\mathbb{E}\sum_{k=1}^{p_2} S_kS_k^\top\right\|\right)^{1/2} +  \mathbb{E}\max_k\|S_k\|.
\end{equation}
If $Z_{ij}\sim N(0,\sigma_i^2)$ for any $i\in [p_1],  j \in [p_2]$. Note that 
\begin{equation*}
\begin{split}
& \left(\mathbb{E}Z_kZ_k^\top Z_k Z_k^\top\right)_{ij} = \mathbb{E} Z_{ik}\sum_{l=1}^{p_1} Z_{lk}^2 Z_{jk} = \left\{\begin{array}{ll}
3\sigma_i^4 + \sigma_i^2\left(\sum_{l\neq i} \sigma_l^2\right), & i=j;\\
0, & i\neq j,
\end{array}\right. = \diag\left(\{2\sigma_i^4+\sigma_i^2\sigma_C^2\}_{i=1}^{p_1}\right) \\
& \left(\mathbb{E}Z_kZ_k^\top\right)^2 = \diag(\sigma_1^4,\ldots, \sigma_{p_1}^4).
\end{split}
\end{equation*}
Thus,
\begin{equation*}
\begin{split}
\left\|\mathbb{E}\sum_{k=1}^{p_2} S_kS_k^\top\right\| = & \left\|\sum_{k=1}^{p_2}\mathbb{E}(Z_kZ_k^\top - \mathbb{E}Z_kZ_k^\top)(Z_kZ_k^\top - \mathbb{E}Z_kZ_k^\top)\right\| \\
= & \left\|\sum_{k=1}^{p_2}\mathbb{E}Z_kZ_k^\top Z_kZ_k^\top - (\mathbb{E}Z_kZ_k^\top)^2\right\|\\
= &  \left\|\diag\left(\{\sigma_i^2\sigma_C^2 + \sigma_i^4\}_{i=1}^{p_1}\right)\right\| = \sigma_\ast^4 + \sigma_{\ast}^2 \sigma_C^2.
\end{split}
\end{equation*}
Meanwhile, let $i^\ast \in [p]$ such that suppose $\sigma_{\ast} = \sigma_{i^\ast}$, then
$$\mathbb{E}\|S_k\| = \mathbb{E}\left\|Z_kZ_k^\top - \mathbb{E}Z_kZ_k^\top\right\| \geq \mathbb{E}\left\|Z_kZ_k^\top\right\| - \left\|\mathbb{E}Z_kZ_k^\top\right\| = \sigma_C^2 - \sigma_{\ast}^2;$$
$$\mathbb{E}\|S_k\| \geq \mathbb{E}\|(S_k)_{i^\ast i^\ast}\| = \mathbb{E}\left|Z_{ i^\ast k}^2 - \mathbb{E}Z_{i^\ast k}^2\right|  \geq c\sigma_{\ast}^2.$$
Combining the previous two inequalities, we have $\mathbb{E}\|S_k\| \geq c\sigma_C^2$. Consequently,
\begin{equation*}
	\mathbb{E}\left\|ZZ^\top - \mathbb{E}ZZ^\top \right\| \overset{\eqref{ineq:Tropp-lower-bound}}{\gtrsim} \sigma_C^2 + \sqrt{p_2}\sigma_{\ast}\sigma_C. \quad \square
\end{equation*}
\end{proof}
\subsection{Proofs for Section \ref{sec:homoskedastic columns}}\label{sec:proof-homo-column}

\begin{proof}[Proof of Lemma \ref{lm:diagonal-deletion-comparison}] Since the diagonal of $\Delta(ZZ^\top)$ is zero, we have the following expansion,
\begin{equation}\label{eq:home-row-expansion}
\begin{split}
& \bbE\tr\left\{\left(\Delta(ZZ^\top)\right)^q\right\} = \sum_{u_1,\ldots,u_1 \in [p_1]} \bbE \prod_{k=1}^q\left(\Delta(ZZ^\top)\right)_{u_k,u_{k+1}} \\
= &  \sum_{\substack{u_1,\ldots,u_1 \in [p_1] \\ v_1 ,\ldots v_q \in [p_2]}} \bbE\prod_{k=1}^q \left(1_{\{u_k \neq u_{k+1}\}}Z_{u_k,v_k}Z_{u_{k+1},v_k}\right).
\end{split}
\end{equation}
Again, the indices on $u$ are in module $q$, i.e., $u_1 = u_{q+1}$. For a cycle $\mathbf c := (u_1 \to v_1 \to u_2 \to v_2 \to \ldots \to u_q \to v_q \to u_1)$, recall the definition of $\alpha_{ij}(\mathbf c)$:
\begin{equation*}
\alpha_{ij}(\mathbf c) = \text{Card} \left\{k: (u_k = i, v_k = j, u_{k+1} \neq i) \text{ or } (u_k \neq i, v_k = j, u_{k+1} = i)\right\}
\end{equation*}
for any $i \in [p_1]$ and $j \in [p_2]$, which counts how many times edge $i \to j$ or $j \to i$ are visited. Now the expansion in \eqref{eq:home-row-expansion} can be further written as
\begin{equation}\label{eq:home-row-expansion2}
\begin{split}
& \bbE\tr\left\{\left(\Delta(ZZ^\top)\right)^q\right\}  = \sum_{\mathbf c \in ([p_1]\times [p_2])^q} \left(\prod_{k=1}^q 1_{\{u_k \neq u_{k+1}\}}\right) \cdot \left(\prod_{(i,j) \in [p_1]\times [p_2]}\bbE Z_{ij}^{{\alpha_{ij}(\mathbf c)}}\right)  \\
= & \sum_{\mathbf c \in ([p_1]\times [p_2])^q} \left(\prod_{k=1}^q 1_{\{u_k \neq u_{k+1}\}}\sigma_{u_k,v_k}\sigma_{u_{k+1},v_k}\right) \cdot \left(\prod_{(i,j) \in [p_1]\times [p_2]}\bbE G^{{\alpha_{ij}(\mathbf c)}}\right)\\
= & \sum_{\mathbf c \in ([p_1]\times [p_2])^q} \left(\prod_{k=1}^q 1_{\{u_k \neq u_{k+1}\}}\sigma_{v_k}^2\right) \cdot \left(\prod_{(i,j) \in [p_1]\times [p_2]}\bbE G^{{\alpha_{ij}(\mathbf c)}}\right).
\end{split}
\end{equation}
We define $m_\alpha(\mathbf c)$ be the number of edges which appear $\alpha$ times in the cycle $\mathbf c$:
\begin{equation*}
m_{\alpha}(\mathbf{c}) = \text{Card}\left\{(i, j)\in [p_1]\times[p_2]: |\{k: u_k \text{ or } u_{k+1} = i, v_k = j,\}| = \alpha \right\}
\end{equation*} 
Let $\mathbf s(\mathbf c)$ be the shape of $\mathbf c$, we have
\begin{equation*}
\prod_{(i,j) \in [p_1]\times [p_2]}\bbE G^{{\alpha_{ij}(\mathbf c)}} = \prod_{\alpha \geq 0} \bbE G^{m_\alpha(\mathbf s(\mathbf c))},
\end{equation*}
where $G\sim N(0,1)$. Next we define the following shape family:
\begin{equation*}
\mathcal S_{p_1,p_2} := \left\{\mathbf s(\mathbf c): m_{\alpha}'(\mathbf c) = 0 \text{ for all odd } \alpha; \text{ and } u_k \neq u_{k+1} \text{ for all }k=1,\ldots,q\right\}.
\end{equation*}
Based on the notations above, one can check the expansion in \eqref{eq:home-row-expansion2} can be further simplified to
\begin{equation}\label{eq:home-row-expansion3}
\begin{split}
\bbE\tr\left\{\left(\Delta(ZZ^\top)\right)^q\right\} & = \sum_{\mathbf s_0 \in \mathcal S_{p_1,p_2}}\sum_{\mathbf c: \mathbf s(\mathbf c) = \mathbf s_0} \left(\prod_{k=1}^q \sigma_{v_k}^2\right) \prod_{\alpha \geq 0}\bbE G^{m_\alpha(\mathbf s_0)} \\
& = \sum_{\mathbf s_0 \in \mathcal S_{p_1,p_2}}\prod_{\alpha \geq 0}\bbE G^{m_\alpha(\mathbf s_0)}\sum_{\mathbf c: \mathbf s(\mathbf c) = \mathbf s_0} \left(\prod_{k=1}^q \sigma_{v_k}^2\right).
\end{split}
\end{equation}
For a fixed shape $\mathbf s_0 \in \mathcal S_{p_1,p_2}$, let $m_L(\mathbf s_0)$ ($m_R(\mathbf s_0)$) be the number of distinct left (right) vertexes visited by cycles with shape $\mathbf s_0$. Now we bound $\sum_{\mathbf c: \mathbf s(\mathbf c) = \mathbf s_0} \left(\prod_{k=1}^q \sigma_{v_k}^2\right)$ via $m_L(\mathbf s_0)$ and $m_R(\mathbf s_0)$. To this end, we first present three facts for any cycles with shape $\mathbf s_0$:
\begin{itemize}
	\item Each visited edges must appear at least twice in the cycles; 
	\item For each right vertex in the cycle, its predecessor and successor in left vertex set must be different;
	\item The cycle is uniquely defined by specifying $m_L(\mathbf s_0)$ left vertexes and $m_R(\mathbf s_0)$ right vertexes; moreover, the summation term is free of the index of the left visited vertexes.
\end{itemize}
These three observations, together with the assumption $\sigma_* = 1$, yield the following bound:
\begin{equation}\label{ineq:home-row-trace-bound1}
\begin{split}
\sum_{\mathbf c: \mathbf s(\mathbf c) = \mathbf s_0} \left(\prod_{k=1}^q \sigma_{v_k}^2\right) \leq p_1(p_1-1)\cdots (p_1-m_L(\mathbf s_0)+1) \left(\sum_{j=1}^n \sigma_j^4\right)^{m_R(\mathbf s_0)}.
\end{split}
\end{equation}
Next we make comparison between $\bbE\tr\left\{\left(\Delta(ZZ^\top)\right)^q\right\}$ and $\bbE\tr\left\{\left(\Delta(HH^\top)\right)^q\right\}$, where $H$ is a $p_1$-by-$m$ random matrix with i.i.d. standard Gaussian entries. Similar, as above, we have
\begin{equation*}
\begin{split}
\bbE\tr\left\{\left(\Delta(HH^\top)\right)^q\right\} & = \sum_{\mathbf s_0 \in \mathcal S_{p_1,p_2}}\prod_{\alpha \geq 0}\bbE G^{m_\alpha(\mathbf s_0)}\sum_{\mathbf c: \mathbf s(\mathbf c) = \mathbf s_0} \left|\left\{\mathbf c: \mathbf s(\mathbf c) = \mathbf s_0\right\}\right|.
\end{split}
\end{equation*}
Setting $m = \lceil \sum_{j=1}^{p_2} \sigma_j^4\rceil + q -1 $, we have
\begin{equation}\label{ineq:home-row-trace-bound2}
\begin{split}
\left|\left\{\mathbf c: \mathbf s(\mathbf c) = \mathbf s_0\right\}\right| & = p_1(p_1-1)\cdots(p_1-m_L(\mathbf s_0)+1)m(m-1)\cdots(m-m_R(\mathbf s_0)+1) \\
& \geq p_1(p_1-1)\cdots(p_1-m_L(\mathbf s_0)+1)(m-m_R(\mathbf s_0)+1)^{m_R(\mathbf s_0)}\\
& \geq  p_1(p_1-1)\cdots (p_1-m_L(\mathbf s_0)+1) \left(\sum_{j=1}^n \sigma_j^4\right)^{m_R(\mathbf s_0)}.
\end{split} 
\end{equation}
Combining \eqref{ineq:home-row-trace-bound1} and \eqref{ineq:home-row-trace-bound2}, we finish the proof.
\end{proof}

\begin{proof}[Proof of Theorem \ref{th:column-wise-concentration-lower}]
Denote $\sigma_R^2 = \sum_j \sigma_j^2, \sigma_{\ast}=\max_i\sigma_i$. We use the general lower bound for expected norm of independent random matrices sum \cite[Theorem I and Section 1.3]{tropp2016expected} as we did in the proof of Theorem \ref{th:row-wise-concentration-lower}. Since $Z_{ij} \sim N(0,\sigma_j^2)$, for any $k \in [p_2]$,
\begin{equation*}
\begin{split}
& \left(\mathbb{E}Z_kZ_k^\top Z_k Z_k^\top\right)_{ij} = \mathbb{E} Z_{ik}\sum_{l=1}^{p_1} Z_{lk}^2 Z_{jk} = \left\{\begin{array}{ll}
3\sigma_k^4 + (p_1-1)\sigma_k^4, & i=j;\\
0, & i\neq j,
\end{array}\right. = \diag\left(\{(p_1+2)\sigma_k^2\}_{i=1}^{p_1}\right) \\
& \left(\mathbb{E}Z_kZ_k^\top\right)^2 = \sigma_k^4I_{p_1}.
\end{split}
\end{equation*}
Thus,
\begin{equation*}
\begin{split}
\left\|\mathbb{E}\sum_{k=1}^{p_2} S_kS_k^\top\right\| = &  \left\|\sum_{k=1}^{p_2}\mathbb{E}Z_kZ_k^\top Z_kZ_k^\top - (\mathbb{E}Z_kZ_k^\top)^2\right\|=  \left\|(p_1+1)\left(\sum_{k=1}^{p_2}\sigma_k^4\right) I_{p_1}\right\| \geq p_1\sum_{k=1}^{p_2}\sigma_k^4.
\end{split}
\end{equation*}
On the other hand, 
$$\mathbb{E}\max_k\|S_k\| \geq \max_k\mathbb{E}\left\|S_k\right\| \geq \max_k\left\{\bbE\left\|Z_kZ_k^\top\right\| - \left\|\mathbb{E}Z_kZ_k^\top\right\|\right\} = (p_1-1)\sigma_*^2.$$
Combining the previous two inequalities and \eqref{ineq:Tropp-lower-bound} in the proof of Theorem \ref{th:row-wise-concentration-lower}, we obtain
\begin{equation*}
\mathbb{E}\left\|ZZ^\top - \mathbb{E}ZZ^\top \right\| \overset{\eqref{ineq:Tropp-lower-bound}}{\gtrsim} \sqrt{p_1\sum_{k=1}^{p_2} \sigma_k^2} + p_1\sigma_*^2. 
\end{equation*}	
\end{proof}

\subsection{Proofs for heteroskedastic clustering}\label{sec:proof-hetero-clustering}

\begin{proof}[Proof of Theorem \ref{th:hetero-clustering}]
We first introduce following three lemmas. 
\begin{lemma}\label{lm:hamming-loss}
For any $x \in \{-1, +1\}^n$ and $z \in \mathbb R$ with $\|z\|_2 = 1$ we have
\begin{equation*}
	d(x,sgn(z)) \leq n \left\|\frac{x}{\sqrt{n}}- z\right\|_2^2.
\end{equation*}
Here $d$ represents the Hamming distance: $d(x,z) = \sum_{i=1}^n 1_{\{x_i \neq y_i\}}$.
\end{lemma}
\begin{proof}
See \cite{lelarge2015reconstruction}. \qquad $\square$
\end{proof}

\begin{lemma}\label{lm:subgaussian-linear-bound}
	Assume that $Z \in \bbR^{p_1 \times p_2}$ has independent sub-Gaussian entries, $\Var(Z_{ij})=\sigma_{ij}^2$, $\sigma_C^2 = \max_j \sum_i \sigma_{ij}^2$, $\sigma_R^2 = \max_i \sum_j \sigma_{ij}^2$, $\sigma_*^2 =\max_{i,j} \sigma_{ij}^2$.
	Assume that $\left\|Z_{ij}/\sigma_{ij}\right\|_{\psi_2} \leq \kappa.$ Let $V \in \mathbb O_{p_2,r}$ be a fixed orthogonal matrix. Then,
	\begin{equation*}
	\bbP\left(\left\|EV\right\| \geq 2(\sigma_C + x)\right) \leq 2\exp\left(5r-\min\left\{\frac{x^4}{\kappa^4 \sigma_*^2 \sigma_C^2},\frac{x^2}{\kappa^2\sigma_*^2}\right\}\right),
	\end{equation*}
	\begin{equation*}
	\bbE\left\|EV\right\| \lesssim \sigma_C + \kappa r^{1/4}(\sigma_*\sigma_C)^{1/2} + \kappa r^{1/2}\sigma_*.
	\end{equation*}
\end{lemma}
\begin{proof}
See \cite[Lemma 3]{zhang2018heteroskedastic}.
\end{proof}

\begin{lemma}[Davis-Kahan]\label{lm:subspace-pertur}
	Let $A$ be an $n$-by-$n$ symmetric matrix with eigenvalues $|\lambda_1| \geq |\lambda_2| \geq \cdots $, with $|\lambda_k| - |\lambda_{k+1}| \geq 2\delta$. Let B be a symmetric matrix such that $\left\|B\right\| < \delta$. Let $A_k$ and $(A + B)_k$ be the spaces spanned by the top $k$ eigenvectors of the respective matrices. Then 
	\begin{equation*}
	\left\|I_k - A_k^\top (A+B)_k\right\| \leq \frac{\|B\|}{\delta}.
	\end{equation*}
\end{lemma}
\begin{proof}
See \cite{davis1970rotation}.
\end{proof} 

Now we are ready for the proof.
Recall that $Y = X + Z$, we can write  
\begin{equation}
	\begin{split}
		Y Y^\top & = X X^\top +X Z^\top +  X^\top + ZZ^\top \\
			& = X X^\top + X Z^\top + Z X^\top + \left(Z Z^\top - \bbE Z Z^\top\right) + \bbE Z Z^\top.
	\end{split}
\end{equation}
Since $\bbE Z Z^\top = \left(\sum_{j=1}^p \sigma_j^2\right) I$, the leading eigenvector of $Y^\top Y$ (i.e., $\hat v$) is the same as that of 
\begin{equation*}
	X X^\top + X Z^\top + Z X^\top + \left(Z Z^\top - \bbE Z Z^\top\right).
\end{equation*}
Since $\frac{1}{\sqrt{n}}l$ is the leading eigenvector of $X^\top X$, it follows that
\begin{equation*}
	\begin{split}
		\bbE\mathcal M(l,\hat l) \overset{\text{Lemma \ref{lm:hamming-loss}}}{\leq} & \bbE \min_{\pm}\left\|\frac{1}{\sqrt{n}} l \pm \hat l\right\|_2^2  \overset{\text{Lemma \ref{lm:subspace-pertur}}}{\leq} \frac{\bbE\left\|X Z^\top + Z X^\top + Z Z^\top - \bbE Z Z^\top\right\|}{n \left\|\mu\right\|_2^2} \\
		& \leq \frac{2\bbE\left\|Z X^\top\right\| + \bbE\left\|Z^\top Z - \bbE Z^\top Z\right\|}{n\left\|\mu\right\|_2^2}\\
		& \overset{\text{Lemma \ref{lm:subgaussian-linear-bound}}}{\lesssim} \frac{n\left\|\mu\right\|_2 \sigma_* + \bbE\left\|Z^\top Z - \bbE Z^\top Z\right\|}{n \left\|\mu\right\|_2^2} \\
		& \overset{\text{Theorem \ref{th:row-wise-concentration}}}{\lesssim} \frac{n\left\|\mu\right\| \sigma_* + \sqrt{n\sum_{i=1}^p \sigma_i^4} + n\sigma_*^2}{n\left\|\mu\right\|_2^2}.
	\end{split}
\end{equation*}
\end{proof}

\begin{proof}[Proof of Theorem \ref{th:hetero-clustering-lower-bound}]
We only need to prove the lower bound under the following two situations:
\begin{itemize}
\item when $\lambda \leq c_1 \sigma_*$, there exists $\{\sigma_i\}_{i=1}^p$ such that $\max_i \sigma_i \leq \sigma_*$, $\sum_i \sigma_i^4 \leq \tilde{\sigma}^4$ and the lower bound holds;
\item when $\lambda \leq c_2 \tilde{\sigma}/n^{1/4}$, there exists $\{\sigma_i\}_{i=1}^p$ such that $\max_i \sigma_i \leq \sigma_*$, $\sum_i \sigma_i^4 \leq \tilde{\sigma}^4$ and the lower bound holds.
\end{itemize}
We start with the first case. We specify $\sigma_1 = \sigma_*$ and take $\sigma_2,\ldots,\sigma_p$ to be arbitrary values that satisfy the constraint of $\mathcal P_{\lambda,l}(\sigma_*,\tilde \sigma)$. Consider the metric space $\{-1,1\}^n$ with the metric 
\begin{equation*}
	\mathcal M(l^{(1)}, l^{(2)}) = \frac{1}{n}\min\left\{\left|i: l^{(1)}_i \neq  l^{(2)}_i\right|, \left|i: l_i^{(i)} \neq - l_i^{(2)}\right|\right\},
\end{equation*} 
By \cite[Lemma 4]{yu1997assouad}, when $n\geq 6$, we can find some constant $c_0$, such that there exists a subset $\{l^{(1)},\ldots,l^{(N)}\} \subset \{-1,1\}^n$ satisfying
\begin{equation*}
	\mathcal{M} (l^{(i_1)}, l^{(i_2)}) \geq 1/3,\qquad \forall 1 \leq i_1 < i_2 \leq N
\end{equation*}
and $N \geq \exp(c_0 n)$. Let $Y^{(i)} = \mu \left(l^{(i)}\right)^\top + Z \in \bbR^{p \times n}$, where $Z_{ij} \overset{ind}{\sim} N(0,\sigma_i^2)$. Let $\mu = [\lambda, 0, \cdots, 0]^\top$, then the KL-divergence between $Y^{(i_i)}$ and $Y^{(i_2)}$ for $i_1 \neq i_2$ is
\begin{equation}
	\begin{split}
 		& D_{KL}(Y^{(i_1)}|Y^{(i_2)}) = \frac{1}{2}\sum_{j=1}^p \sigma_j^{-2}\mu_j^2 \left\|l^{(i_1)} - l^{(i_2)}\right\|_2^2 \leq 4n\Sigma_{j=1}^p \sigma_j^{-2}\mu_j^2  = 4n\sigma_1^{-2}\lambda^2 = 4n\lambda^2/\sigma_*^2.
	\end{split}
\end{equation}
By the generalized Fano's lemma, we have
\begin{equation*}
	\inf_{\hat l} \sup_{\mathcal P_{l,\lambda}(\sigma_*,\tilde{\sigma})} \bbE\mathcal M(l,\hat l) \geq \frac{1}{3}\left(1-\frac{4n\lambda^2/\sigma_*^2 + \log 2}{c_0n}\right) \geq \frac{1}{4}.
\end{equation*}
In the last inequality we use the assumption that $\lambda \leq c_1 \sigma_*^2$ for some sufficiently small constant $c_1$.

Now we consider the second situation. We specify $\sigma_1^4 = \sigma_2^4 = \ldots = \sigma_p^4 = \frac{\tilde{\sigma}^4}{p}$. When the variance structure reduces to a homoskedastic structure, we have the following lower bound result which is already established. 
\begin{lemma}\label{lm:homo-clustering-lower}
Suppose $\sigma_1^2 = \cdots = \sigma_p^2 = 1$, there exists $c_{2}$, $C$ such that if $n \geq C$,  
\begin{equation*}
	\inf_{\hat l} \sup_{\substack{\|\mu\|_2 \leq c_2(p/n)^{1/4}\\ l \in \{-1,1\}^n}} \bbE \mathcal M(\hat l, l) \geq 1/4.
\end{equation*}
\end{lemma}
\begin{proof}
See \cite[Theorem 6]{cai2018rate}.
\end{proof}
Based on Lemma \ref{lm:homo-clustering-lower} and homoskedasticity of $\mu$ and $\sigma$, if we set $\lambda < \frac{c_2\tilde{\sigma}}{p^{1/4}}\cdot \left(\frac{p}{n}\right)^{1/4} = c_2\tilde{\sigma}/n^{1/4}$ in our setting, we obtain
\begin{equation*}
\inf_{\hat l} \sup_{\mathcal P_{l,\lambda}(\sigma_*,\tilde{\sigma})} \bbE\mathcal M(l,\hat l)  \geq 1/4.
\end{equation*}
This finishes the proof.
\end{proof}

\bibliographystyle{amsplain}
\bibliography{reference}



\newpage
\begin{appendix}
\section{Proofs of technical Lemmas}\label{sec:proof-technical-lemma}

We collect the proofs of Lemma \ref{lm:Gaussian-moments}, \ref{lm:heavy-tail-comparison}, and \ref{lm:Gaussian-moment-hetero-PCA} in this section.
\begin{proof}[Proof of Lemma \ref{lm:Gaussian-moments}]
We first consider the proof of \eqref{ineq:E-G(G^2-1)}. Note that if $G\sim N(0, 1)$, 
$$\mathbb{E}G^{d} = \left\{
\begin{array}{ll}
(d-1)!!, & d\geq 0, \text{ and $d$ is even};\\
0, & d \geq 0, \text{ and $d$ is odd}.
\end{array}\right.$$
In addition, $(-1)!! = 1, (-3)!! = -1$. When $\alpha$ is odd, only odd moments of $G$ appear in the expansion of $G^\alpha(G^2-1)^\beta$, then clearly $\mathbb{E} G^\alpha(G^2 - 1)^\beta =0$. When $\alpha$ is even and $\alpha + 2\beta \geq 4$, 
\begin{equation*}
\begin{split}
\mathbb{E} G^\alpha(G^2 - 1)^\beta = & \sum_{j=0}^\beta \bbE G^{\alpha+2\beta - 2j}(-1)^{j} \binom{\beta}{j} = \sum_{j=0}^\beta (-1)^j (\alpha + 2\beta-2j-1)!! \cdot \frac{\beta!}{(\beta - j)!j!}\\
\geq & \sum_{\substack{0\leq j \leq \beta\\j \text{ is even}}} \left\{\frac{(\alpha+2\beta-2j-1)!!\beta!}{(\beta-j)!j!} - \frac{(\alpha+2\beta - 2(j+1)-1)!!\beta!}{(\beta-j-1)!(j+1)!} \right\}\\
= & \sum_{\substack{0\leq j \leq \beta\\j \text{ is even}}} \frac{(\alpha+2\beta - 2j-3)!!\beta!}{(\beta-j)!(j+1)!}\cdot \left\{(\alpha+2\beta-2j-1)(j+1) - (\beta-j) \right\}
\end{split}
\end{equation*}
\begin{itemize}
	\item If $j = \beta$,
	\begin{equation*}
	\begin{split}
	& \frac{(\alpha+2\beta - 2j-3)!!\beta!}{(\beta-j)!(j+1)!}\cdot \left\{(\alpha+2\beta-2j-1)(j+1) - (\beta-j) \right\} \\
	= & \frac{(\alpha+2\beta - 2j-3)!!\beta!}{(\beta-j)!(j+1)!}\cdot(\alpha+2\beta-2j-1)(j+1) \geq 0;
	\end{split}
	\end{equation*}
	\item If $\beta - 1 \geq j\geq \frac{\beta-1}{2}$, 
	$$(\alpha+2\beta - 2j-1)(j+1) \geq (\alpha+2\beta - 2(\beta-1)-1)\left(\frac{\beta-1}{2}+1\right) \geq  \frac{\beta+1}{2} \geq \beta-j;$$ 
	\item if $0\leq j < \frac{\beta-1}{2}$, $$(\alpha+2\beta - 2j-1)(j+1) \geq \alpha + 2\beta - (\beta-1)-1 \geq \beta - j.$$ 
\end{itemize}
Thus, we always have 
$$\frac{(\alpha+2\beta - 2j-3)!!\beta!}{(\beta-j)!(j+1)!}\cdot \left\{(\alpha+2\beta-2j-1)(j+1) - (\beta-j)\right\} \geq 0, \quad \forall 0\leq j\leq \beta, j \text{ is even},$$ 
and
\begin{equation*}
\begin{split}
\mathbb{E} G^\alpha(G^2 - 1)^\beta \geq & \sum_{j=0}^0 \frac{(\alpha+2\beta - 2j-3)!!\beta!}{(\beta-j)!(j+1)!}\cdot \left\{(\alpha+2\beta-2j-1)(j+1) - (\beta-j) \right\}\\
= & (\alpha + 2\beta - 3)!! \cdot (\alpha+\beta-1),
\end{split}
\end{equation*}
which has finished the proof of \eqref{ineq:E-G(G^2-1)}.

Next we consider the upper bound of $\mathbb{E}G^\alpha(G^2-1)^\beta$. 
\begin{equation*}
\begin{split}
\mathbb{E} G^\alpha(G^2 - 1)^\beta = & \sum_{j=0}^\beta \bbE G^{\alpha+2\beta - 2j}(-1)^{j} \binom{\beta}{j} = \sum_{j=0}^\beta (-1)^j (\alpha + 2\beta-2j-1)!! \cdot \frac{\beta!}{(\beta - j)!j!}\\
\leq & (\alpha+2\beta - 1)!! - \sum_{\substack{0\leq j \leq \beta\\j \text{ is odd}}} \left\{\frac{(\alpha+2\beta-2j-1)!!\beta!}{(\beta-j)!j!} - \frac{(\alpha+2\beta - 2(j+1)-1)!!\beta!}{(\beta-j-1)!(j+1)!} \right\}\\
= & (\alpha+2\beta-1)!! -  \sum_{\substack{0\leq j \leq \beta\\j \text{ is odd}}} \frac{(\alpha+2\beta - 2j-3)!!\beta!}{(\beta-j)!(j+1)!}\cdot \left\{(\alpha+2\beta-2j-1)(j+1) - (\beta-j) \right\}
\end{split}
\end{equation*}
Similarly as the previous argument, we can show for any odd $1\leq j \leq \beta$,
\begin{equation*}
\frac{(\alpha+2\beta - 2j-3)!!\beta!}{(\beta-j)!(j+1)!}\cdot \left\{(\alpha+2\beta-2j-1)(j+1) - (\beta-j) \right\} \geq 0,
\end{equation*}
thus,
\begin{equation*}
\mathbb{E}G^\alpha(G^2-1)^\beta \leq (\alpha+2\beta-1)!!.
\end{equation*}

Then we consider the proof of sub-Gaussian case \eqref{ineq:E-G(G^2-1)sub-Gaussian}. When $\alpha$ is odd, the statement clearly holds as $Z^\alpha (Z^2-1)^\beta$ has symmetric distribution then $\mathbb{E}Z^\alpha(Z^2 - 1)^\beta = 0$. When $\alpha$ is even, since $Z^2\geq 0$, we must have $|Z^2-1| = \max\{Z^2-1, 1-Z^2\} \leq Z^2 \vee 1$, thus
\begin{equation*}
\begin{split}
\left|\mathbb{E}Z^\alpha (Z^2-1)^\beta\right| \leq & \left|\mathbb{E}Z^\alpha(Z^2-1)^\beta 1_{\{|Z|\leq 1\}} + \mathbb{E}Z^\alpha(Z^2-1)^\beta 1_{\{|Z|> 1\}}\right|\\
\leq & 1 + \mathbb{E}Z^\alpha |Z^2|^\beta = \mathbb{E}|Z|^{\alpha+2\beta} + 1.
\end{split}
\end{equation*}
Since $\bbE Z^2 = 1$, we have $\kappa \geq 1/\sqrt{2}$. Thus, 
$$
\mathbb{E}|Z|^{\alpha+2\beta} + 1 \leq (\kappa+1)^{\alpha+2\beta}(\alpha+2\beta)^{(\alpha+2\beta)/2} \leq (3\kappa)^{\alpha + 2\beta}(\alpha+2\beta)^{(\alpha+2\beta)/2}.
$$
It is easy to see \eqref{ineq:E-G(G^2-1)sub-Gaussian} holds when $\alpha + 2\beta \leq 2$. 

When $\alpha+2\beta \geq 4$, by the relationship between double factorial and Gamma function\footnote{See \url{https://en.wikipedia.org/wiki/Double_factorial}} and the lower bound of Gamma function \cite{batir2017bounds}, we have
\begin{equation*}
\begin{split}
& (\alpha+2\beta-3)!!(\alpha+\beta-1) = \frac{2^{\alpha/2+\beta-1}}{\sqrt{\pi}}\Gamma\left(\frac{\alpha}{2} + \beta-1 + \frac{1}{2}\right) (\alpha+\beta-1)\\
\geq & \frac{2^{\alpha/2+\beta-1}(\alpha+\beta-1)}{\sqrt{\pi}}\cdot \sqrt{2\pi} x^x e^{-x} \left(x^2 + x/3 + 0.04\right)^{1/4}\\
\geq & C^{-(\alpha+2\beta)} \cdot (\alpha+2\beta)^{(\alpha+2\beta)/2} = (C\kappa)^{\alpha+2\beta}(\alpha+2\beta)^{(\alpha+2\beta)/2} \geq \mathbb{E}Z^{\alpha}(Z^2-1)^\beta.
\end{split}
\end{equation*}
Here, $x = \frac{\alpha+2\beta-3}{2}$.
\end{proof}

\begin{proof}[Proof of Lemma \ref{lm:heavy-tail-comparison}]
Firstly we have
\begin{equation*}
	\begin{split}
		& \bbE F_{ij}^\alpha (F_{ij}^2 - 1)^\beta = \sum_{j=0}^\beta \bbE F_{ij}^{\alpha+2\beta-2j} (-1)^j \binom{\beta}{j}\\
		& \qquad = \frac{1}{\sqrt{\pi}}\sum_{j=0}^\beta (-1)^j \frac{\beta !}{(\beta - j)! j!}(x_j-1)!! \cdot 2^\frac{(b-1)x_j}{2}\Gamma\left(\frac{(b-1)x_j+1}{2}\right) \\
		& \qquad \geq \frac{1}{\sqrt{\pi}}\sum_{\substack{j\leq \beta \\ j\text{ is even}}}\left\{\frac{\beta !}{(\beta-j)!j!}\left(x_j-1\right)!! \cdot 2^{\frac{(b-1)x_j}{2}}\Gamma\left(\frac{(b-1)x_j+1}{2}\right)\right. \\
		& \qquad \qquad\qquad\qquad\qquad  - \left. \frac{\beta!}{(\beta-j-1)!(j+1)!}(x_j-3)!!2^{\frac{(b-1)(x_j-2)}{2}}\Gamma\left(\frac{(b-1)(x_j-2)+1}{2}\right)\right\}\\
		& \qquad \geq \frac{1}{\sqrt{\pi}}\sum_{\substack{j\leq \beta \\ j\text{ is even}}} \frac{\beta!}{(\beta-j)!(j+1)!}(x_j-3)!!2^\frac{(b-1)(x_j-2)}{2}\Gamma\left(\frac{(b-1)(x_j-2)+1}{2}\right) \\
		& \qquad \qquad\qquad\qquad  \cdot \left[(j+1)(x_j-1) - (\beta-j)\right],
	\end{split}
\end{equation*}
where $x_j:=\alpha+2\beta-2j$ and the last inequality comes from the strictly increasing property of Gamma function. By the proof of Lemma \ref{lm:Gaussian-moments}, we know
\begin{equation*}
	\frac{\beta!}{(\beta-j)!(j+1)!}(x_j-3)!!\cdot\left((j+1)(x_j-1)-(\beta-j)\right) \geq 0.
\end{equation*}
Thus, 
\begin{equation*}
	\begin{split}
	\bbE F_{ij}^\alpha (F_{ij}^2 - 1)^\beta &\geq \frac{\alpha+\beta-1}{\sqrt{\pi}}(\alpha+2\beta-3)!!\cdot 2^{\frac{(b-1)(\alpha+2\beta-2)}{2}}\Gamma\left(\frac{(b-1)(\alpha+2\beta-2)+1}{2}\right) \\
	& = \frac{1}{\pi}2^{\frac{b(\alpha+2\beta-2)}{2}}\Gamma\left(\frac{\alpha+2\beta-1}{2}\right)\Gamma\left(\frac{(b-1)(\alpha+2\beta-2)+1}{2}\right).
	\end{split}
\end{equation*}
When $\alpha+2\beta \geq \left(2 + \frac{3}{b-1}\right) \vee 5$, by the lower bound of Gamma function \cite{batir2017bounds}, we further have 
\begin{equation*}
	\begin{split}
		\bbE F_{ij}^\alpha (F_{ij}^2 - 1)^\beta &\geq 2\cdot 2^{\frac{b(\alpha+2\beta-2)}{2}}x^xe^{-x}\left(x^2+\frac{x}{3}+0.04\right)^{1/4}y^ye^{-y}\left(y^2+\frac{y}{3}+0.04\right)^{1/4} \\
		& \geq (c_b)^{\alpha+2\beta} \cdot (\alpha+2\beta)^{(\alpha+2\beta)/2} \cdot \left((b-1)(\alpha+2\beta)\right)^{(b-1)(\alpha+2\beta)/2} \\
		& \geq (c_b')^{\alpha+2\beta} \cdot (\alpha+2\beta)^{b(\alpha+2\beta)/2},
	\end{split}
\end{equation*}
where $x = \frac{\alpha+2\beta-3}{2}, y = \frac{(b-1)(\alpha+2\beta-2)-1}{2}$ and $c_b>0$ is some constant that only depends on $b$. 

When $2\leq \alpha+2\beta < \left(2 + \frac{3}{b-1}\right) \vee 5$, we can find another universal constant $c_b''$ such that
\begin{equation*}
	\frac{1}{\pi}2^{\frac{b(\alpha+2\beta-2)}{2}}\Gamma\left(\frac{\alpha+2\beta-1}{2}\right)\Gamma\left(\frac{(b-1)(\alpha+2\beta-2)+1}{2}\right) \geq (c_b'')^{\alpha+2\beta} \cdot (\alpha+2\beta)^{b(\alpha+2\beta)/2}.
\end{equation*}
In conclusion, we proved that
\begin{equation*}
	\bbE F_{ij}^\alpha (F_{ij}^2 - 1)^\beta \geq (C_b\kappa)^{\alpha+2\beta}\cdot\bbE E_{ij}^\alpha (E_{ij}^2 - 1)^\beta 
\end{equation*}
for any $\alpha,\beta \geq 0$. Thus \eqref{ineq:E-G(G^2-1)heavy-tail} is proved. 
\end{proof}

\begin{proof}[Proof of Lemma \ref{lm:Gaussian-moment-hetero-PCA}]
If either $(x_1, x_3, x_5) = (0,0,0)$ or $(x_2, x_4, x_5) = (0,0,0)$, the statement \eqref{ineq:Gaussian-moment-hetero-PCA} immediately follows from the proof of Lemma \ref{lm:Gaussian-moments} and the statement of \eqref{ineq:Gaussian-moment-hetero-PCA-gaussian} becomes identity; 
if either $x_1+x_5$ or $x_2+x_5$ is odd, the left hand side of \eqref{ineq:Gaussian-moment-hetero-PCA} \eqref{ineq:Gaussian-moment-hetero-PCA-gaussian} are zero since $Z_1$ and $Z_2$ are symmetric distributed and independent. Meanwhile, the right hand side of \eqref{ineq:Gaussian-moment-hetero-PCA} is non-negative (Lemma \ref{lm:Gaussian-moments}), thus \eqref{ineq:Gaussian-moment-hetero-PCA} holds if either $x_1+x_5$ or $x_2+x_5$ is odd. When $(x_1, x_3, x_5) = (0,1,0)$ (or $(x_2, x_4, x_5) = (0,1,0)$), by similar arguments one can show \eqref{ineq:Gaussian-moment-hetero-PCA} holds. 

Thus we only need to prove the inequality when both $x_1+x_5$ and $x_2+x_5$ are even, and 
$$x_1+x_5+x_3 \geq 2,\quad \text{and}\quad  x_2+x_5+x_4 \geq 2.$$ 
By Lemma \ref{lm:Gaussian-moments}, we have
\begin{equation*}
\begin{split}
& \left|\mathbb{E}Z_1^{x_1+x_5}Z_2^{x_2+x_5}(Z_1^2-1)^{x_3} (Z_2^2 -1)^{x_4}\right|\\
= &  \left|\left(\mathbb{E}Z_1^{x_1+x_5}(Z_1^2-1)^{x_3}\right)\cdot \left(\mathbb{E}Z_2^{x_2+x_5}(Z_2^2 -1)^{x_4}\right)\right|\\
\leq & \left(C\kappa\right)^{x_1+x_2 + 2(x_3+x_4+x_5)} \cdot \left|\left(\mathbb{E}G^{x_1+x_5}(G^2-1)^{x_3}\right) \cdot \left(\mathbb{E}G^{x_2+x_5}(G^2-1)^{x_4}\right)\right|\\
\leq & (C\kappa)^{x_1+x_2+2(x_3+x_4+x_5)} \cdot (x_1+x_5+2x_3-1)!! \cdot (x_2+x_5+2x_4-1)!!.\\
\end{split}
\end{equation*}
Since for any odd positive integers $x, y$, $x!! \cdot y!! \leq (x+y-1)!!$, we have
\begin{equation*}
(x_1+x_5+2x_3-1)!! \cdot (x_2+x_5+2x_4-1)!! \leq (x_1+x_2+2(x_3+x_4+x_5)-3)!!
\end{equation*}
Therefore,
\begin{equation*}
\begin{split}
& \left|\mathbb{E}Z_1^{x_1+x_5}Z_2^{x_2+x_5}(Z_1^2-1)^{x_3} (Z_2^2 -1)^{x_4} \right|\\
\leq & (C\kappa)^{x_1+x_2+2(x_3+x_4+x_5)} \cdot \left(x_1+x_2 + 2(x_3+x_4+x_5)-3\right)!!\cdot (x_1+x_2+x_3+x_4+x_5)\\
\overset{\text{Lemma \ref{lm:Gaussian-moments}}}{\leq} & (C\kappa)^{x_1+x_2+2(x_3+x_4+x_5)}\cdot \mathbb{E} (G^2 - 1)^{x_3+x_4+x_5}G^{x_1+x_2},
\end{split}
\end{equation*}
which has finished the proof of \eqref{ineq:Gaussian-moment-hetero-PCA}.

If $Z_1, Z_2, G$ are standard Gaussian, by Lemma \ref{lm:Gaussian-moments}, 
\begin{equation*}
\begin{split}
& \left|\mathbb{E}Z_1^{x_1}Z_2^{x_2}(Z_1^2-1)^{x_3} (Z_2^2 -1)^{x_4} (Z_1Z_2)^{x_5}\right|\\
= &  \left|\left(\mathbb{E}Z_1^{x_1+x_5}(Z_1^2-1)^{x_3}\right)\right|\cdot \left|\left(\mathbb{E}Z_2^{x_2+x_5}(Z_2^2 -1)^{x_4}\right)\right|\\
\leq & (x_1+x_5+2x_3-1)!! \cdot (x_2+x_5+2x_4-1)!!\\
\leq & \left(x_1+x_2 + 2(x_3+x_4+x_5)-3\right)!!\cdot (x_1+x_2+x_3+x_4+x_5)\\
\overset{\text{Lemma \ref{lm:Gaussian-moments}}}{\leq} & \mathbb{E}G^{x_1+x_2} (G^2 - 1)^{x_3+x_4+x_5},\\
\end{split}
\end{equation*}
which has finished the proof of \eqref{ineq:Gaussian-moment-hetero-PCA-gaussian}.
\end{proof}
\end{appendix}

\end{document}